\documentclass[reqno, 11pt]{amsart}
\usepackage{graphicx, color, amsmath, amsfonts, amssymb, amsthm, amscd}
\usepackage[colorlinks=true, pdfstartview=FitV, linkcolor=green,
  citecolor=green, urlcolor=green,pagebackref=false]{hyperref}
\usepackage[numeric,initials,nobysame]{amsrefs}
\usepackage[]{bbm,enumerate,paralist}
\numberwithin{equation}{section}

\DeclareMathSymbol{\leqslant}{\mathalpha}{AMSa}{"36} 
\DeclareMathSymbol{\geqslant}{\mathalpha}{AMSa}{"3E} 
\DeclareMathSymbol{\eset}{\mathalpha}{AMSb}{"3F}     
\renewcommand{\leq}{\;\leqslant\;}                   
\renewcommand{\geq}{\;\geqslant\;}                   
\newcommand{\Var} {{\ensuremath{\mathbb{V}\textup{ar}}}}

\newcommand{\maxtwo}[2]{\max_{\substack{#1 \\ #2}}} 

\makeatletter
\def\captionfont@{\footnotesize}
\def\captionheadfont@{\scshape}
\long\def\@makecaption#1#2{%
  \vspace{2mm}
  \setbox\@tempboxa\vbox{\color@setgroup
    \advance\hsize-6pc\noindent
    \captionfont@\captionheadfont@#1\@xp\@ifnotempty\@xp
        {\@cdr#2\@nil}{.\captionfont@\upshape\enspace#2}%
    \unskip\kern-6pc\par
    \global\setbox\@ne\lastbox\color@endgroup}%
  \ifhbox\@ne 
    \setbox\@ne\hbox{\unhbox\@ne\unskip\unskip\unpenalty\unkern}%
  \fi
  \ifdim\wd\@tempboxa=\z@ 
    \setbox\@ne\hbox to\columnwidth{\hss\kern-6pc\box\@ne\hss}%
  \else 
    \setbox\@ne\vbox{\unvbox\@tempboxa\parskip\z@skip
        \noindent\unhbox\@ne\advance\hsize-6pc\par}%
\fi
  \ifnum\@tempcnta<64 
    \addvspace\abovecaptionskip
    \moveright 3pc\box\@ne
  \else 
    \moveright 3pc\box\@ne
\nobreak
\vskip\belowcaptionskip
\fi
\relax
}
\makeatother

\def\writefig#1 #2 #3 {\rlap{\kern #1 truecm
\raise #2 truecm \hbox{#3}}}

\newtheorem{thm}{Theorem}[section]
\newtheorem{lem}[thm]{Lemma}
\newtheorem{prop}[thm]{Proposition}
\newtheorem{cor}[thm]{Corollary}
\newtheorem{rem}[thm]{Remark}

\newcommand{\cB}{\ensuremath{\mathcal B}}

\newcommand{\cF}{\ensuremath{\mathcal F}}
\newcommand{\cG}{\ensuremath{\mathcal G}}
\newcommand{\cH}{\ensuremath{\mathcal H}}

\newcommand{\cS}{\ensuremath{\mathcal S}}

\newcommand{\frB}{\ensuremath{\mathfrak B}}

\newcommand{\frp}{\ensuremath{\mathfrak p}}

\newcommand{\bbE}{{\ensuremath{\mathbb E}} }

\newcommand{\bbN}{{\ensuremath{\mathbb N}} }

\newcommand{\bbP}{{\ensuremath{\mathbb P}} }

\newcommand{\bbR}{{\ensuremath{\mathbb R}} }
\newcommand{\bbS}{{\ensuremath{\mathbb S}} }

\newcommand{\bbZ}{{\ensuremath{\mathbb Z}} }

\newcommand{\bS} {{\bf S}}

\newcommand{\bQ} {{\bf Q}}
\newcommand{\bW} {{\bf W}}

\newcommand{\rmD}{{\ensuremath{\mathrm D}} }

\newcommand{\rmT}{{\ensuremath{\mathrm T}} }

\newcommand{\ga}{\alpha}
\newcommand{\gb}{\beta}
\newcommand{\gga}{\gamma}            
\newcommand{\gd}{\delta}
\newcommand{\gvep}{\varepsilon}       
\newcommand{\gep}{\epsilon}           

\newcommand{\gz}{\zeta}

\newcommand{\gO}{\Omega}
\newcommand{\gl}{\lambda}

\newcommand{\gs}{\sigma}
\newcommand{\gS}{\Sigma}
\newcommand{\gt}{\theta}

\newcommand{\sfe}{{\sf e}}

\newcommand{\normI}[1] {\| #1 \|_{{\scriptscriptstyle 1}}}
\newcommand{\normII}[1]{\| #1 \|_{{\scriptscriptstyle 2}}}
\newcommand{\normIs}[1] {\left \| #1 \right \|_{{\scriptscriptstyle 1}}}

\newcommand{\normsup}[1]{\|#1\|_{{\scriptscriptstyle\infty}}}
\newcommand{\normTV}[1]{\|#1\|_{{\scriptscriptstyle \textrm{TV}}}}

\def\1{\ifmmode {1\hskip -3pt \rm{I}}
\else {\hbox {$1\hskip -3pt \rm{I}$}}\fi} 

\newcommand{\lb}{\left(}
\newcommand{\rb}{\right)}
\newcommand{\lbr}{\left\{}
\newcommand{\rbr}{\right\}}
\newcommand{\lan}{\left\langle}
\newcommand{\ran}{\right\rangle}
\newcommand{\labs}{\left |}
\newcommand{\rabs}{\right |}
\newcommand{\lsb}{\left [}
\newcommand{\rsb}{\right ]}
\newcommand{\lpr}{\left.}
\newcommand{\rpr}{\right.}

\newtheorem{bigthm}{Theorem}   

\newcommand{\wt}{\widetilde}

\newcommand{\grad}{\nabla}           

\newcommand{\wh}{\widehat}
\newcommand{\ol}{\overline}

\newcommand{\inv}[1]{\tfrac{1}{#1}}

\newcommand{\Escx}[1]{\bbE^{\textrm{SC}, {#1}}}

\newcommand{\Pst}{\bbP^{\star}}

\newcommand{\bydef}{\triangleq}

\newcommand{\minwith}{\land}
\newcommand{\maxwith}{\lor}

\newcommand{\symdiff}{\varDelta}

\newcommand{\forcenewline}{$ $ \\}

\newcommand{\vq}{\wh{\sfe}}
\newcommand{\whS}{\wh{S}}

\newcommand{\whs}{\wh{s}}
\newcommand{\mix}{\textsc{mix}}
\newcommand{\one}{\mathbbm{1}}

\begin{document}

\title[]{Glauber dynamics for the mean-field Potts model}

\author{P.\ Cuff}
\address{Paul Cuff\hfill\break
Department of Electrical Engineering\\
Princeton University\\
Princeton, NJ 08544.}
\email{cuff@princeton.edu}
\urladdr{}

\author{J.\ Ding}
\address{Jian Ding\hfill\break
Department of Mathematics\\
Stanford University\\
Stanford, CA 94305.}
\email{jianding@math.stanford.edu}
\urladdr{}

\author{O.\ Louidor}
\address{Oren Louidor\hfill\break
Department of Mathematics\\
UCLA\\
Los Angeles, CA 90095.}
\email{louidor@math.ucla.edu}
\urladdr{}

\author{E.\ Lubetzky}
\address{Eyal Lubetzky\hfill\break
Microsoft Research\\
One Microsoft Way\\
Redmond, WA 98052-6399, USA.}
\email{eyal@microsoft.com}
\urladdr{}

\author{Y.\ Peres}
\address{Yuval Peres\hfill\break
Microsoft Research\\
One Microsoft Way\\
Redmond, WA 98052-6399, USA.}
\email{peres@microsoft.com}
\urladdr{}

\author{A.\ Sly}
\address{Allan Sly\hfill\break
Department of Statistics\\
UC Berkeley\\
Berkeley, CA 94720, USA.}
\email{sly@stat.berkeley.edu}
\urladdr{}

\setcounter{page}{1}

\begin{abstract}
We study Glauber dynamics for the mean-field (Curie-Weiss) Potts model with $q\geq 3$ states and show that it undergoes a critical slowdown at an inverse-temperature $\gb_s(q)$ strictly lower than the critical $\beta_c(q)$ for uniqueness of the thermodynamic limit. The dynamical critical $\gb_s(q)$ is the spinodal point marking the onset of metastability.

We prove that when $\beta<\gb_s(q)$ the mixing time is asymptotically $C(\gb, q) n \log n$ and the dynamics exhibits the cutoff phenomena, a sharp transition in mixing, with a window of order $n$.  At $\beta=\gb_s(q)$ the dynamics no longer exhibits cutoff and its mixing obeys a power-law of order $n^{4/3}$.  For $\beta>\gb_s(q)$ the mixing time is exponentially large in $n$.
Furthermore, as $\beta\uparrow \beta_s$ with $n$, the mixing time interpolates smoothly from subcritical to critical behavior, with the latter reached at a scaling window of $O(n^{-2/3})$ around $\beta_s$.
These results form the first complete analysis of mixing around the critical dynamical temperature --- including the critical power law --- for a model with a first order phase transition.
\end{abstract}
\maketitle
\vspace{-0.5cm}

\section{Introduction and Results}
\label{sec:Introduction}

We study the dynamics of the Potts model on the complete graph (mean-field) known as the {\it Curie-Weiss Potts} model.
For $n \geq 1$, $\gb \geq 0$, the Curie-Weiss Potts distribution is a probability measure on
$\gS_n = Q^V$ where $Q=\{1,\dots, q\}$ and $V=\{1,\dots, n\}$, defined by
\[
\mu_n(\gs) = Z_{\beta,n}^{-1} \exp \Big \{ (\gb/n) \sum_{u,v \in V} \one_{\gs(u) = \gs(v)} \Big \},
\]
where $\gs \in \gS_n$ and $Z_{\beta,n}$ is the normalizing constant. When $q=2$ this is the classic Ising model while in this paper we will focus on the case $q\geq 3$ for an integer $q$ (for an extension to non-integer $q$ via the random cluster model, see e.g.~\cite{Grimmett}). We use the standard notation
$\gb_c(q)$ for the (explicitly known) threshold value between the ordered and the disordered phases (see~\cite{CET}).

Throughout the paper $(\gs_t)_{t \geq 0}$ will denote the discrete time Glauber dynamics for this model, namely, starting from $\gs_0$, at each step we choose a vertex $u \in V$ uniformly and set
\[
    \gs_{t+1}(v) = \left\{
        \begin{array}{ll}
            \gs_t(v) & \textrm{if $v \neq u$} \\
            k & \textrm{with probability } \mu_n \lb \gs(u)=k \,\big|\, \gs(w)=\gs_t(w) \  \forall w \neq u \rb\mbox{ if $v=u$}.
        \end{array}
    \right.
\]
We denote by $P_n$ the transition kernel for this Markov process and by  $\bbP_{\gs_0}$ the underlying probability measure.
We will measure the distance between the distribution of the chain at time $t$ and its stationary distribution $\mu_n$ via the total-variation norm. Accordingly,
\[
d^{\gs_0}_t(n) = \normTV{\bbP_{\gs_0} \lb \gs_t \in \cdot \rb - \mu_n}
\quad \text{and} \quad
d_t(n) = \max_{\gs_0 \in \gS_n} d^{\gs_0}_t(n) .
\]
For $\gep \in (0,1)$, the $\gep$-mixing time is the number of steps until the total-variation distance to stationarity is at most $\gep$ in the worst case, i.e.:
\[
t_{\mix(\gep)}(n) = \inf \{ t : \: d_t(n) \leq \gep \}
\]
and by convention we set $t_{\mix}(n) := t_{\mix(1/4)}(n)$.
If for any fixed $\gep \in (0,1)$
\[
w_{\gep}(n) \bydef t_{\mix(\gep)}(n) - t_{\mix(1-\gep)}(n) = o(t_{\mix(1/4)}(n)) \quad \quad \text{as} \quad n \to \infty ,
\]
we say that the family of Markov chains exhibits the \emph{cutoff phenomenon}, which describes a sharp drop in the total variation distance from close to $1$ to close $0$ (in an interval of time of smaller order than $t_{\mix}(n)$ denoted as the \emph{cutoff window}).
Observe that cutoff occurs if and only if $t_{\mix(\delta)}(n) /  t_{\mix(\epsilon)}(n) \to 1$ as $n\to\infty$ for any fixed $\delta,\epsilon\in(0,1)$.

\subsection{Results}
\label{sub:Results}
We show that the dynamics for the Curie-Weiss model undergoes a critical slowdown
at an inverse-temperature $\gb_s(q) > 0$. This dynamical threshold is given by
\begin{equation}
\label{eqn:BDDef}
\gb_s(q) = \sup \lbr \gb \geq 0 : \:
	\lb 1 + (q-1) \mathrm{e}^{2\gb \frac{1 - qx}{q-1}} \rb^{-1} -x \neq 0
	\, \text{ for all } \, x \in (1/q, 1) \rbr.
\end{equation}
Unlike mean-field Ising, for which $\gb_s(2)=\gb_c(2)=1$, the dynamical transition for $q\geq 3$
occurs at a strictly higher temperature than the static phase transition, i.e., $\gb_s(q) < \gb_c(q)$.

\begin{figure}
\centering
\includegraphics[width=0.5\textwidth]{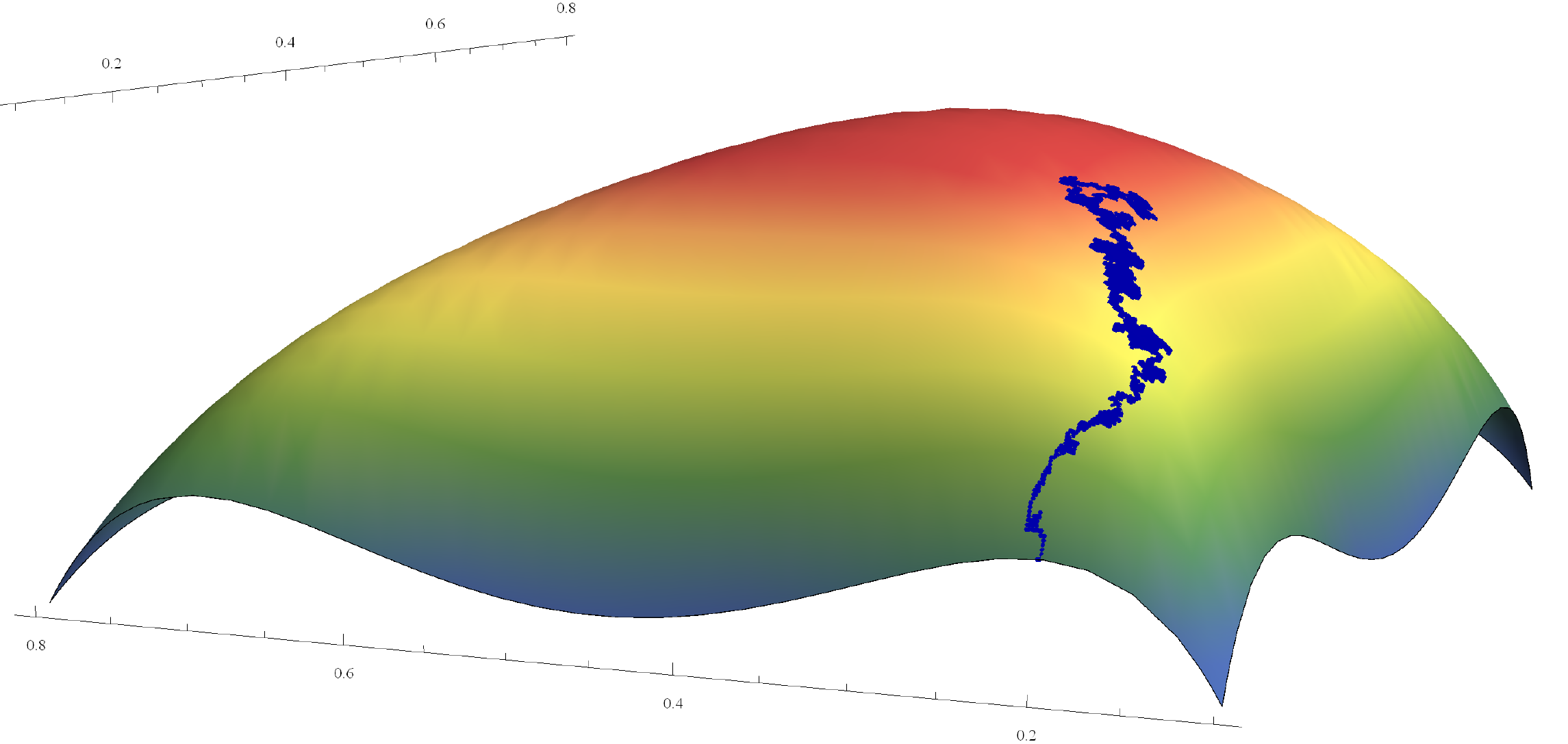}
\caption{Rapid mixing with cutoff in the subcritical regime of $\beta<\gb_s(q)$ for $q=3$.
Blue curve marks the magnetization vector of the Glauber dynamics along time.
}
\label{f:subcritical}
\end{figure}

Our first result addresses the regime $\gb < \gb_s(q)$, where rapid mixing occurs within $O(n \log n)$ steps and the dynamics exhibits cutoff  with a window of size $O(n)$ (see Fig.~\ref{f:subcritical}).
\begin{bigthm}
\label{thm:SubCriticalCutoff}
Let $q\geq 3$ be an integer. If $\gb < \gb_s(q)$ then the Glauber dynamics for the $q$-state Curie-Weiss Potts model exhibits cutoff at mixing time
\begin{equation}
\label{eqn:SubscriticalMixingAsymptotics}
	t_{\mix}(n) = \ga_1(\gb, q) n \log n 	
\end{equation}
with cutoff window $w_{\gep}(n) = O_{\gep} (n)$ where $\ga_1(\gb, q) = [2 \lb 1 - 2\gb/q \rb]^{-1}$.
\end{bigthm}
We proceed to analyze the order of the mixing time as $\gb(n) \to \gb_s(q)$ as $n\to\infty$,
\begin{equation}
\label{eqn:BetaNDef}
\gb(n) = \gb_s(q) - \xi(n)
\end{equation}
where $\xi(n) \to 0$ as $n \to \infty$. The asymptotics of the mixing time will, of course, depend on how fast $\xi$ decays, but it turns out that cutoff is observed only iff the decay is slow enough.
This is captured in the following theorem.
\begin{bigthm}
\label{thm:NearCriticalMixing}
Let $q\geq 3$ be an integer. With $\beta(n)$ given as in equation~\eqref{eqn:BetaNDef} we have:
\begin{enumerate}
\item\label{item:ThmNearCritical1}
If $\lim_{n\to \infty} n^{2/3} \xi(n) = \infty$ then the Glauber dynamics has cutoff with mixing time and cutoff window given by
\begin{align}
\label{eqn:CriticalMixingThm_OutOfCriticalWindow}
t_{\mix}(n) &= \ga_1(\gb(n), q) n \log n + \ga_2(q) n/\sqrt{\xi(n)} \, ,\nonumber\\
			w_{\gep}(n) &= O_{\gep}\big( n + \sqrt{n/\xi(n)^{5/2}} ~\big) \, ,
\end{align}
where $\ga_2(q)$ is a positive constant and   $\ga_1$ is the constant defined in Theorem~\emph{\ref{thm:SubCriticalCutoff}}.
\item\label{item:ThmNearCritical2}
If $0 \leq \liminf_{n\to\infty} n^{2/3} \xi(n) \leq \limsup_{n\to\infty} n^{2/3} \xi(n) < \infty$
then the dynamics does not exhibit cutoff and has mixing time
\begin{equation}
\label{eqn:CriticalMixingThm_InCriticalWindow}
t_{\mix(\gep)}(n) = \Theta_{\gep} \big( n^{4/3} \big)\,.
\end{equation}
\end{enumerate}
\end{bigthm}
\noindent Part~\eqref{item:ThmNearCritical2} of Theorem~\ref{thm:NearCriticalMixing} in particular applies at criticality $\gb = \gb_s(q)$ where the mixing time is of order $n^{4/3}$ with a scaling window of order $n^{-2/3}$ (in contrast, the mixing time for the critical Ising model is of order $n^{3/2}$ with a window of $\sqrt{n}$).

Finally, above $\gb_s(q)$ the mixing time is exponentially large in $n$, as depicted in Fig.~\ref{f:supercritical}.
\begin{bigthm}
\label{thm:SuperCriticalSlowMixing}
Let $q\geq 3$ be an integer, and fix $\gb > \gb_s(q)$. For every $0<\gep < 1$ there exist $C_1, C_2 > 0$ such that for all $n$,
\[
t_{\mix(\gep)}(n) \geq C_1 \exp (C_2 n)\,.
\]
\end{bigthm}

\begin{figure}
\centering
\includegraphics[width=0.48\textwidth]{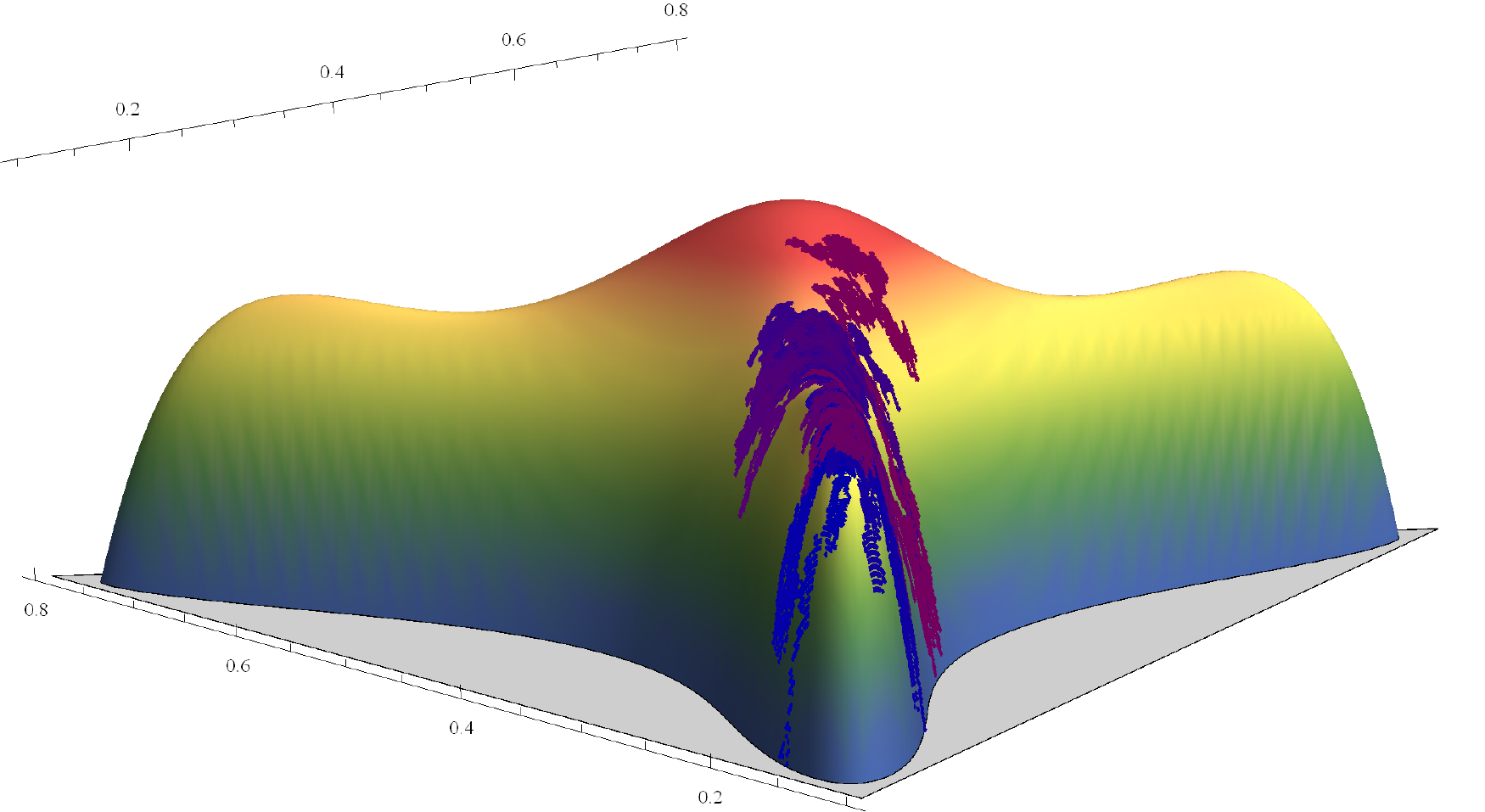}
\includegraphics[width=0.48\textwidth]{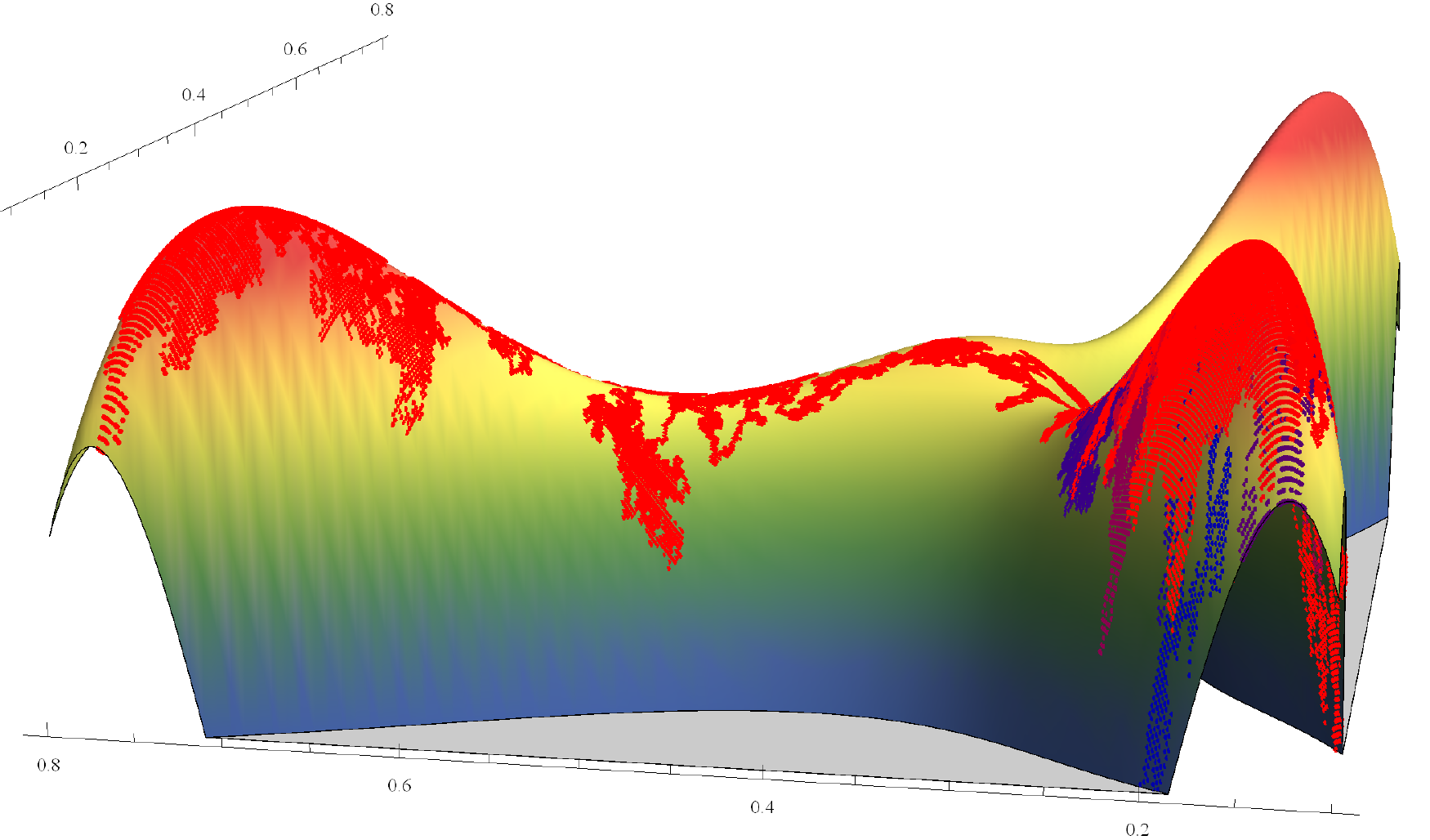}
\caption{Slow mixing without cutoff in the supercritical regime $\beta>\gb_s$ for $q=3$. On left $\gb_s<\beta<\beta_c$ and on right $\beta>\beta_c$. Curve color marks time from blue to red.}
\label{f:supercritical}
\end{figure}

Combined these results give a complete analysis of the mixing time of Glauber dynamics for the Curie-Weiss Potts model.

The slowdown in the mixing of the dynamics occurring as soon as $\gb \geq \gb_s(q)$ (be it power-law at $\gb_s(q)$ or exponential mixing above this point) is
due to the existence of states from which the Markov chain takes a long time to escape.
However, in the range $\gb \in [\gb_s(q), \gb_c(q))$ the subset of initial configurations from which mixing is slow is exponentially small in probability. One can then ask instead about the mixing time from typical starting locations, known as \emph{essential mixing}. Define the mixing time started from a subset of initial configurations $\wt{\gS}_n \subseteq \gS_n$ via
$d^{\wt{\gS}_n}_t(n) = \max_{\gs \in \wt{\gS}_n} d^{\gs}_t(n)$
as well as
\[
t^{\wt{\gS}_n}_{\mix(\gep)}(n) = \inf \{ t : \: d^{\wt{\gS}_n}_t(n) \leq \gep \}
\quad \text{and} \quad
w^{\wt{\gS}_n}_{\gep}(n) = t^{\wt{\gS}_n}_{\mix(\gep)}(n) - t^{\wt{\gS}_n}_{\mix(1-\gep)}(n).
\]
With these definitions we have the following result, showing that the subcritical mixing time behavior from Theorem~\ref{thm:SubCriticalCutoff} extends all the way to $\beta < \gb_s(q)$ once one excludes a subset of initial configurations with a total mass that is exponentially small in $n$.
\begin{bigthm}
\label{thm:EssentialMixing}
Let $q\geq 3$ be an integer and let $\gb < \gb_c(q)$.
There exist constants $C_1, C_2 > 0$ and subsets $\wt{\gS}_n \subseteq\gS_n$ such
that the Glauber dynamics has cutoff with mixing time and cutoff window given by
\[
	t^{\wt{\gS}_n}_{\mix}(n) = \ga_1(\gb, q) n \log n 	
		\quad ; \quad \quad
	w^{\wt{\gS}_n}_{\gep}(n) = O_{\gep} (n)\,,
\]
where $\mu_n (\gS_n \setminus \wt{\gS}_n) \leq C_1 e^{-C_2 n}$ and $\ga_1$ is the constant in Theorem~\emph{\ref{thm:SubCriticalCutoff}}.
\end{bigthm}

\begin{figure}
\centering
\includegraphics[width=0.5\textwidth]{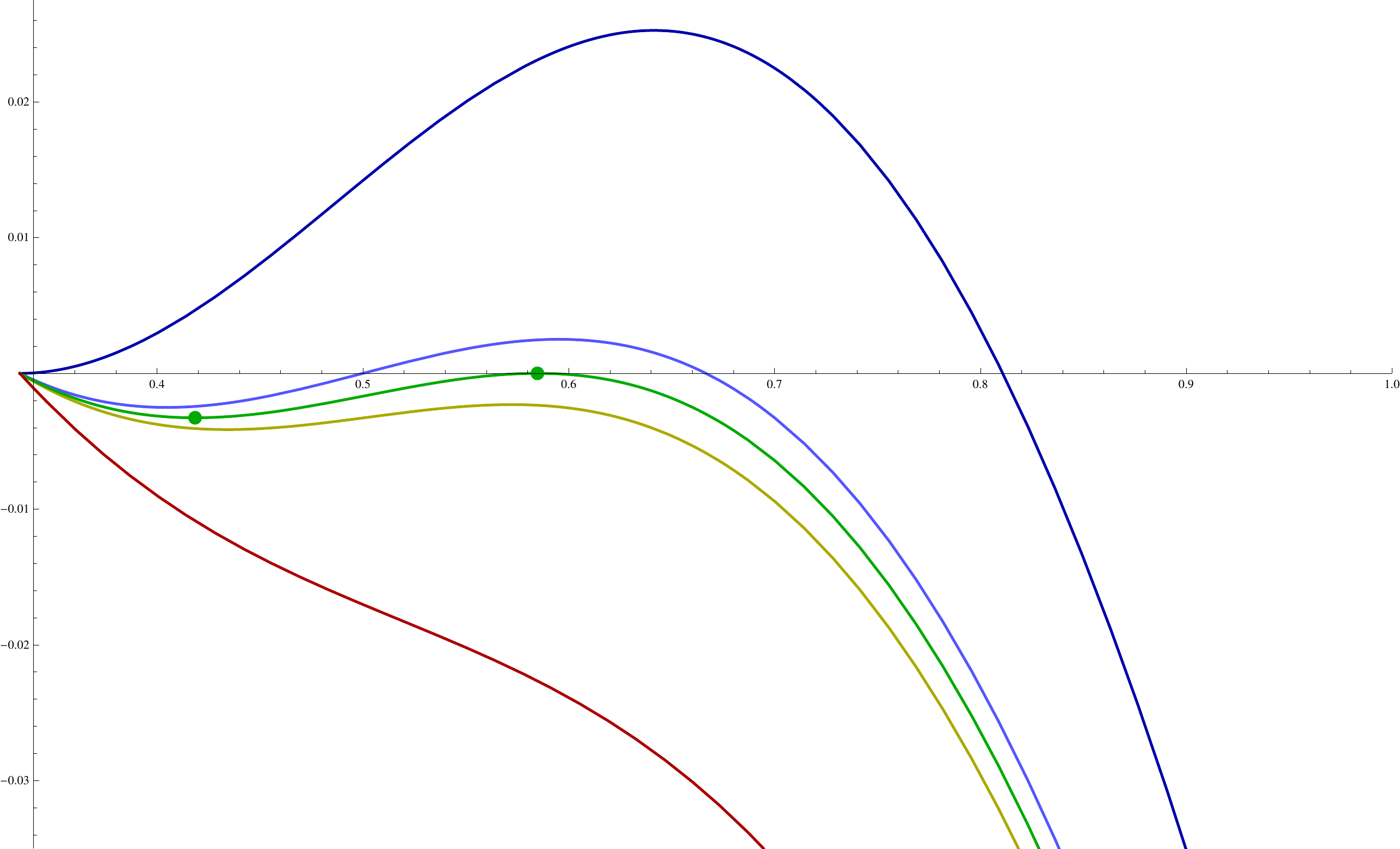}
\caption{Minimal drift towards $1/q$ of a single coordinate of $S_t$ as a function of its value for the Curie-Weiss Potts model with $q=3$ and different values of $\beta$. Two lowest curves correspond to $\beta < \beta_s$, green middle curve has $\beta=\beta_s$ (with points marking its two local extrema), second from top curve has $\beta_s < \beta < \beta_c$ and top curve has $\beta = \beta_c$.}
\label{f:drift}
\end{figure}

\subsection{Related work}

Through several decades of work by mathematicians, physicists and computer scientists a general picture of how the mixing time varies with the temperature has been developed.  It is believed that in a wide class of models and geometries the mixing time undergoes the following ``critical slowdown''. For some critical inverse-temperature $\beta_0$
and a geometric parameter $L(n)$, where $n$ is the size of the system, we should have:
\begin{itemize}[\indent$\bullet$]
  \item High temperature ($0\leq \beta< \beta_d$): mixing time of order $n \log n$ with cutoff.
  \item Critical temperature ($\beta=\beta_d$): mixing time of order $n L(n)^z$ for some fixed $z>0$.
  \item Low temperature ($\beta>\beta_d$): mixing time of order $\exp(\tau_\beta L(n))$ for some fixed $\tau_\beta>0$.
\end{itemize}
For a more comprehensive description of critical slowdown see~\cites{Martinelli97,DLPtree,LS:10}.
It should be noted that to demonstrate the above phenomenon in full, one needs to derive precise estimates on the mixing time up to the critical temperature, which can be quite challenging.

Perhaps the most studied model in this context is Ising. For the complete graph, a comprehensive treatment is given in \cites{DLP-cens,DLP,LLP}, where critical slowdown (as described above) around the uniqueness threshold $\beta_c$ is established in full. In this setting, finer statements about the asymptotics of the mixing time can be made. For instance, in~\cite{DLP} the case where $\beta$ approaches $\beta_c$ with the size of the system is analyzed (in Theorem~\ref{thm:NearCriticalMixing} here we consider this case as well).
The same picture, yet with the notable exclusion of a cutoff proof at high temperatures, is also known on the $d$-regular tree where~\cite{BKMP} established the high and low temperature regimes and recently~\cite{DLPtree} proved polynomial mixing at criticality.

From a mathematical physics point of view, the most interesting underlying graph to consider is the lattice $\mathbb{Z}^d$. For $d=2$ the full critical slowdown is now known: for a box with $n$ vertices the mixing time is $O(n\log n)$
throughout the high temperature regime~\cites{MO,MO2} whereas it is $\exp((\tau_\beta+o(1))n)$ throughout the low temperature regime~\cites{CGMS,CCS,Thomas} with $\tau_\beta$ being the surface tension. The $d=2$ picture was very recently completed by two of the authors establishing cutoff in the high temperature regime~\cite{LS:09} and polynomial mixing at the critical temperature~\cite{LS:10}.  For a more comprehensive survey of recent literature for Ising on the lattice see~\cite{LS:10}.

Turning back to the Potts model, understanding the kinetic picture here is of interest, not just as an extension of the results for Ising, but as an example of a model with a first order phase transition. Unlike in Ising, the free energy in the Potts model on various graphs and values of $q$ undergoes a first order phase transition as the temperature is varied. This is certainly true for all $q \geq 3$ in the mean-field approximation, i.e.\ on the complete graph as treated here, but also known to be the case
on $\bbZ^d$ for $d \geq 2$ and $q > Q(d)$ for some $Q(d) < \infty$ \cite{Grimmett} (although most values of $Q(d)$ are not known rigorously, it was shown that
$Q(2)=4$~\cite{Baxter} and $Q(d) < 3$ for all $d$ large enough \cite{Biskup}).

A first order phase transition has direct implications on the dynamics of the system. For one, the coexistence of phases at criticality, implies slow mixing. This is because getting from one phase to another requires passing through a large free energy barrier, i.e.\ states which are exponentially unlikely. Indeed, in \cites{BCKFTVV:99,BCT} the mixing time for Potts on a box with $n$ vertices in $\bbZ^d$ for any fixed $d \geq 2$ and sufficiently large $q$ is shown to be exponential in the surface area of the box for any $\beta$ larger or (notably) equal to the uniqueness threshold $\beta_c(d,q)$. This should be compared with the aforementioned polynomial mixing of Glauber dynamics for Ising at criticality. In fact, coexistence of the ordered and disordered phases also accounts for the slow mixing of the Swendsen-Wang dynamics at the critical temperature. This is shown in \cites{BCKFTVV:99,BCT} for $\bbZ^d$ under a similar range of $d$ and $q$ and in~\cite{GorJer:99} for the complete graph. Other dynamics also exhibit slow mixing at criticality \cite{BhaRan:04}.

First order phase transitions are expected to lead to metastability type phenomena on the lattice in some instances. There has been extensive work on this topic (see~\cite{Binder,Bovier} and the references therein)
yet the picture remains incomplete. It is expected that the transition to equilibrium will be
carried through a nucleation process, which has an $O(1)$ lifetime and therefore does not affect the order of the mixing time in contrast to the mean-field case. This is affirmed, for instance, in Ising where $O(n\log n)$ mixing time is known for low enough temperatures under an (arbitrarily small) non-zero external field, despite the first order phase transition (in the field) around $0$.
For related works see e.g.~\cites{BC,BCC,CL,SS,RTMS} as well as~\cites{Martinelli97} and the references there.

Similarly, the Potts model on the lattice should feature rapid mixing of $O(n \log n)$ throughout the sub-critical regime $\beta < \beta_c(d,q)$
due to the vanishing surface-area-to-volume ratio. Thus, contrary to the critical slowdown picture predicted for Ising, whenever there is a first order phase transition it should be accompanied by a sharp transition from fast mixing at $\beta<\beta_c$ to an exponentially slow mixing at $\beta_c$ in lieu of a critical power law.
However, fast mixing of the Potts model on $\mathbb{Z}^d$ for $\beta <\beta_c$ is not rigorously known except at very high temperatures (where it follows from standard arguments~\cite{Martinelli97}). For sufficiently high temperatures, cutoff was very recently shown in~\cite{LS:12}.

On the complete graph however, metastability is apparent. In the absence of geometry,
the order parameter sufficiently characterizes the state of the system and thus the dynamics and its stationary distribution are described effectively by the free energy of the system as a function of the order parameter. While coexistence of phases implies that at criticality the free energy is minimized at more than one value of the order parameter (corresponding to each phase), continuity entails that some of these global minima will turn into local minima just below or above criticality. These local minimizers correspond exactly to the metastable states and the value (or curve) of the thermodynamic parameter (e.g. temperature) beyond which these local minima cease to appear is called {\em spinodal}.

\begin{figure}
\vspace{-0.25cm}
\centering
\begin{tabular}{cc}
$\gb \ll \gb_c$ &
$\gb \gg \gb_c$  \\
\includegraphics[height=0.21\textheight,width=0.5\textwidth]{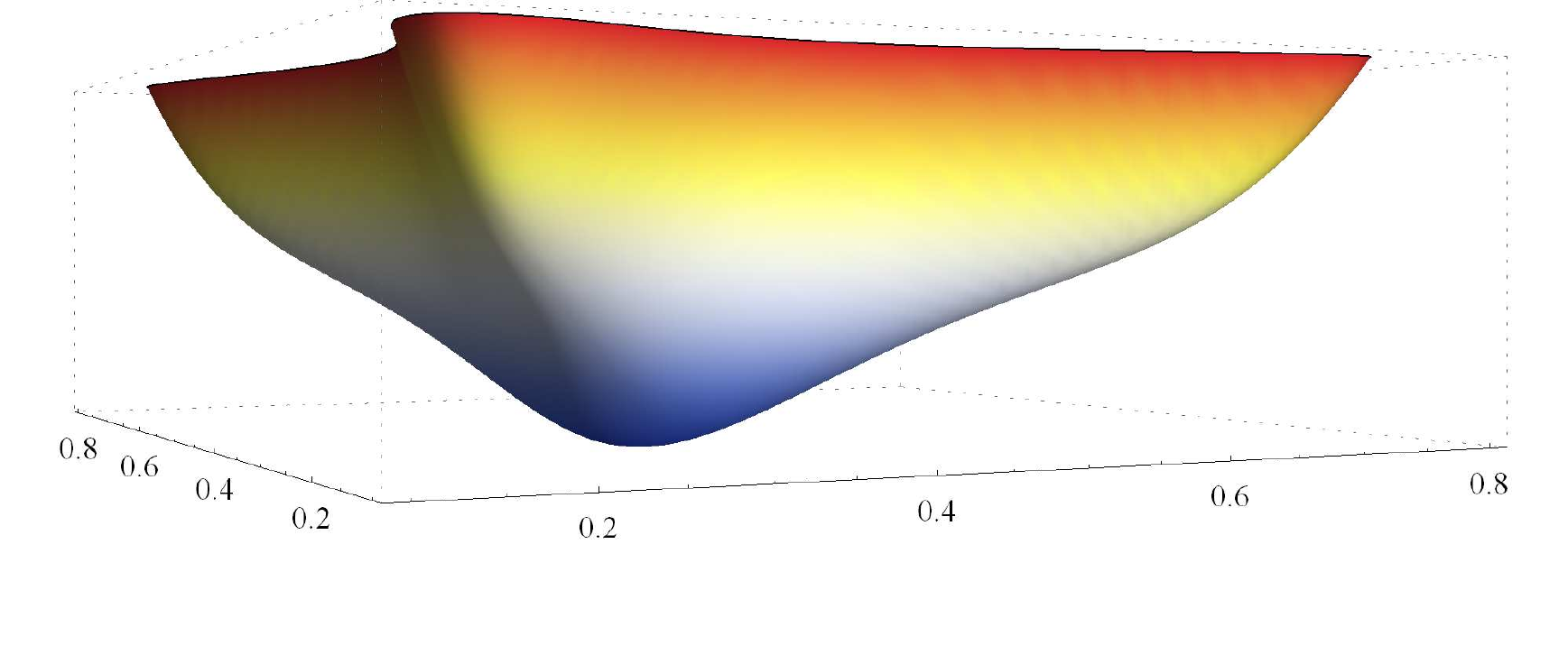} &
\includegraphics[height=0.21\textheight,width=0.5\textwidth]{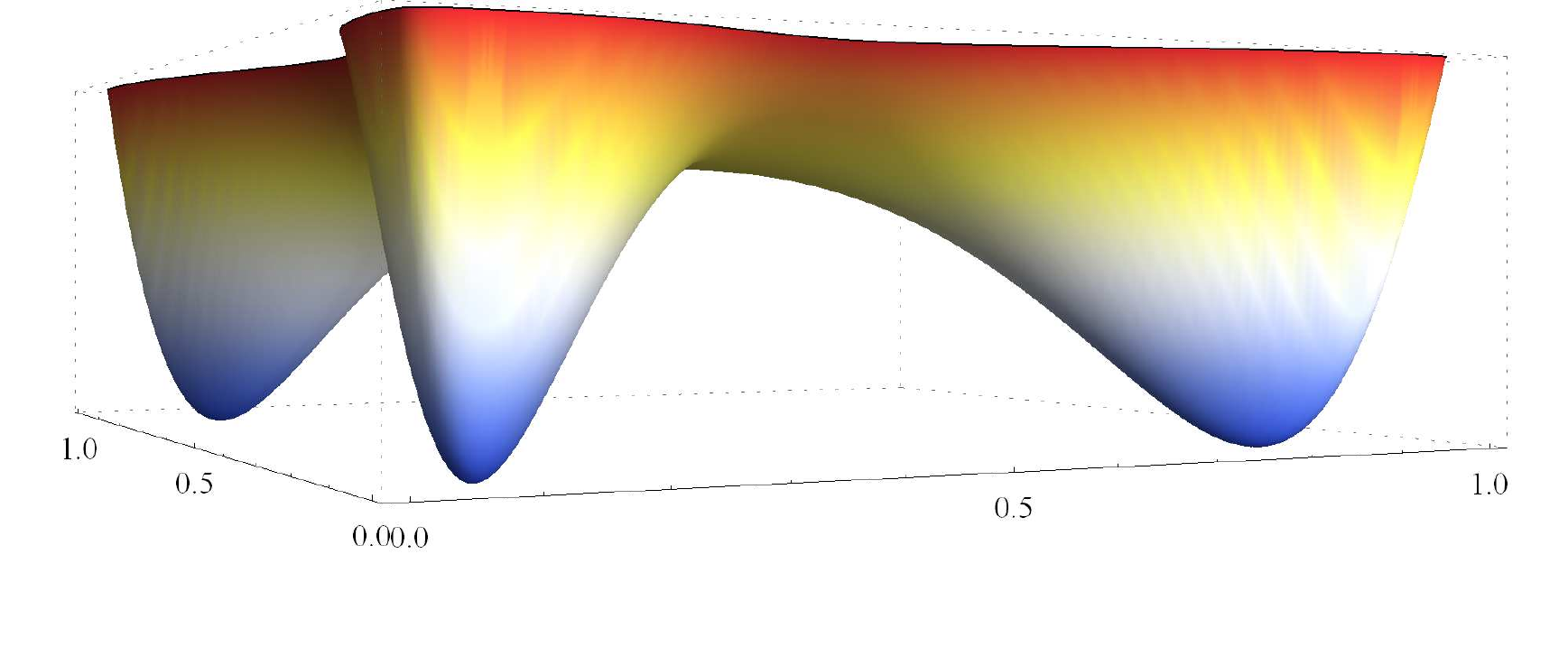} \\
\end{tabular}
%
\centering
$\gb = \gb_c$  \\
\includegraphics[height=0.21\textheight,width=0.5\textwidth]{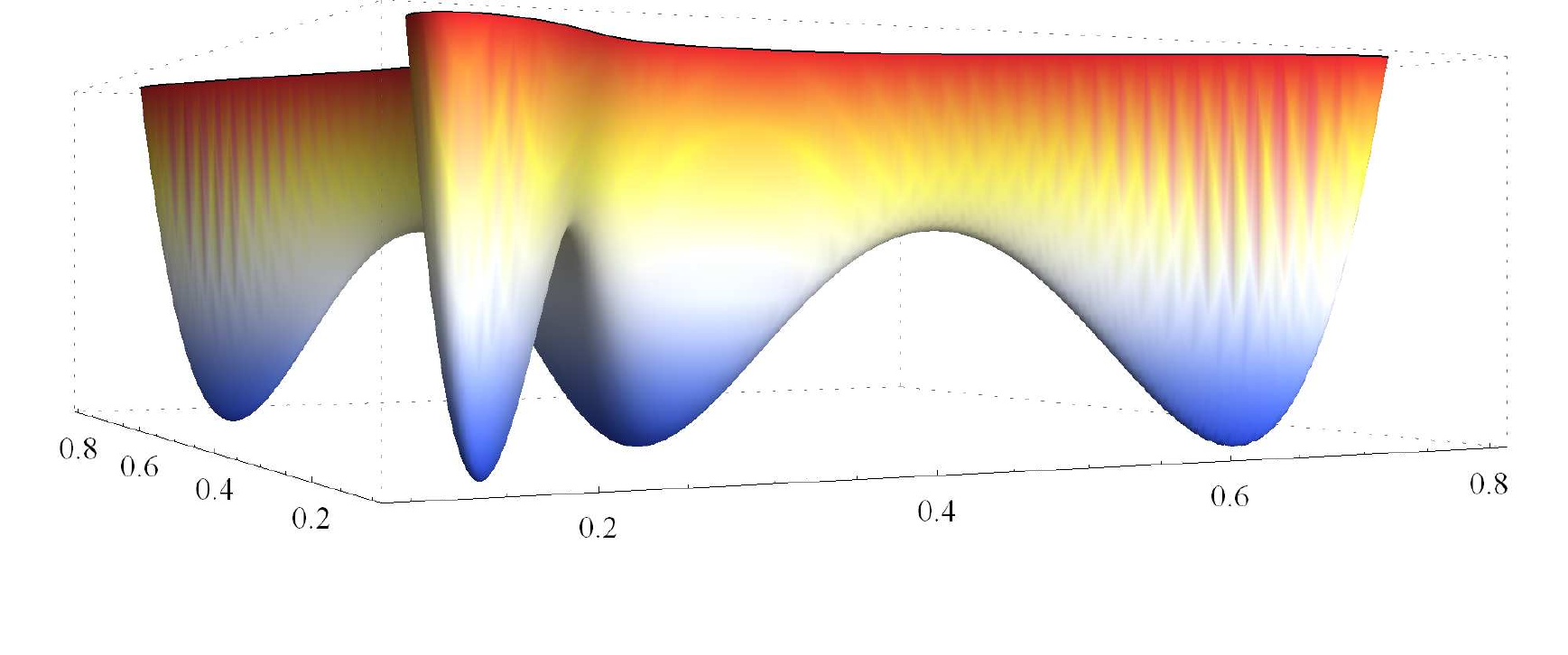}
%
\centering
\begin{tabular}{cc}
$\gb \in (\gb_s, \gb_c)$ &
$\gb \in (\gb_c, \gb_S)$ \\
\includegraphics[height=0.21\textheight,width=0.5\textwidth]{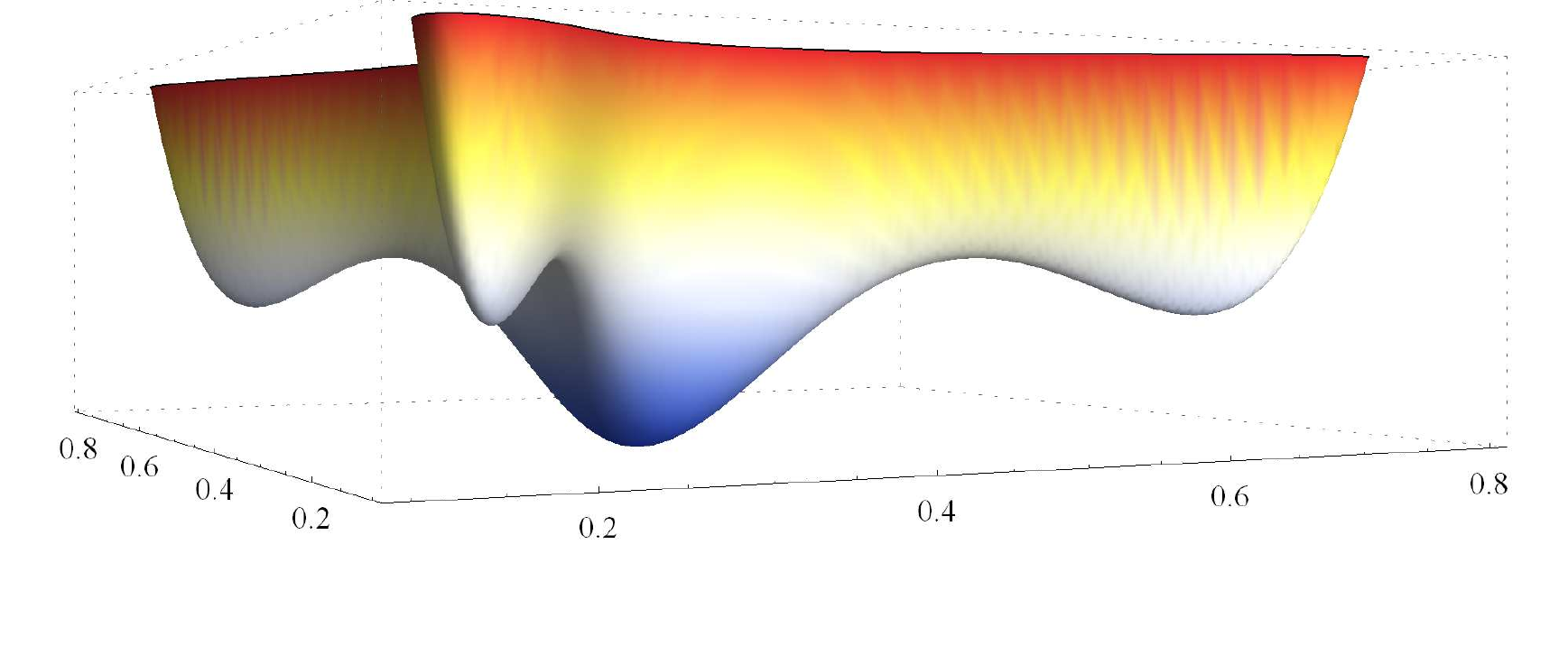} &
\includegraphics[height=0.21\textheight,width=0.5\textwidth]{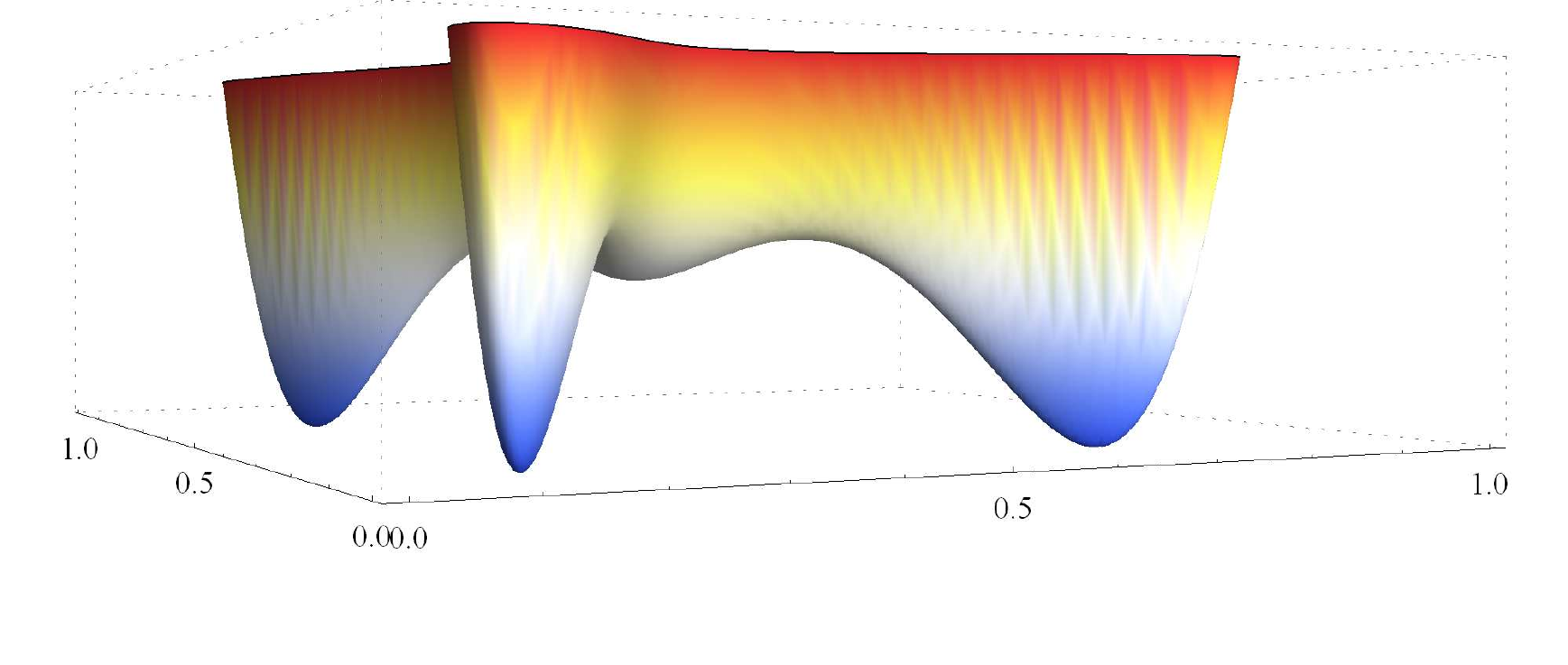}
\\
\end{tabular}
\vspace{-1cm}
\caption{The free energy as a function of the proportions vector $s$ for the Potts model on the complete graph with $q=3$ and large $n$. The simplex $\cS = \{s \in \bbR_+^3  :\, \normI{s} = 1\}$ is mapped into the $XY$ plane via $(s^1,s^2,s^3) \mapsto (s^1, s^2, 1-s^1-s^2)$.}
\vspace{-0.5cm}
\label{f:RateFunctions}
\end{figure}

Consider the model at hand, namely Potts on the complete graph
(for general background on the Potts model, see e.g.~\cites{BGJ:96,GMRZ:06,Grimmett}). The order parameter here, analogous to the magnetization in the Ising model, is the {\em vector of proportions} of each color $s \in \cS = \{x \in \bbR_+^q  :\, \normI{x} = 1\}$.
It is well known that there exists $\beta_c(q)$, below which the Potts distribution $\mu_n$ is supported almost entirely on configurations with roughly equal (about $1/q$) proportions of each color and above which $\mu_n$ is supported almost entirely on configurations where one of the $q$ colors is dominant. In the former case we say that in equilibrium the system is in the disordered phase, while in the latter case we say that in equilibrium the system is in one of the $q$ ordered phases, corresponding to the $q$-colors.

Up to relabeling of the vertices configurations are essentially described by the proportions vector $s$ and as such, on a logarithmic scale, the Potts distribution can be read from the graph of the free energy as a function of $s$. As depicted in Figure~\ref{f:RateFunctions} (showing $q=3$, the situation is qualitatively the same for all $q > 2$), when $\beta < \beta_c$ the free energy has a single global minimum at the center, corresponding to the disordered phase, while for $\beta > \beta_c$ there are $q$ ``on-axis'' global minima,
corresponding to the $q$ ordered phases obtainable from one another through a permutation of the coordinates. At $\beta_c$, coexistence of the ordered and disordered phases is evident in the presence of $q+1$ global minima of the free energy. For more details see, e.g.,~\cite{CET}.

Below $\beta_c$ but sufficiently close to it, the free energy, globally minimized only at the center, has $q$ local minima in place of the $q$ global minima
which corresponded to the ordered phases at criticality. Once $\beta$ is too small, these local minima disappear. The threshold value for the appearance of these local minima is the spinodal inverse temperature $\beta_s$ (there is a similar behavior above $\beta_c$ marked by a second spinodal temperature $\beta_S$, as illustrated by Figure~\ref{f:RateFunctions}, but we do not address this regime in the paper).	

Once the system starts from an initial configuration whose proportions vector is close to a local minimizer, the system will spend a time which is exponential in $n$ near this minimizer before escaping to the global minimizer and reaching equilibrium. This is because away from a local minimizer, energy increases locally exponentially (in $n$), i.e.\ there is an energy barrier of an exponential order to cross. Thus, as $n \to \infty$ the system will spend an unbounded amount of time at a non-equilibrium state, which will be seemingly stable. In terms of the mixing time of the dynamics, as the definition involves the worst case initial configuration, metastable states will result in exponentially slow mixing.

Our result is a rigorous affirmation of this picture. Although the definition of $\beta_s$ in~\eqref{eqn:BDDef} seems different than the one given above for the spinodal inverse temperature, it can be shown, in fact, that this is indeed the threshold value of $\beta$ for the emergence of local minima below $\beta_c$. Theorem~\ref{thm:SuperCriticalSlowMixing} then asserts that above $\beta_s$ mixing is exponentially slow while Theorem~\ref{thm:SubCriticalCutoff} shows that below $\beta_s$ mixing is still fast. The set of configurations whose exclusion in Theorem~\ref{thm:EssentialMixing} leads to fast mixing all the way up to (but below) $\beta_c$
are precisely the ones from which the process will get stuck in a metastable state. Indeed as the free energy of such initial configurations is higher than that of configurations near the globally minimizing stable state, such configurations will have a probability which is exponentially small in the size of the system.

Furthermore, the transition from fast to slow mixing passes through polynomial mixing which occurs at $\beta_s$ and in its vicinity (Theorem~\ref{thm:NearCriticalMixing}). This in fact establishes that the aforementioned critical slowdown phenomenon occurs here as well, albeit at the spinodal rather than the uniqueness threshold. We predict that this should be the case for the dynamical behavior on other mean-field geometries such as an Erd\H{o}s-R\'enyi random graph or a random regular graph.

\begin{figure}
\centering
\begin{tabular}{cc}
 \includegraphics[height=0.20\textheight,width=0.4\textwidth]{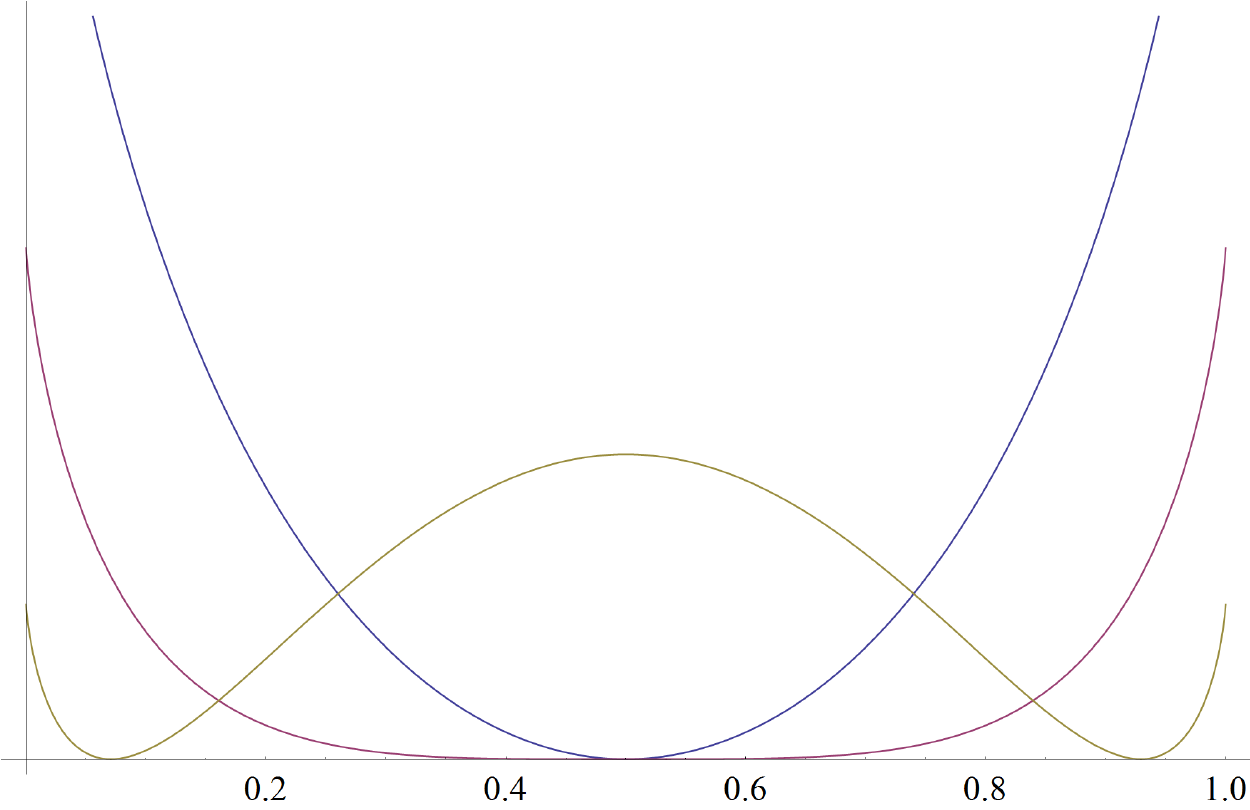} &
 \includegraphics[height=0.20\textheight,width=0.4\textwidth]{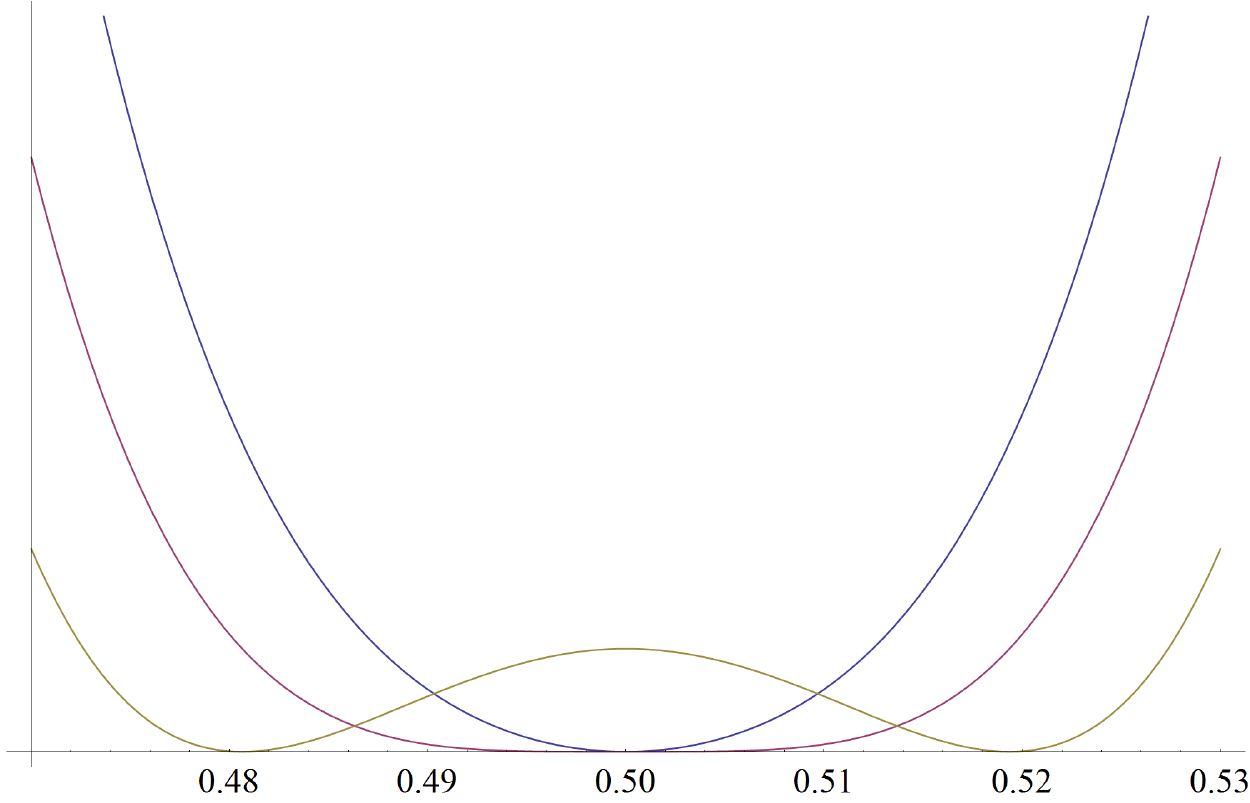} \\
 $\gb < \gb_c(2)$, $\gb = \gb_c(2)$, $\gb > \gb_c(2)$&
 $\gb < \gb_c(2)$, $\gb = \gb_c(2)$, $\gb > \gb_c(2)$
\end{tabular}
\vspace{-0.25cm}
\caption{The free energy as a function of the magnetization $m$ for the Ising model on the complete graph for large $n$. No phase coexistence at $\beta_c$ and the global maximizers for $\beta > \beta_c$ are seen to emerge continuously from $m=0.5$.}
\vspace{-0.25cm}
\label{f:RateFunctionsIsing}
\end{figure}

For a (non-rigorous) treatment of metastability and its effect on the rate of convergence to equilibrium in other mean-field models with a first order phase transition, see for example~\cites{GWL,KW}. A rigorous analysis of such a system (the Blume-Capel model), below and above criticality, was recently carried out in~\cite{KOT}.
It is illuminating to contrast the graph of the free energy as a function of the proportions vector in the Potts model to that of the free energy as a function of the magnetization in Curie-Weiss Ising, given in Figure~\ref{f:RateFunctionsIsing}. The second order phase transition and lack of phase coexistence at the critical temperature, implies the absence of a local minima at any value below or above $\beta_c$. As a consequence there is no spinodal temperature and mixing is fast throughout the whole range $\beta<\beta_c$.

\subsection{Proof Ideas}
As discussed before, up to a permutation of the vertices, configurations can be described by their proportions vector. Formally, for a configuration $\gs \in \gS_n$, we denote by $S(\gs)$ the $q$-dimensional vector $(S^1(\gs), \dots S^q(\gs))$, where
\[
	S^k(\gs) = \frac{1}{n} \sum_{v \in V} \one_{\{\gs(v)=k\}}\,.
\]
Note that $S(\gs) \in \cS$ where $\cS = \{x \in \bbR_+^q  :\, \normI{x} = 1\}$.
Now it is not difficult to see that $S_t = S(\gs_t)$ is itself a Markov process with state space $\cS$ and stationary distribution $\pi_n = \mu_n\circ S^{-1}$, the distribution of $S(\gs)$ under $\mu_n$. We shall refer to this process as the {\em proportions chain}. Figures \ref{f:subcritical} and \ref{f:supercritical} show a realization of the proportion chains superimposed on the free energy graph plotted upside down for better visibility.
3 different values of $\beta$, corresponding to 3 different regimes are exhibited. The color of the curve, representing time, shows the temporal evolution of the proportion chain.  Notice how local minima (shown as local maxima) ``trap'' the chain for a long time.

As a projection of the chain, $S_t$ mixes at least as fast as $\gs_t$. Moreover, when starting in one of the $q$ configurations where all sites have the same color, a symmetry argument reveals that the mixing time of $S_t$ is equal to that of $\gs_t$. One therefore must control the effect of starting from a initial state which is not monochrome.  Using a coupling argument we show that the difference in the mixing times is of the same order as the cutoff window for $S_t$ and can thus be absorbed into our error estimates. It will then suffice to analyze the proportions chain, which is of a lesser complexity than the original process. In particular the state space of $S_t$ has a fixed $q-1$ dimension, independently of $n$, and its transition probabilities can be easily calculated.

Next, we show that when $\gb < \gb_c(q)$ most of the mass of the stationary distribution $\pi_n$ is concentrated on {\em balanced} states whose distance from the ``equi-proportionality'' vector $(1/q, \dots, 1/q)$ is $O(1/\sqrt{n})$. A simple coupling argument then shows that the Markov chain is mixed soon after arriving at such a state.  Thus the main effort in the proof becomes finding sharp estimates on the time is takes for $S_t$ to reach a balanced state from a worst-case initial configuration, in different regimes of $\beta$.

It turns out that these hitting times are determined by the function $D_\beta(x)$, which is defined as (up to a multiplication by $1/n$) the drift of one coordinate of $S_t$ when that coordinate has value $x$ in the worst case, i.e.\ the minimum drift towards $1/q$, where the minimum is taken over all possible values for the remaining coordinates. An explicit formula for $D_\beta(x)$ can be obtained~\eqref{eqn:DOfSDef}. Its graph is plotted in Figure~\ref{f:drift} for $x \in [1/q, 1]$ and different values of $\beta$. For $\beta \ll \beta_c$, this drift is strictly negative in $(1/q, 1]$ and thus each coordinate quickly (in $O(n \log n)$ time) gets to within $1/\sqrt{n}$ of $1/q$. On the other hand, the function $D_\beta(x)$ is monotone increasing in $\beta$ and therefore for sufficiently large $\beta$, it will no longer be negative throughout $(1/q, 1]$ - there will be an interval in $(1/q, 1]$ where it is positive. Such an interval will take an exponential amount of time to traverse and this will lead to an exponential mixing time. The smallest $\beta$ for which this happens is, by definition, $\beta_s$. This $\beta_s$ in turn coincides with the inverse temperature at which local minima begin to appear in the free energy as a function of $s$. In fact, to show exponentially slow mixing, we use standard conductance arguments, which in face of local minima in the free energy give exponential mixing quite automatically.

The most delicate analysis is in the critical regime where $\beta$ is near or equal to $\beta_s$. In this case the $x$-axis is tangential to the graph of $D_\beta(x)$ at its peak (the green curve in Figure~\ref{f:drift}) and the challenge is in finding the asymptotics of the passage time through the tangential point on the $x$-axis (left green dot in the figure). As the drift there is $0$, locally around this point, a coordinate of $S_t$ behaves as a random walk and Doob's decomposition of a suitably chosen function of the coordinate yields the right passage time estimates.

\subsection{Organization}
Section~\ref{sec:prelim} sets notation and contains some useful facts on the Curie-Weiss Potts model, as well as tools (and a few non-standard variations on them) needed in the analysis of mixing time. In Section~\ref{sec:DriftAnalysis} we derive basic properties of the proportions chain that will be repeatedly used in the remainder of the proof. In Section~\ref{sec:SubcriticalRegime} we analyze the case $\gb < \gb_s(q)$ and prove Theorem~\ref{thm:SubCriticalCutoff} while Section~\ref{sec:SuperCriticalRegime} treats the case $\gb > \gb_s(q)$ and establishes Theorem~\ref{thm:SuperCriticalSlowMixing}. The near critical regime is analyzed in Section~\ref{sec:CriticalRegime}, which includes the proof of Theorem~\ref{thm:NearCriticalMixing}. The final section, Section~\ref{sec:EssentialMixing}, gives the proof of Theorem~\ref{thm:EssentialMixing}.

\section{Preliminaries}\label{sec:prelim}

\subsection{Notation}
\label{sub:Notation}
We let $[a,b]$ denote the set $\{a, \dots ,b\}$ for $a,b \in \bbZ$.
We use the same notation for vector and scalar valued variables.
For an $m$-dimensional vector $s$, we denote by $s^k$ its $k$-th component and for $I=(i_1, \dots, i_k) \subseteq [1,m]$, $s^I = (s^{i_1}, \dots, s^{i_k})$.
Matrix-valued variables will appear in bold. We let $\bW^{m,k}$ denote the $(m,k)$ element of $\bW$ and let $\bW^m$ denote is its $m$-th row.

We write $\sfe_i$ for the unit vector in the $i$-th direction and $\vq = (1/q, 1/q, \dots, 1/q) \in \bbR^q$ for the {\it equiproportionality} vector. For $s \in \bbR^{q}$,  we denote $\wh{s} = s - \wh{e}$.

Most of our vectors will live on the simplex $\cS = \{x \in \bbR_+^q  :\, \normI{x} = 1\}$ or even $\cS_n = \cS \cap \inv{n} \bbZ^q$. Occasionally we would like to further limit this set and for $\rho > 0$ we define
\[
\begin{array}{lll}
	\cS^{\rho} = \{s \in \cS : \normsup{\wh{s}} < \rho\},
	& \cS^{\rho}_n = \cS^{\rho} \cap \inv{n} \bbZ^q,
	& \gS^{\rho}_n = S^{-1}(\cS^{\rho}_n) \\
	  \cS^{\rho+} = \{s \in \cS : \: s^k < 1/q + \rho \ \  \forall k \in [1,q]\},\quad
	& \cS^{\rho+}_n = \cS^{\rho+} \cap \inv{n} \bbZ^q, \quad
	&\gS^{\rho+}_n = S^{-1}(\cS^{\rho+}_n) .
\end{array}
\]
Note that $\cS^{\rho+} \subseteq \cS^{(q-1)\rho}$ and similar relations hold for $\cS_n$ and $\gS_n$.

Vectors in $\cS$ will often be viewed also as distributions on $[1,q]$. A coupling of $\nu$, $\wt{\nu} \in \cS$ is the joint distribution of two random variables $X$, $\wt{X}$, defined on the same probability space and marginally distributed according to $\nu$, $\wt{\nu}$. If $\bbP^{\star}$ is the underlying probability measure then we always have $\normTV{\nu-\wt{\nu}} \leq \bbP^{\star}(X \neq \wt{X})$. We shall call this coupling a {\it best coupling}
if it satisfies $\normTV{\nu-\wt{\nu}} = \bbP^{\star}(X \neq \wt{X})$. Such a coupling always exists.

In the course of the proofs, we introduce various couplings of two copies of the Glauber dynamics $(\gs_t)_t$. For the second copy we shall use the notation $\wt{\gs}_t$ and $\wt{S}_t = S(\wt{\gs}_t)$. Couplings will be identified by their underlying probability measure, for which we will use the symbol $\bbP$ with a superscript that changes from coupling to coupling, e.g.\ $\bbP^{BC}$. A subscript will indicate initial state or states, e.g.\ $\bbP^{BC}_{\gs_0, \wt{\gs}_0}$. The expectation and variance, $\bbE[\cdot]$ and $\Var(\cdot)$ resp., will be decorated in the same way as the underlying measure with respect to which they are defined. The $\sigma$-algebra $\cF_t$ will always include all the randomness up-to time $t$. For example, with a single chain $(\gs_t)_t$ this is the $\gs$-algebra generated by $\{\gs_s :\: s \leq t\}$, for a coupling of two chains $(\gs_t)_t$, $(\wt{\gs}_t)_t$, it is the one generated by $\{\gs_s,\wt{\gs}_s :\: s \leq t\}$, etc.

\subsection{Large Deviations Results for the Curie-Weiss Potts Distribution}
\label{sub:CWLDP}
In this subsection we recall several results concerning the concentration of the
proportions vector measures $\pi_n$. See, e.g.,~\cites{CET,EW} for proofs of these results.

It is a consequence of Sanov's Theorem together with an application of Varadhan's Lemma that the sequence
$(\pi_n)_{n \geq 1}$ satisfies a large deviation principle (LDP) on $\cS$ with rate function
\begin{equation}
\label{eqn:CWRateFnDef}
I_{\gb, q}(s) = \sum_{k=1}^q s^k \log (q s^k) - \gb \normII{s}^2 - C,
\end{equation}
where $C$ is chosen so that $\min_{s \in \cS} I_{\gb, q}(s) = 0$.
The minimizing set $\gvep_{\gb, q} = \{s : \, I_{\gb, q}(s) = 0 \}$ which is
the support of the weak limit $\frp_{\gb, q}$ of $(\pi_n)_{n \geq 1}$ is then
\begin{equation}
\label{eqn:MinOfPottsRateFn}
\gvep_{\gb, q} = \left\{
    \begin{array}{ll}
        \lbr \vq \rbr                                                           & \textrm{if $\gb < \gb_c(q)$} \\
        \lbr \vq, \rmT^1 \check{s}_{\gb_c, q}, \rmT^2 \check{s}_{\gb_c, q},
            \dots , \rmT^q \check{s}_{\gb_c, q} \rbr                           & \textrm{if $\gb = \gb_c(q)$} \\
        \lbr \rmT^1 \check{s}_{\gb, q}, \rmT^2 \check{s}_{\gb, q},
            \dots , \rmT^q \check{s}_{\gb, q} \rbr                             & \textrm{if $\gb > \gb_c(q)$}
    \end{array}
\right.,
\end{equation}
where
\begin{equation}
\label{eqn:GbCFormula}
\gb_c(q)=\frac{(q-1)\log(q-1)}{q-2}
\end{equation}
and
\begin{equation}
\label{eqn:SCheckDef}
\check{s}_{\gb, q} = \lb \check{s}_{\gb, q}^1,
\frac{1 - \check{s}_{\gb, q}^1}{q-1}, \dots, \frac{1 - \check{s}_{\gb, q}^1}{q-1} \rb.
\end{equation}
The function $\gb \mapsto \check{s}_{\gb, q}^1$ is continuous and increasing
on $[\gb_c(q), \infty)$ and $\rmT^k:\, \cS \to \cS$ interchanges the 1-$st$ and $k$-th coordinates.
Furthermore, the value of $\check{s}_{\gb, q}$ for all $\gb, q$ is known in
implicit form and in particular for $\gb = \gb_c(q)$, we have
\begin{equation}
\label{eqn:SCheckCritDef}
\check{s}_{\gb_c(q), q} = \lb 1-\tfrac{1}{q}, \tfrac{1}{q(q-1)}, \dots, \tfrac{1}{q(q-1)} \rb.
\end{equation}
This is true for all $q \geq 3$. For $q=2$, \eqref{eqn:CWRateFnDef},\eqref{eqn:MinOfPottsRateFn}, \eqref{eqn:SCheckDef}, \eqref{eqn:SCheckCritDef} still hold, but the critical inverse-temperature is now
\begin{equation}
\label{eqn:GbCFormulaQ2}
\gb_c(2) = 1.
\end{equation}

It is here that a fundamental difference between $q=2$ and $q>2$ can be observed. If $q=2$ then $\check{s}_{\gb_c(2), 2} = \vq$, in which case $\gb \mapsto \frp_{\gb, 2}$ is
continuous for all $\gb \geq 0$. On the other hand, if $q \geq 3$ we have
$\check{s}_{\gb_c(q), q} \neq \vq $ and $\gb \mapsto \frp_{\gb, q}$
is discontinuous at $\gb_c(q)$. Thus, as it is recorded in the Physics literature, the system
exhibits a first order phase transition if $q \geq 3$, but only a second order phase transition if $q = 2$.

\subsection{Hitting Time Estimates for General Supermartingales}
We will require some standard hitting time estimates for supermartingales and related processes.
\begin{lem}
\label{lem:ImprovedAzuma}
For $x_0 \in \bbR$, let $(X_t)_{t \geq 0}$ be a discrete time process, adapted
to $(\cF_t)_{t \geq 0}$ which satisfies
\begin{enumerate}
\item
	\label{cond:ImprovedAzumaLem1}
	$\exists \gd \geq 0 \,: \quad
		\bbE_{x_0} \lsb \lpr X_{t+1} - X_t \rabs \cF_t \rsb \leq -\gd$ \textbf{on $\lbr X_t \geq 0
\rbr$} for all $t \geq 0$.
\item
	\label{cond:ImprovedAzumaLem2}
	$\exists R > 0 \,:\quad \labs X_{t+1} - X_t \rabs \leq R ,\, \forall t \geq 0$.
\item
	\label{cond:ImprovedAzumaLem3}
	$X_0 = x_0$.
\end{enumerate}
where $\bbP_{x_0}$ is the underlying probability measure.
Let $\tau_x^- = \inf \{ t : \: X_t \leq x \}$ and
$\tau_x^+ = \inf \{ t : \: X_t \geq x \}$.
The following holds.
\begin{enumerate}
\item
	\label{item:ImprovedAzuma1}
	If $\gd > 0$ then for any $t_1 \geq 0$:
	\begin{equation}
		\label{eqn:AzumaLem1}
		\bbP_{x_0} \lb \tau_0^- > t_1 \rb \leq
		  \exp \lbr -\frac{\lb \gd t_1 - x_0 \rb^2}{8 t_1 R^2} \rbr \, .
	\end{equation}
\item
	\label{item:ImprovedAzuma2}
	If $x_0 \leq 0$, then for any $x_1>0$ and $t_2\geq 0$,
	\begin{equation}
	\label{eqn:AzumaLem2}
	\bbP_{x_0} \lb \tau_{x_1}^+ \leq t_2 \rb \leq 2 \exp \lbr -\frac{(x_1-R)^2}{8t_2 R^2} \rbr .
	\end{equation}
\item
	\label{item:ImprovedAzuma3}
	If $x_0 \leq 0$, $\gd > 0$, then for any $x_1>0$ and $t_3 \geq 0$:
	\begin{equation}
	\label{eqn:ImprovedAzuma3}
	\bbP_{x_0} \lb \tau_{x_1}^+ \leq t_3 \rb \leq t_3^2 \exp \lbr
		-\frac{(x_1-R) \gd^2}{8 R^3} \rbr .
	\end{equation}
\end{enumerate}
\end{lem}
\begin{proof}
Starting with~\eqref{item:ImprovedAzuma1}, if $x_0 < 0$, there is nothing to prove. Otherwise, let $Z_t$ be a supermartingale independent of $X_t$, which starts from $0$, has drift $-\gd$ and steps which are bounded by $R$. Set $Y_t = X_{t \minwith \tau_0^-} + Z_{\lb t -\tau_0^-\rb^+}$ and write $Y_t = M_t + A_t$ as its Doob Decomposition, with $M_t$ a martingale, $A_t$ a predictable processes and $A_0 = x_0$. Clearly $A_t \leq x_0 -\gd t$ and $|M_{t+1} - M_t| \leq 2R$, $\bbP_{x_0}$-a.s. The Hoeffding-Azuma inequality implies:
\begin{eqnarray*}
	\bbP_{x_0} \lb \tau_0^- > t_1 \rb
		& \leq 	& \bbP_{x_0} \lb Y_{t_1} > 0 \rb \\
		& \leq	& \bbP_{x_0} \lb M_{t_1} > \gd t_1 - x_0 \rb \\
		& \leq	& \exp \lbr -\frac{\lb \gd t_1 - x_0 \rb^2}{8 t_1 R^2} \rbr
\end{eqnarray*}
as desired.

Now set $Y_t = X_{t \minwith \tau_{x_1}^+}$ and observe that $\lb Y_t \rb_{t \geq 0}$ satisfies conditions \ref{cond:ImprovedAzumaLem1}--\ref{cond:ImprovedAzumaLem3}, and consequently it is enough to prove
\eqref{eqn:AzumaLem2}, with $\bbP_{x_0} \lb Y_{t_2} > x_1 \rb$ as the LHS. Therefore, for all $t$, let $W_{t+1} = \bbE_{x_0} \lsb \lpr Y_{t+1} - Y_t \rabs \cF_t \rsb$ and $Z_{t+1} = Y_{t+1} - Y_t - W_{t+1}$. Then clearly, $W_{t+1}$, $Z_t$ are $\cF_t$-measurable,
$W_{t+1} \leq 0$ on $\lbr Y_t \geq 0 \rbr$ and $\bbE_{x_0} \lsb Z_{t+1} | \cF_t \rsb = 0$. This is a
Doob-type decomposition. Now, define $M_t$ for all $t$ inductively as follows.
\[
	M_0 = 0	\ ; \quad M_{t+1} = M_t + \hbox{sign}(M_t) Z_{t+1}
\]
and set $N_t = |M_t|$. We claim the following:
\begin{enumerate}
\item
	$\lb M_t \rb_{t \in \bbN}$ is an $\lb \cF_t \rb_{t \geq 0}$-adapted martingale.
\item
	$N_0 = 0 \ ; \quad N_{t+1} = \labs N_t + Z_{t+1} \rabs ,\, \forall t$.
\item
	$Y_t \leq N_t + R$.
\end{enumerate}
The first two assertions follow straightforwardly from the construction. The third one, can be
proven by induction, since it clearly holds for $t=0$ and assuming $Y_{t-1} \leq N_{t-1} + R$,
if $Y_{t-1} \geq 0$, then
\begin{eqnarray*}
	Y_t
		& = 	& Y_{t-1} + (Y_t - Y_{t-1})	\\
		& \leq	& N_{t-1} + R + Z_t \leq N_t + R ,
\end{eqnarray*}
and if $Y_{t-1} < 0$, then
\[
	Y_t \leq R \leq N_t + R .
\]

Finally, by the Hoeffding-Azuma inequality applied to $M_t$ we have
\begin{eqnarray*}
	\bbP_{x_0} \lb Y_{t_2} > x_1 \rb
		& \leq	& \bbP_{x_0} \lb |M_{t_2}| > x_1 - R \rb \\
		& \leq 	& 2 \exp \lb -\frac{(x_1-R)^2}{8t_2 R^2} \rb .
\end{eqnarray*}
This shows \eqref{eqn:AzumaLem2}.

For part \eqref{item:ImprovedAzuma3},
let $\tau_0^-(s) = \inf \{t \geq s: X_t \leq 0 \}$
and $Y_t(s)$, $M_t(s)$ and $A_t(s)$ defined as in the proof of part
\eqref{item:ImprovedAzuma1} only with $\tau_0^-(s)$ replacing $\tau_0^-$.
Then,
\begin{eqnarray*}
\bbP_{x_0} \lb \tau_{x_1}^+ \leq t_3 \rb
	& 	\leq	& \sum_{0\leq s_1 < s_2 \leq t_3} \bbP_{x_0} \lb
		X_{s_1} \leq R, \, X_{s_2} \geq x_1 ,\, X_t > 0 \;
			\forall s_1 \leq t \leq s_2 \rb \\
	& 	\leq	& \sum_{0 \leq s_1 < s_2 \leq t_3} \bbP_{x_0} \lb
		 Y_{s_2}(s_1) \geq x_1, Y_{s_1}(s_1) \leq R \rb	\\
	& 	\leq	& \sum_{0\leq s_1 < s_2 \leq t_3} \bbP_{x_0} \lb
		 	M_{s_2}(s_1) - M_{s_1}(s_1) \geq \gd(s_2 - s_1) + (x_1 - R)^+ \rb \\
	&  \leq & t_3^2 \exp \lbr -\frac{\gd^2 \frac{x_1 - R}{R}}{8 R^2} \rbr
\end{eqnarray*}
where the last inequality follows from Hoeffding-Azuma inequality and since the summands are non zero only when
$0 \leq s_1 < s_2 \leq t_3$ and $(s_2 - s_1)R \geq x_1 - R$.
\end{proof}

\begin{lem}
\label{lem:HittingTimeLowerBound}
Let $(X_t)_{t \in \bbN}$ be a process adapted to $(\cF_t)_{t \in \bbN}$ and
satisfying the following conditions for some $a > 2\gd \geq 0$:
\begin{enumerate}
\item
	\label{item:HittingTimeLowerBoundLem1}
	$X_{t+1} - X_t \in \{-1, 0, 1\}$.
\item
	\label{item:HittingTimeLowerBoundLem2}
	$\bbE[X_{t+1} - X_t | \cF_t] \geq -\gd$.
\item
	\label{item:HittingTimeLowerBoundLem3}
	$\Var(X_{t+1} | \cF_t) \geq a$
\item
	\label{item:HittingTimeLowerBoundLem4}
	$X_0 \geq 0$
\end{enumerate}

Let $\tau_r^{+} = \inf \{t: X_t \geq r \}$. Then
\begin{equation}
\label{eqn:HittingTimeLowerBound}
\bbP (\tau_r^{+} \leq t) \geq C_1 \exp \{-C_2 (r/\sqrt{t} + \gd \sqrt{t})^2 \} + O(t^{-1/2})
\end{equation}
where $C_1, C_2$ are positive constants which depends only on $a$.

\end{lem}
\begin{proof}
It is easy to verify that conditions \ref{item:HittingTimeLowerBoundLem1}--\ref{item:HittingTimeLowerBoundLem3}
imply
\begin{equation}
\label{eqn:RandP}
\bbP (X_{t+1} \neq X_t | \cF_t) \geq a	\quad ; \quad
	\bbP (X_{t+1} - X_t = -1 | \cF_t ,\, X_{t+1} \neq X_t) \leq \inv{2} + \tfrac{\gd}{2a} \,.
\end{equation}
Now, let $T_0 = 0$ and $T_k = \inf \{t > T_{k-1} : \: X_t \neq X_{T_{k-1}} \}$,
$Y_k = X_{T_k}$ for $k \geq 1$. Also, let $N_k = T_k - T_{k-1}$ and $Z_k = Y_k - Y_{k-1}$.
From \eqref{eqn:RandP}, it is not hard to see that we can couple $(N_k)_k$, $(Z_k)_k$ with
two i.i.d sequences $(\wt{N}_k)_k$,$(\wt{Z}_k)_k$ such that $\wt{N}_k \geq N_k$, $\wt{Z}_k \leq Z_k$
a.s. and
\[
\wt{N}_k \sim \mbox{Geom}(a) \quad ; \quad
	\wt{Z}_k = \lbr \begin{array}{ll}
	                 	+1 \ & \text{w.p } 1/2 - \gd/2a	\\
						-1 & \text{w.p } 1/2 + \gd/2a
	                \end{array} \rpr \,.
\]
Consequently, with $\wt{Y}_k = \sum_{m \leq k} \wt{Z}_m$ and $\wt{T}_k = \sum_{m \leq k} \wt{N}_m$
we have:
\begin{eqnarray*}
\bbP (\tau_r^{+} \leq t)
	& \geq		& \bbP (Y_{ta/2} \geq r ,\, T_{ta/2} \leq t)	\\
	& \geq 		& \bbP (\wt{Y}_{ta/2} \geq r) - \bbP(\wt{T}_{ta/2} > t)	\\
	& \geq		& C_1 \exp \lbr -C_2 \frac{(r+\gd t)^2}{t} \rbr + O(t^{-1/2}) + \exp \{-C_3 t \} \,,
\end{eqnarray*}
where the last inequality follows by the local CLT for $\wt{Y}_{k}$, which is a nearest neighbor random walk whose steps have mean $-\gd/a$ and variance $1-(\gd/a)^2$ and also by Cramer's theroem for $\wt{T}_k$. All constants depend only on $a$.
\end{proof}

Finally, we will also make use of the following result from~\cite{LPW}:
\begin{lem}[\cite{LPW}*{Proposition 17.20}]
\label{lem:DiffusionHittingTime}
Let $\lb Z_t \rb_{t \geq 0}$ be a non-negative supermartingale
adapted to $\lb \cG_t \rb_{t \geq 0}$ and $N$ a stopping time. Suppose that:
\begin{enumerate}
	\item $Z_0 = z_0$
	\item $|Z_{t+1} - Z_t| \leq B$
	\item $\exists \gs > 0$ such that $\Var \lb Z_{t+1} | \cG_t \rb > \gs^2$
		  on the event $\lbr N > t \rbr$.
\end{enumerate}
If $u > 4B^2/(3\gs^2)$, then:
\[
	\bbP_{z_0} \lb N > u \rb \leq \frac{4 z_0}{\gs \sqrt{u}}.
\]
\end{lem}

\subsection{Variance Lemma}
The following is a straight-forward extension of \cite{LLP}*{Lemma 2.6} to vector valued Markov processes. We include a proof here for completeness.
\begin{lem}
\label{lem:VarBoundFromContraction}
Let $\lb Z_t \rb$ be a Markov chain taking values in $\bbR^d$ and with transition matrix
$P$. Write $\bbP_{z_0}$, $\bbE_{z_0}$ for its probability measure
and expectation respectively, when $Z_0 = z_0$.
Suppose that there is some $0 < \eta < 1$ such that for all pairs of starting states
$\lb z_0, \wt{z}_0 \rb$,
\begin{eqnarray}
\label{eqn:VarBoundContractionCond}
\normII{\bbE_{z_0} Z_t - \bbE_{\wt{z}_0} Z_t} \leq \eta^t \normII{z_0 - \wt{z}_0}
\end{eqnarray}
Then $v_t \bydef \sup_{z_0} \Var_{z_0} \lb Z_t \rb = \sup_{z_0} \bbE_{z_0} \normII{Z_t - \bbE_{z_0} Z_t}^2$ satisfies:
\begin{equation}
\label{eqn:VarBoundContractionResult}
v_t \leq v_1 \min \lbr t, \lb 1-\eta^2 \rb^{-1} \rbr\,.
\end{equation}
\end{lem}

\begin{proof}
Let $Z_t$ and $Z_t^{*}$ be independent copies of the chain with the same starting state $z_0$. By
the assumption \eqref{eqn:VarBoundContractionCond}, we obtain that
\[\normII{\bbE_{z_0}[Z_t \mid Z_1 = z_1] - \bbE_{z_0}[Z_t^{*} \mid Z_1^* = z_1^*]} = \normII{\bbE_{z_1}[Z_{t-1}] - \bbE_{z_1^*} [Z_{t-1}^*]\|_2 \leq \eta^{t-1} \|z_1 - z_1^*}\,.\]
Hence, we see that
\[\Var_{z_0} (\bbE_{z_0}[Z_t \mid Z_1]) = \frac{1}{2}\bbE_{z_0}\normII{(\bbE_{Z_1}[Z_{t-1}] - \bbE_{Z_1^*}[Z_{t-1}^*])}^2 \leq \frac{\eta^{2(t-1)}}{2}\bbE_{z_0}\normII{Z_1 - Z_1^*}^2 \leq \eta^{2(t-1)}v_1.\]
Combined with the total variance formula, it follows that
\[v_t \leq \sup_{z_0}\{\bbE_{z_0}[\Var_{z_0}(Z_t \mid Z_1)] + \Var_{z_0}(\bbE_{z_0}[Z_t \mid Z_1])\} \leq v_{t-1} + \eta^{2(t-1)} v_1\,,\]
which then gives that $v_t = \sum_{i=1}^{t}(v_{i} - v_{i-1}) \leq \sum_{i=1}^{t} \eta^{2(t-1)}v_1$, implying the desired upper bound immediately.
\end{proof}

\subsection{Bottleneck Ratio}
Let $P$ be an irreducible, aperiodic transition kernel for a Markov chain on $S$ with stationary measure $\pi$. The {\it bottleneck ratio} of a set $A \subseteq S$ is:
\[
\Phi(A) = \frac{\sum_{x \in A, y \notin A} \pi(x) P(x,y)}{\sum_{x \in A} \pi(x)}
	\leq \frac{\pi(\partial_{P} A)}{\pi(A)},
\]
where $\partial_{P} A = \{x \in A \;:\ P(x,y) > 0 \text{ for some } y \notin A\}$.
The {\it bottleneck ratio} of the chain is
\begin{equation}
  \label{eq-bottleneck-ratio-def}
  \Phi_* = \min_{A : \pi(A) \leq \inv{2}} \Phi(A) \,.
\end{equation}
The following result, due to~\cites{AM,LawlerSokal,SJ} in several similar forms (see, e.g.,~\cite{LPW}*{Theorem 7.3}) relates the bottleneck ratio with the mixing time of the chain.
\begin{thm}
\label{thm:CheegersInequality}
If $\Phi_*$ is the bottleneck ratio defined in~\eqref{eq-bottleneck-ratio-def} then $
	t_{\mix(1/4)} \geq \frac{1}{4\Phi_*} $.
\end{thm}

\section{Drift Analysis for the Proportions Chain}
\label{sec:DriftAnalysis}
In this section we prove various results concerning the drift of the process $S_t$. We analyze both the one coordinate process $S^1_t$ and the distance-to-equiproportionality $\normII{S_t - \vq}$. In the course of this analysis, we also define two couplings which will be of independent use later on and prove a uniform bound on the variance of $S_t$.

\subsection{The Drift of One Proportion Coordinate}
\label{sub:OneCoordinateDrift}
From symmetry, it is enough to analyze the drift of $S^1_t$.
For $\gb \geq 0$, define $g_{\gb}:\cS \to \cS$ as
\[
g_{\gb}(s) = (g^1_{\gb}(s), \dots, g^q_{\gb}(s))
\quad \text{;} \quad
g_{\gb}^k(s) = \frac{e^{2\gb s^k}}{\sum_{j=1}^q e^{2\gb s^j}}
\,.
\]
We can express the drift of $S_t^1$ as follows:
\begin{eqnarray}
\nonumber
\bbE \lsb \left. S^1_{t+1} - S^1_t \rabs \cF_t \rsb
    & =	& \frac{1}{n} \Big[ -S_t^1 + \sum_{k=1}^q g_{\gb}^1 \lb S_t - \nonumber
\tfrac{1}{n} \sfe_k \rb S_t^k \Big] \\
\nonumber
    & =	& \frac{1}{n} \lsb -S_t^1 + g_{\gb}^1 \lb S_t \rb \rsb + O \lb n^{-2} \rb \\
\label{eqn:OneCoordinateDrift}
    & =	& \frac{1}{n} d_{\gb}(S_t) + O \lb n^{-2} \rb
\end{eqnarray}
with
\begin{equation}
\label{eqn:DOfSDef}
d_{\gb}(s) \bydef -s^1 + g_{\gb}^1(s).
\end{equation}
The function $d_\gb: \cS \to \bbR$ thus describes (up
to a constant factor of $n^{-1}$ and an error term) the {\it drift} of the first coordinate
given the current proportions vector. It turns out the
rapid mixing hinges on whether $d_{\gb}(s)$ is strictly negative whenever $s^1 > 1/q$
(and for any values for the remaining coordinates of $s$). Accordingly we define $D_{\gb}:[0,1] \to \bbR$ as
\begin{align}
\label{eqn:LittleDOfBetaDef}
D_{\gb}(x) &\bydef    \maxtwo{s \in \cS}{s^1 = x} d_{\gb}(s)
= d_{\gb}\big( x, \tfrac{1-x}{q-1}, \dots, \tfrac{1-x}{q-1} \big)\\
&  =    -x + \frac{\exp(2\gb x)}{\exp(2\gb x) + (q-1) \exp(2\gb \tfrac{1-x}{q-1})}
\end{align}
and check when $D_{\gb}(x)$ is strictly negative for all $x \in (1/q, 1]$.
We will see in Proposition~\ref{prop:DOfBetaProperties1} below that this happens
if and only if $\gb < \gb_s(q)$, where $\gb_s(q) > 0$ is defined in
\eqref{eqn:BDDef}.

For $\gb \geq 0$, define $s^*(\gb)$ and $s^\sharp(\gb)$ as:
\begin{align}
\nonumber
s^*(\gb) &\bydef \sup \lbr s \in \lsb 1/q,1 \rb : \: \tfrac{d}{ds} D_{\gb}(s) = 0 \rbr\\
\nonumber
s^\sharp(\gb) &\bydef \inf \lbr s \in \lb 1/q, 1 \rb : \: D_{\gb}(s) \geq 0 \rbr ,
\end{align}
with the $\inf$ or $\sup$ being $1$, if the respective sets are empty. We may now state:
\begin{prop}
\label{prop:DOfBetaProperties1}
For all $q \geq 3$ the following holds:
\begin{enumerate}
\item
	\label{item:DOfBetaProp1:1}
	We have that $D_{\gb}(s)$ is increasing in $\gb$ if $s \in \lsb 1/q, 1 \rsb$
	and for all $\gb \geq 0$, that $D_{\gb} (1/q) = 0$ and that $D_{\gb}(1) < 0$.
\item
	\label{item:DOfBetaProp1:2}
	That $s^\sharp(\gb) > \inv{q}$ if $\gb < q/2$.
\item
	\label{item:DOfBetaProp1:3}
	The following statements are equivalent if $\gb < q/2$:
	\begin{enumerate}
	\item\label{item-i}
		$s^\sharp(\gb) = 1$
	\item\label{item-ii}
		$D_{\gb}(s)$ has no roots in $(1/q, 1]$.
	\item\label{item-iii}
		$D_{\gb}(s^*(\gb)) < 0$.
	\end{enumerate}
\item
	\label{item:DOfBetaProp1:4}
	All the statements in part \eqref{item:DOfBetaProp1:3} hold
	if and only if $\gb < \gb_s(q)$.	
\end{enumerate}
\end{prop}
\begin{proof}
We start by proving part \eqref{item:DOfBetaProp1:1}. It is clear that
\[D_\gb(s) = -s + \frac{1}{1 + (q-1) \exp(2\gb \tfrac{1 - qs}{q-1})}\,,\]
and hence is strictly increasing in $\gb$ if $1/q<  s \leq
1$. Furthermore, we have that
\[D_\gb(1/q) = -\frac{1}{q} + \frac{\mathrm{e}^{2\gb/q}}{ \mathrm{e}^{2\gb/q} + (q-1) \mathrm{e}^{2\gb/q}} = -\frac{1}{q} + \frac{1}{q} = 0\,,\]
as well as
\[D_\gb(1) = -1 + \frac{1}{1 + (q-1)\mathrm{e}^{-\gb}} < 0\,.\]

For part \eqref{item:DOfBetaProp1:2}, taking the derivative of $D_{\gb}(s)$ with respect to $s$ evaluated at $1/q$, one obtains $\lpr \frac{d}{ds} D_{\gb}(s) \rabs_{s=1/q} = -1 + \frac{2\gb}{q}$ which is negative if $\gb < q/2$. Together with $D_{\gb}(1/q) = 0$, this completes the proof.

For part \ref{item:DOfBetaProp1:3}, we first show that there exists at most two points in
$[1/q, 1]$ such that $\frac{d}{d s}D_\gb(s) = 0$. To see
this, we compute the first derivative and obtain that
\[\frac{d}{ds}D_\gb(s)  = -1+ \frac{2q\gb \exp(2\gb\frac{1 - qs}{q-1})}{[1+ (q-1) \exp(2\gb\frac{1-qs}{q-1})]^2} = -1 + \frac{2q\gb}{q-1}h\big((q-1) \mathrm{e}^{2\gb\frac{1-qs}{q-1}}\big)\,,\]
where $h(x) =
\frac{x}{(1+x)^2}$. Obviously, there are at most two zeros for $-1 +
\frac{2q\gb}{q-1}h(x)$ and since $(q-1)
\exp(2\gb\frac{1-qs}{q-1})$ is a strictly monotone in $s$,
we conclude that there are at most two points such that
$\frac{d}{ds}D_\gb(s)$ vanishes.

Notice also that $\frac{d}{ds}D_\gb(1/q) <0$ provided that $\gb
< q/2$ and hence $D_\gb(1/q + \xi) < 0$ for all $\xi \leq \xi_0$,
where $\xi_0$ is a sufficiently small positive number. We are now
ready to derive the equivalence stated in the proposition. Observing
that $D_\gb(s)$ is a smooth function and $D_\gb(1) < 0$, we
deduce that $\eqref{item-i} \Rightarrow \eqref{item-ii} \Rightarrow
\eqref{item-iii}$. It remains to prove that $\eqref{item-iii}
\Rightarrow \eqref{item-i}$. Suppose now that \eqref{item-iii} holds
and there exists $s_0 \in ( 1/q,1]$ such that $D_\gb(s_0)\geq 0$.
Recalling that $D_\gb(1/q + \xi)<0$ and $D_\gb(1) < 0$, we
deduce the following:
\begin{itemize}
\item If $s_0 < s^*$, we will then have at least two zeros in $(1/q,
s^*)$ for $\frac{d}{ds}D_\gb(s)$.
\item If $s_0 > s^*$, we will then have at least one zero in $(s^*,
1)$ for $\frac{d}{ds}D_\gb(s)$.
\end{itemize}
We see that the first case contradicts with the fact that there can
be at most two zeros for $\frac{d}{ds}D_\gb (s)$ and the second case
contradicts our definition of $s^*$. Altogether, we established
that $\eqref{item-iii} \Rightarrow \eqref{item-i}$.

As for the last part, continuity and part \eqref{item:DOfBetaProp1:1} imply
that \eqref{eqn:BDDef} is equivalent to
\[
	\gb_s = \sup \{\gb \geq 0 : \: D_\gb(s) < 0 \mbox{ for all } s\in (1/q, 1]\}\,.
\]
Since $D_\gb(s)$ is increasing in $\gb$ for
all $s\in [1/q, 1]$, it follows that  for all $\gb < \gb_s(q)$, we
indeed have $D_\gb(s) < 0$ for all $s\in (1/q, 1]$. On the other
hand, by the continuity of the function $D_\gb(s)$ and our
definition of $\gb_s$, we know that there exists $s_M\in (1/q, 1]$
such that $D_{\gb_s}(s_M) =0$. Now, using the result of part
\eqref{item:DOfBetaProp1:1}
we conclude that $D_{\gb}(s_M) >0$ for any $\gb > \gb_s$,
completing the proof.
\end{proof}

In the following proposition we discuss the relation between $\gb_s(q)$ and $\gb_c(q)$.
\begin{prop}\label{l:betaMlessThanBetaC}
For $q \geq 3$ we have that $0 < \gb_s(q) < \gb_c(q) < q$ while $\gb_s(2) = \gb_c(2)=1$.
\end{prop}

\begin{proof}
As recalled in the Subsection~\ref{sub:CWLDP},
\[
\gb_c(2) = 1
	\qquad \text{and} \qquad
\gb_c(q)=\frac{(q-1)\log(q-1)}{q-2} \, ;\; q \geq 3 .
\]
In the $q=2$ case, it is easy to verify that  $\frac{d}{ds} D_{\gb}(s) < 0$ for all $s \in [\frac12, 1]$ if $\gb < 1$ and
$\frac{d}{ds} D_{\gb}(\frac12) > 0$ if $\gb > 1$. Since in addition $D_{\gb}(\frac12) = 0$, $D_{\gb}(1) < 0$, we obtain $\gb_s(2) = 1$.

For $q \geq 3$ we have $\gb_c(q) < q/2$ and therefore
\[
D_{\gb_c}\lb\frac{q-1}{q}\rb = - \frac{q-1}{q} + \frac{1}{1 + (q-1) \exp\big(2\gb_c \frac{1 - q\frac{q-1}{q}}{q-1}\big)} =0
\]
and
\begin{align*}
\left. \frac{d}{ds} D_{\gb_c}(s) \right|_{s=\frac{q-1}{q}} &= -1+ \frac{2\gb_c q \exp\big(2\gb_c \frac{1 - q\frac{q-1}{q}}{q-1}\big)}{\lb 1 + (q-1) \exp\big(2\gb_c \frac{1 - q\frac{q-1}{q}}{q-1}\big)\rb^2}\\
&=\frac{2(q-1)\log(q-1)-q(q-2)}{q(q-2)}\,.
\end{align*}
Now if $\phi(q)=2(q-1)\log(q-1)-q(q-2)$ then $\phi(2)=\phi'(2)=0$ and $\phi''(s)=-2+\frac2{q-1}$ which is negative when $q>2$.  It follows that $\phi(q)$ is negative when $q>2$ and hence for $q\geq 3$,
\[
\left. \frac{d}{ds} D_{\gb_c}(s) \right|_{s=\frac{q-1}{q}}<0.
\]
So for small enough $\epsilon>0$ and $s\in(\frac{q-1}{q}-\epsilon,\frac{q-1}{q})$ we have that $D_{\gb_c}(s)>0$ and hence $\sup_s  D_{\gb_c}(s)>0$.  By the smoothness of $D_\gb(s)$ this implies that there exists $\gb<\gb_c$ such that $\sup_s  D_{\gb}(s)>0$ which establishes that $\gb_s<\gb_c$.
\end{proof}

We will make repeated use of the the following proposition throughout the paper.
\begin{prop}
\label{clm:StayInRho}
\forcenewline
\begin{enumerate}
\item
	\label{item:ClmStayInRho_1}
	Assume $\gb < q/2$. For all $0 < \rho_0 < \rho$ small enough,
	there exists $\gga>0$ and $C, c>0$ such that for all $n$ with
	$t = e^{\gga n}$ we have
	\begin{equation}
	\label{eqn:ClmStayInRho_1}
		\bbP_{\gs_0} \lb \exists 0\leq s \leq t : \: \gs_s \notin \gS_n^{\rho+}
			\rb \leq C e^{-c n}
	\end{equation}
	for all $\gs_0 \in \gS_n^{\rho_0^+}$.
\item
	\label{item:ClmStayInRho_2}
   Assume $\gb < q/2$.
	For all $r_0 > 0$, $\gga > 0$ there exists $C,c>0$ such that
	for all $n$ and $r > r_0$
	with $t = \gga n$, $\rho_0 = \frac{r_0}{\sqrt{n}}$ and
	$\rho = \frac{r}{\sqrt{n}}$
	we have
	\begin{equation}
	\label{eqn:ClmStayInRho_2}
		\bbP_{\gs_0} \lb \exists 0 \leq s \leq t : \: \gs_s \notin 		
			\gS_n^{\rho+}  \rb \leq C e^{-c r^2}
	\end{equation}
	for all $\gs_0 \in \gS_n^{\rho_0+}$.
\item
	\label{item:ClmStayInRho_3}
	Assume $\gb < \gb_s(q)$.
	For all $\rho>0$ there exists $\gga > 0$ and $C, c>0$, such that
	for all $n$, with $t = \gga n$ we have
	\begin{equation}
	\label{eqn:ClmStayInRho_3}
		\bbP_{\gs_0} \lb \gs_t \notin \gS_n^{\rho+} \rb \leq C e^{-c n}
	\end{equation}
	for all $\gs_0 \in \gS_n$.
\end{enumerate}
\end{prop}
\begin{proof}
Consider the process $(S^1_t - s^1_0 : \: t \geq 0)$ until the first time it is to the right of $1/q + \rho$. If $n$ is large enough and, in Case~\eqref{item:ClmStayInRho_1}, if $\rho$ is small enough, Proposition~\ref{prop:DOfBetaProperties1} Part~\eqref{item:DOfBetaProp1:1} and Eq.~\eqref{eqn:OneCoordinateDrift}
imply that this process satisfies the conditions of Lemma \ref{lem:ImprovedAzuma}.
Parts~\eqref{item:ClmStayInRho_1} and~\eqref{item:ClmStayInRho_2} of the proposition then follow from Parts~\eqref{item:ImprovedAzuma3} and~\eqref{item:ImprovedAzuma2} of the lemma and summing over all coordinates.

As for part \eqref{item:ImprovedAzuma3}, consider this time
$(S^1_t - 1/q+\rho/2 :\: t \geq 0)$. Since $\gb < \gb_s(q)$, it follows from
Proposition \ref{prop:DOfBetaProperties1} (Parts~\eqref{item:DOfBetaProp1:1} and~\eqref{item:DOfBetaProp1:3}) that this process also satisfies the conditions of Lemma \ref{lem:ImprovedAzuma}, with $\gd > C' n^{-1}$, for some positive constant $C'$. Now set $\gga = 2/C'$ and apply 
part \eqref{item:ImprovedAzuma1} of the Lemma to conclude that except with probability exponentially small in $n$, $S^1_t \leq 1/q+\rho/2$ for some $t \leq \gga n$. Once this happens, by Lemma~\ref{lem:ImprovedAzuma} Part~\eqref{item:ImprovedAzuma3}, as in the proof of
\eqref{eqn:ClmStayInRho_1}, we have $S^1_{\gga n} \leq 1/q+\rho$ again except with probability tending exponentially fast to zero with $n$. It remains to use union bound to complete the proof.
\end{proof}

\subsection{Bounded Dynamics}
The {\it bounded dynamics} is a process that evolves like $\gs_t$, only that $S(\gs_t)$ is forced to stay close to $\vq$ by rejecting
transitions which violate this condition.
Formally, fix $\rho>0$ and let $\lb \gs_t \rb_{t \geq 0}$ be a
Markov chain on $\gS_n^{\rho+}$, which evolves as follows. Start from some
$\gs_0 \in \gS_n^{\rho+}$ and at step $t+1$:
\begin{itemize}
\item
	Draw $\wt{\gs}_{t+1}$ according to $P_n(\gs_t, \cdot)$, where $P_n$ is the original transition kernel.
\item
 If $\wt{\gs}_{t+1} \in \gS_n^{\rho+}$ set $\gs_{t+1} = \wt{\gs}_{t+1}$ and
	otherwise set $\gs_{t+1} = \gs_t$.
\end{itemize}
We shall denote by $\bbP^{\rho}$ the underlying probability measure.

The unbounded and $\rho$-bounded dynamics admit a natural coupling, under which the two processes start from the same configuration and evolve together until time $\tau = \inf \{t \geq 0 :\: S_t \notin \cS^{\rho+}\}$, where $S_t$ is the unbounded process. This leads to the following two immediate observations which will be useful later.
\begin{enumerate}
\item
For any integer $t$ and bounded function $f:(\gS_n)^{t+1} \mapsto \bbR$:
\begin{equation}
\label{eqn:BoundedUnBoundedProp1}
|\bbE f(\gs_{[0,t]}) - \bbE^{\rho} f(\gs_{[0,t]})| \leq 2 \normsup{f} \bbP(\tau \leq t).
\end{equation}
\item
In particular for any set $A \subseteq (\gS_n)^{t+1}$:
\begin{equation}
\label{eqn:BoundedUnBoundedProp2}
|\bbP (\gs_{[0,t]} \in A) - \bbP^{\rho} (\gs_{[0,t]} \in A)| \leq 2 \bbP(\tau \leq t)
\end{equation}
\end{enumerate}

\subsection{Synchronized Coupling}
\label{sub:SynchronizedCoupling}
The {\it synchronized coupling} is a (Markov) coupling of two $\rho$-bounded dynamics in which the two chains ``synchronize'' their steps as much as possible.
Formally, define $(\gs_t)_{t\geq 0}$, $(\wt{\gs}_t)_{t\geq 0}$ on the same probability space such that starting from $\gs_0$, $\wt{\gs}_0$, at time $t+1$:
\begin{enumerate}
	\item Choose colors ${I_{t+1}}$, $\wt{I}_{t+1}$ according to an optimal coupling of $S_t$, $\wt{S}_t$.
	\item Choose colors ${J_{t+1}}$, $\wt{J}_{t+1}$, according to an optimal coupling of $g_{\gb} \lb S_t - n^{-1} \sfe_{I_{t+1}} \rb$,
						$g_{\gb} \lb \wt{S}_t - n^{-1} \sfe_{\wt{I}_{t+1}}\rb$.
	\item Change a uniformly chosen vertex of color $I_{t+1}$ in $\gs_t$
		  to have color $J_{t+1}$ in $\gs_{t+1}$, but only if $\gs_{t+1} \in \gS_n^{\rho+}$.
	\item Change a uniformly chosen vertex of color $\wt{I}_{t+1}$ in 	
				$\wt{\gs}_t$ to have color $\wt{J}_{t+1}$ in $\wt{\gs}_{t+1}$,
                    but only if $\wt{\gs}_{t+1} \in \gS_n^{\rho+}$.
\end{enumerate}
We shall write $\bbP^{SC,\rho}_{\gs_0, \wt{\gs}_0}$ for the underlying measure and omit $\rho$ if it is large enough for the dynamics not to be bounded.

\smallskip
The following shows that this coupling contracts the $\normI{\cdot}$ distance of the proportions vector.
\begin{lem}
\label{lem:SyncCouplingCoalesence}
There exists $C(\gb,q)>0$ such that for any $\rho > 0$, uniformly in $\gs_0, \wt{\gs}_0 \in \gS_n^{\rho+}$
as $n \to \infty$
\begin{equation}
\label{eqn:SyncCouplingCoalesence}
\bbE^{SC, \rho}_{\gs_0, \wt{\gs}_0} \normI{S_t - \wt{S}_t} \leq
    \lb  p + \frac{C\rho}{n}  \rb^t \normI{s_0 - \wt{s}_0}\,,
\end{equation}
where
\begin{equation}
\label{eqn:RhoDefinition}
p = p(n, \gb, q) = 1 - \frac{1-2\gb/q}{n}\,.
\end{equation}
\end{lem}

\begin{proof}
For $s, \wt{s} \in \cS_n$
by a Taylor expansion of $g_{\gb}$ around $s$, then another expansion for $\grad g_{\gb}$ around $\vq$
one has:
\[
\normIs{g_{\gb}(s) - g_{\gb}(\wt{s})} = \frac{2\gb}{q} \normI{s - \wt{s}} \lb 1 + O \lb \normI{s - \vq} + \normI{\wt{s} - \vq} \rb \rb
\,,\]
where we use the easily verified:
\begin{eqnarray}
\partial_j g_{\gb}^k = \lbr
    \begin{array}{ll}
        -2\gb g_{\gb}^j g_{\gb}^k                        & k \neq j \,,\\
        -2\gb \lb g_{\gb}^k \rb^2 + 2\gb g_{\gb}^k        & k = j \,.
    \end{array}
\right.
\end{eqnarray}

Now under the bounded dynamics, $I_{t+1} \neq \wt{I}_{t+1}$ implies that
$S_t^{I_{t+1}} > \wt{S}_t^{I_{t+1}}$ and $\wt{S}_t^{\wt{I}_{t+1}} > S_t^{\wt{I}_{t+1}}$ while  $J_{t+1} \neq \wt{J}_{t+1}$ implies that $ \lb S_t - n^{-1} \sfe_{I_{t+1}} \rb^{J_{t+1}} >  \lb \wt{S}_t - n^{-1} \sfe_{\wt{I}_{t+1}} \rb^{J_{t+1}}$ and $\lb \wt{S}_t - n^{-1} \sfe_{\wt{I}_{t+1}} \rb^{\wt{J}_{t+1}} >  \lb S_t - n^{-1} \sfe_{I_{t+1}} \rb^{\wt{J}_{t+1}}$.  It follows by the definition of the coupling that
\[
\normI{S_{t+1} - \wt{S}_{t+1}} -\normI{S_t - \wt{S}_t}  = -\frac2n \left[ \mathbb{I}_{\{ I_{t+1} \neq \wt{I}_{t+1} \}} -  \mathbb{I}_{\{ J_{t+1} \neq \wt{J}_{t+1} \}}\right].
\]
Recalling that for $s \in \cS_n$, $\normTV{s} = \inv{2} \normI{s}$ and
that under the best coupling of distributions $s$, $\wt{s}$ the
probability of disagreement is $\normTV{s-\wt{s}}$, we have:
\begin{eqnarray*}
\lefteqn{\bbE^{SC, \rho}_{\gs_0, \wt{\gs}_0} \normI{S_1 - \wt{S}_1}} \\
    & \leq  &
        \normI{s_0 - \wt{s}_0}
        + \frac{2}{n} \lb -\inv{2}\normI{s_0 - \wt{s}_0} + \inv{2} \normIs{g_{\gb}(s_0) -
                g_{\gb}(\wt{s}_0)} + O(n^{-1}) \rb   \\
    & \leq  &
        \normI{s_0 - \wt{s}_0} \lb
            1 - \frac{1 - 2\gb/q + O(\rho)}{n} + O(n^{-2}) \rb\,.
\end{eqnarray*}
The result follows by iteration.
\end{proof}

\subsection{Uniform Variance Bound}

\begin{lem}
\label{lem:UniformVarBound}
Assume $\gb < q/2$. There exists $\rho_0 = \rho_0 (\gb, q)$ such that if $\rho \leq \rho_0$
\begin{equation}
\label{eqn:UniformVarBound}
\Var_{\gs_0}^{\rho} \lb S_t \rb = O \lb n^{-1} \rb\,,
\end{equation}
uniformly in $\gs_0 \in \gS_n^{\rho+}$ and $t \geq 0$, and there exists
$\gga_0 > 0$ such that
\begin{equation}
\label{eqn:AlmostUniformVarBound}
\Var_{\gs_0} \lb S_t \rb = O \lb n^{-1} \rb\,,
\end{equation}
uniformly in $\gs_0 \in \gS_n^{\rho_0+}$ and
$t \leq e^{\gga_0 n}$.
\end{lem}
\begin{proof}
Equation \eqref{eqn:UniformVarBound} will follow directly from Lemma~\ref{lem:VarBoundFromContraction} applied to $S_t$ under the $\rho$-bounded dynamics. Indeed, Lemma~\ref{lem:SyncCouplingCoalesence} gives a stronger version of Condition~\ref{eqn:VarBoundContractionCond} with
$\eta = p + \rho O(n^{-2})$. Now, if $\gb < q/2$ and $\rho$ is small enough, we have $\eta \leq 1 - \frac{1 - \frac{2\gb}{q}}{2n}$ for large enough $n$. Then \eqref{eqn:UniformVarBound} follows from
\eqref{eqn:VarBoundContractionResult} since $\Var^{\rho}_{\gs_0} S_1 = O \lb n^{-2} \rb$.

For \eqref{eqn:AlmostUniformVarBound}, find $\rho^{\prime} < \rho$ and use
\eqref{eqn:ClmStayInRho_1} and \eqref{eqn:BoundedUnBoundedProp1} to conclude
that for all $\gs_0 \in \gS_n^{\rho^{\prime}+}$ and $t \leq e^{\gga n}$ for some $\gga = \gga(\rho, \rho^{\prime})$:
\[
\Var_{\gs_0} \lb S_t \rb = \Var^{\rho}_{\gs_0} \lb S_t \rb + o(n^{-1})
 = O \lb n^{-1} \rb\,. \qedhere
\]
\end{proof}
\begin{cor}
\label{cor:VarOFMuN}
For $\gb < \gb_c(q)$, we have $\Var_{\mu_n} (S) = O(n^{-1})$.
\end{cor}
\begin{proof}
Fix $\rho < \rho_0$, where $\rho_0$ is given in Lemma \ref{lem:UniformVarBound}
and notice that the bounded dynamics is reversible with respect to the Potts measure $\mu_n$ restricted to $\gS_n^{\rho+}$. Therefore the bounded dynamics has
$\mu^{\rho}_n(\cdot) = \mu_n(\cdot | \gs \in \gS_n^{\rho+})$ as its stationary measure. From the large deviation analysis in Subsection~\ref{sub:CWLDP} it is straightforward to conclude that if $\gb < \gb_c(q)$
\[\mu_n(S \not \in \cS^{\rho+}) \leq \mathrm{e}^{-C n}\,,\]
for some $C > 0$ and $n$ large enough, depending on $\rho$ and $\gb$. Therefore, we have $\normTV{\mu_n - \mu^{\rho}_n} \leq \mathrm{e}^{-Cn}$. Since $\mu_n$, $\mu^{\rho}_n$ live on a compact space, this gives $\Var_{\mu_n}(S) \leq \Var_{\mu^{\rho}_n}(S) + \mathrm{e}^{-Cn}$. Since  $\bbP^{\rho}_{\gs_0} (\gs_t \in \cdot)$ converges to $\mu_n^{\rho}$ as $t \to \infty$ for any fixed $\gs_0 \in \gS^{\rho+}_n$, Lemma \ref{lem:UniformVarBound} can be extended to $\gs_0$ chosen from $\mu_n^{\rho}$. This completes the proof.
\end{proof}

\subsection{The Drift of the Distance to Equiproportionality}
Here we show that $\whS_t$ has drift towards 0.  Write $S_{t+1} = S_t + \xi_{t+1}$ where we have that for $i , \,j =1,\,...,\,q$,
\begin{align*}
\bbP\left(\xi_{t+1} = \tfrac{1}{n} \lb \sfe_j - \sfe_i \rb \right)
    = S_t^i g_{\gb}^j \lb S_t - \tfrac{1}{n} \sfe_i \rb =  S_t^i g_{\gb}^j \lb S_t \rb + O \lb n^{-1} \rb\,.
\end{align*}
Then,
\begin{eqnarray*}
\bbE \lsb \normII{S_{t+1} - \vq}^2  \,| \, S_t \rsb
    & =     &   \bbE \lsb \normII{S_t - \vq + \xi_{t+1}}^2  \,| \, S_t \rsb \\
    & =     &   \normII{S_t - \vq}^2 + \bbE \lsb \normII{\xi_{t+1}}^2  | \, S_t \rsb +
                2 \lan \bbE \lsb \xi_{t+1} | \, S_t \rsb , S_t \ran \\
    & =     & \normII{\whS_t}^2 + \frac{2}{n^2} \lb 1-h(S_t) \rb +
                \frac{2}{n} \sum_{j=1}^q \frac{e^{2\gb S_t^j}}{\sum_{k=1}^q e^{2\gb S_t^k}} S_t^j -
                \frac{2}{n} \normII{S_t}^2 + O \lb n^{-2} \rb \\
    & =     & \normII{\whS_t}^2 \lb 1-\tfrac{2}{n} \rb + \lb 2h(S_t) - \tfrac{2}{q} \rb n^{-1} +
                \lb 2-2h(S_t) \rb n^{-2} + O \lb n^{-2} \rb\,,
\end{eqnarray*}
where $h(s) \bydef \sum_{k=1}^q g_{\gb}^k(s) s^k$.
Notice that since $g^k_\beta(s)=g^k_\beta(s-\vq)$ we have that $h(s) = 1/q + h(\whs)$ and its gradient and Hessian are:
\[
\rmD_1 h(0) = \frac{1}{q} \quad \mbox{ and }\qquad \rmD_2 h(0) = \frac{4 \gb}{q} P\,.
\]
where $P$ is a projection matrix onto $\lb \vq \rb^{\perp}$.
Therefore, we may write
\[
h(s) = \frac{1}{q} + \frac{2\gb}{q} \normII{\whs}^2 + O \lb \normII{\whs}^3 \rb\,.
\]
This gives
\begin{eqnarray}
\nonumber
\bbE \lsb \normII{\whS_{t+1}}^2 | \cF_t \rsb
    & =     & \normII{\whS_t}^2 \lb 1 - \frac{2 \lb 1 - 2\gb/q \rb}{n} \rb +
                n^{-1} O \lb \normII{\whS_t}^3\rb + O \lb n^{-2} \rb \\
\label{eqn:RecursionForSHat}
    & =     & \normII{\whS_t}^2 p^2 +
                n^{-1} O \lb \normII{\whS_t}^3\rb + O \lb n^{-2} \rb\,.
\end{eqnarray}
where $p$ is defined in \eqref{eqn:RhoDefinition}.

\subsection{Contraction for the Distance to Equiproportionality}
\label{sub:DistanceContraction}
Fix $\gb < q/2$ and $\rho < \rho_0$ where $\rho_0$ is given in Lemma~\ref{lem:UniformVarBound}. For what follows, assume that $\gs_0 \in \gS^{\rho+}_n$ and $t \leq e^{\gga_0 n}$ where $\gga_0$ is also given.
Then, taking expectation in equation~\eqref{eqn:RecursionForSHat},
we get:
\begin{equation}
\label{eqn:RecursionForMeanOfBoundedSHat}
\bbE_{\gs_0} \normII{\wh{S}_{t+1}} =
	p^2 \bbE_{\gs_0} \normII{\wh{S}_t}^2 +
   \lb \bbE_{\gs_0} \normII{\wh{S}_t}^3 \rb O \lb n^{-1} \rb +                     	
	O \lb n^{-2} \rb\,.
\end{equation}
Now by Taylor expansion of $s \mapsto \normII{s}^3$ around $\bbE_{\gs_0} \wh{S}_t$
in view of \eqref{eqn:AlmostUniformVarBound},
\begin{eqnarray*}
\bbE_{\gs_0} \normII{\wh{S}_t}^3
    & =     &   \normII{\bbE_{\gs_0} \wh{S}_t}^3 + \bbE_{\gs_0} \lan
                    \rmD_1 \normII{\cdot}^3 \lb \bbE_{\gs_0} \wh{S}_t \rb\,,\,
                    \wh{S}_t - \bbE_{\gs_0} \wh{S}_t
                \ran + O \lb \bbE_{\gs_0} \normII{\wh{S}_t - \bbE_{\gs_0} \wh{S}_t}^2 \rb   \\
    & =     &   \normII{\bbE_{\gs_0} \wh{S}_t}^3 + O \lb n^{-1} \rb
    = \lb \bbE_{\gs_0} \normII{\wh{S}_t}^2 + O \lb n^{-1} \rb \rb^{3/2} + O \lb n^{-1} \rb \\
    & =     & \lb \bbE_{\gs_0} \normII{\wh{S}_t}^2 \rb^{3/2} + O \lb n^{-1} \rb \,.
\end{eqnarray*}
Then
\begin{eqnarray}
\label{eqn:RecursionForMeanOfBoundedSHat_NiceForm}
\bbE_{\gs_0}  \lsb \normII{\wh{S}_{t+1}}^2 \rsb =
    p^2 \bbE_{\gs_0} \normII{\wh{S}_t}^2 +
    \lb \bbE_{\gs_0} \normII{\wh{S}_t}^2 \rb^{3/2} O \lb n^{-1} \rb +
    O \lb n^{-2} \rb\,.
\end{eqnarray}
This will in turn imply:
\begin{prop}
\label{prop:ContractionForNorm}
Fix $\gb < q/2$. There exist $\rho_0 = \rho_0(\gb, q)>0$ and $C=C(\gb, q) > 0$ such that
if $\rho \leq \rho_0$ there exists $\gga(\rho) > 0$
such that:
\begin{eqnarray}
\label{eqn:ExplicitContractionForSHat}
\bbE_{\gs_0} \normII{\wh{S}_t}^2
    = p^{2 t} (\normII{\wh{s}_0}^2 + C\rho^3)+
        O \lb n^{-1} \rb,
\end{eqnarray}
uniformly in $\gs_0 \in \gS_n^{\rho+}$ and $t \leq e^{\gga(\rho) n}$,
where $p = p(n, \gb, q)$ is defined in \ref{eqn:RhoDefinition}.
\end{prop}

\begin{proof}
Set $\gl_t \bydef \bbE_{\gs_0} \normII{\wh{S}_t}^2$. It follows from
\eqref{eqn:RecursionForMeanOfBoundedSHat_NiceForm} and
\eqref{eqn:ClmStayInRho_1} that
for any $\rho < \ol{\rho} < \rho_0$, where $\rho_0$ is given in Lemma
\ref{lem:UniformVarBound}, there exists $\gga = \gga(\rho, \ol{\rho}) > 0$ such that uniformly in $\gs_0 \in \gS_n^{\rho+}$ and $t \leq e^{\gga n}$:
 \[
\gl_{t+1} \leq \gl_t \lb p^2 + \ol{\rho} O \lb n^{-1} \rb \rb + O \lb n^{-2} \rb
	= \gl_t \lb p + \ol{\rho} O \lb n^{-1} \rb \rb^2 + O \lb n^{-2} \rb\,.
\]

We now use the following fact, which can be easily verified.
If $\lb \gl_t \rb_{t \geq 0}$ is a sequence
satisfying:
\[
\gl_{t+1} = p \gl_t + a r^t + b,
\]
for some $p \neq r$, $p \neq 1$, $a$ and $b$, then:
\begin{equation}
\label{eqn:RecursionSolution_SHat}
\gl_{t} = \gl_0 p^t + a \frac{p^t - r^t}{p-r} + \frac{b}{1-p} \lb 1 - p^t \rb \,.
\end{equation}
Apply this (with $a=0$) and use the monotonicity in $\gl_t$ of the
right hand side above (at least if $n$ is large enough), to conclude:
\[
\gl_t \leq \gl_0 \lb p + \ol{\rho} O(n^{-1}) \rb^{2t} + \inv{1-\lb p + \ol{\rho} O(n^{-1}) \rb^2} O \lb n^{-2} \rb	
	\leq C \ol{\rho}^2 (p + \ol{\rho} O(n^{-1}))^{2t} + O \lb n^{-1} \rb\,.
\]
Plugging this a priori bound back into
\eqref{eqn:RecursionForMeanOfBoundedSHat_NiceForm} we see that:
\[
\gl_{t+1} = p^2 \gl_t + \ol{\rho}^3 \lb p + \ol{\rho} O(n^{-1}) \rb^{3t} O \lb n^{-1} \rb + O \lb n^{-2} \rb\,.
\]
Using \eqref{eqn:RecursionSolution_SHat} again
and choose $\ol{\rho}$ small enough to obtain
\begin{eqnarray*}
\gl_t
    & =     & \gl_0 p^{2t} +
	\frac{p^{2t} - \lb p + \ol{\rho} O(n^{-1}) \rb^{3t}}
		{p^2-\lb p + \ol{\rho} O(n^{-1}) \rb^3} \ol{\rho}^3 O \lb n^{-1} \rb +
                    \inv{1-p^2} O \lb n^{-2} \rb     \\
    & =     & (\gl_0 + O(\ol{\rho}^3)) p^{2t} + O \lb n^{-1} \rb\,.
\end{eqnarray*}
as desired.
\end{proof}

\section{Mixing in the Subcritical Regime}
\label{sec:SubcriticalRegime}
In this section we prove Theorem \ref{thm:SubCriticalCutoff}. Recall that
$\ga_1 = \ga_1(\gb, q) = \frac{1}{2 \lb 1 - 2\gb/q \rb}$ and set
\begin{equation}
\label{eqn:TOfGammaDef}
t^{\ga_1}(n) = \ga_1 n \log n
\quad \text{;} \quad
t^{\ga_1}_{\gga}(n) = \ga_1 n \log n + \gga n.
\end{equation}

\subsection{Proof of Lower Bound in Theorem \ref{thm:SubCriticalCutoff}}
\begin{proof}
\label{sub:ProofOfLowerBound}
The analysis in this subsection pertains to all $\gb < \gb_c(q)$.
Fix $0 < \rho_2 < \rho_1 < \rho_0$, where $\rho_0$ is given in Proposition~\ref{prop:ContractionForNorm} and let $\gs_0 \in \gS_n$ be such that
$\rho_2 < \normII{\wh{s}_0} < \rho_1$. Then if $t < t^{\ga_1}_{\gga}(n)$ and $\rho_1$ is small enough, Proposition \ref{prop:ContractionForNorm} implies
\begin{align*}
\bbE_{\gs_0} \normII{\wh{S}_t}^2
	 \geq     \frac{\rho_2^2}{2}
					\big( 1 - \tfrac{1-2\gb/q}{n}\big)^{2 t^{\ga_1}_{\gamma}(n)} +
					O \lb n^{-1} \rb  \geq     \frac{1}{n} e^{-(1 - 2\gb/q)\gamma}\,,
\end{align*}
for sufficiently large $-\gamma$ depending on $\rho_2$ and large enough $n$. Combined with the uniform variance bound given in Lemma \ref{lem:UniformVarBound}, it follows that for large enough $n$
\[\bbE_{\gs_0} \normII{\wh{S}_{t}} \geq \frac{\mathrm{e}^{-(1 - 2\gb/q) \gamma /2}}{\sqrt{n}}\,.\]
Applying Chebyshev's inequality and using Lemma \ref{lem:UniformVarBound} again, we conclude that uniformly in all $r > 0$, $t \leq t^{\ga_1}_{\gga}(n)$ and $\gs_0 \in \gS_n^{\rho_1+} \setminus \gS_n^{\rho_2+}$
\begin{align}
\label{eqn:UpperBoundOnSqrtNCloseToEquilibrium}
\bbP_{\gs_0} \lb \normII{\wh{S}_{t}} < \frac{r}{\sqrt{n}} \rb
	& \leq\frac{\Var_{\gs_0} ( \wh{S}_{t} )} {\big(\frac{\mathrm{e}^{-(1 - 2\gb/q) \gamma /2}}{\sqrt{n}} - \frac{r}{\sqrt{n}}
					\big)^2} = O\Big((\mathrm{e}^{-(1 - 2\gb/q) \gamma /2} - r)^{-2}\Big)\,.
\end{align}
In particular, this implies
\begin{equation}
\label{eqn:UpperBoundOnSqrtNCloseToEquilibrium1}
	\lim_{\gga \to -\infty} \limsup_{n \to \infty} \bbP_{\gs_0} \lb \normII{\wh{S}_{t^{\ga_1}_{\gga}(n)}} < \frac{r}{\sqrt{n}} \rb = 0 .
\end{equation}

On the other hand,
$\bbE_{\mu_n} S_t = \vq$ and from Corollary \ref{cor:VarOFMuN} it follows that $\Var_{\mu_n} S_t = O \lb n^{-1} \rb$ for $\gb < \gb_c(q)$.
Therefore another application of  Chebyshev's inequality yields that
\begin{equation}
\label{eqn:StationaryMeasureConcentration}
\mu_n \lb \normII{\wh{s}_t} < \frac{r}{\sqrt{n}} \rb \geq 1 - \frac{O(1)}{r^2}\,,
\end{equation}
for all $t \geq 0$. Altogether, we have that for any $r > 0$,
\[
\lim_{\gga \to -\infty} \liminf_{n \to \infty} d_{t^{\ga_1}_{\gga}(n)}(n) \geq 1 - \frac{O(1)}{r^2}\
\]
and it remains to send $r \to \infty$.
\end{proof}

In the remainder of the section we prove the upper bound on the mixing time when $\gb <\gb_s(q)$. The proof is based on upper bounding the coalescence time of two coupled dynamics, one starting from any configuration in $\gS_n$ and the other starting from the stationary distribution $\mu_n$. This coupling will be done in several stages with different couplings from one stage to the next. In what follows, $(\gs_t)_{t \geq 0}$ and $(\wt{\gs})_{t \geq 0}$ will denote the two coupled processes.

\subsection{\texorpdfstring{$O(n^{-1/2})$ from Coalescence}{O(1/n\^(\textonehalf)) from Coalescence}}
\label{sub:OSqrtFromCoalesence}
We now show that with arbitrarily high probability, $S_t$ gets $O(n^{-1/2})$-close to $\vq$ in $O(n \log n)$ steps, if initially its distance is at most $\rho$, where $\rho$ is small enough. More precisely,
\begin{lem}
\label{lem:OSqrtFromCoalesence}
Fix $\gb < q/2$. Then for all $r > 0$:
\[
\bbP_{\gs_0} \lb S_{t^{\ga_1}(n)} \notin \cS^{\frac{r}{\sqrt{n}}} \rb
	= O(r^{-1})\,,
\]
uniformly in $\gs_0 \in \gS_n^{\rho_0+}$ where $\rho_0 = \rho_0(\gb, q)$ is defined in Proposition \ref{prop:ContractionForNorm} and $t^{\ga_1}(n)$ is defined in \eqref{eqn:TOfGammaDef}.
\end{lem}
\begin{proof}
This follows immediately from Proposition~\ref{prop:ContractionForNorm} and a first moment argument:
\begin{align*}
\bbP_{\gs_0} \big( S_{t^{\ga_1}(n)} \notin \cS^{\frac{r}{\sqrt{n}}}\big) &\leq
\bbP_{\gs_0} \big( \normII{\wh{S}_{t^{\ga_1}(n)}} \geq r n^{-\inv{2}}\big) \\
&\leq
\frac{\bbE_{\gs_0} \normII{\wh{S}_{t^{\ga_1}(n)}}}{r n^{-1/2}}\leq \frac{\big(\bbE_{\gs_0}
\normII{\wh{S}_{t^{\ga_1}(n)}}^2\big)^{1/2}}{r n^{-1/2}} = O\big(\tfrac{1}{r}\big) \,.\qedhere
\end{align*}
\end{proof}

\subsection{\texorpdfstring{$O(n^{-1})$ from Coalescence}{O(1/n) from Coalescence}}

To get the correct order of the mixing time it is not sufficient to simply use the drift to couple the chains as the drift is very weak when $S_t$ is close to $\vq$.  As such, in this section we define a different coupling of the dynamics which will bring $\gs_t$ and $\wt{\gs}_t$ to distance $O(n^{-1})$ apart in linear time. This will be achieved one coordinate after the other. We begin by giving a general definition of, what we call, a semi-independent coupling and then use it to define the coupling of the dynamics.

Let $\nu$, $\wt{\nu}$ be two positive distributions on $\gO_m = [1,m]$ and fix  a non-empty $A \subseteq \gO_m$, where $m$ is some positive integer. We shall write $\nu|_A$ for the conditional distribution given $A$, i.e.\
$\nu|_A(x) = \nu(x)/\nu(A)$, for $x \in A$.
The {\it $A$-semi-independent coupling} of $\nu$, $\wt{\nu}$ is a coupling of two random variables $X$ and $\wt{X}$ with underlying measure
$\Pst$, constructed according to the following procedure:
\begin{enumerate}
 	\item
		Choose $U \in [0,1]$ uniformly.
	\item
		If $U \leq \min \lbr \nu(A),\, \wt{\nu}(A) \rbr$, draw $X$ and $\wt{X}$
		using a best coupling of $\lb \nu|_A,\, \wt{\nu}|_A \rb$.
	\item Otherwise, independently:
		\begin{enumerate}
			\item Draw $X$ according to $\nu|_A$  if $U < \nu(A)$ and
					according to $\nu|_{A^c}$ if $U \geq \nu(A)$.
			\item Draw $\wt{X}$ according to $\wt{\nu}|_A$ if $U < \wt{\nu}(A)$ and
					according to $\wt{\nu}|_{A^c}$ if $U \geq \wt{\nu}(A)$.
		\end{enumerate}
\end{enumerate}
Clearly a $\gO_m$-semi-independent coupling is a best coupling and
for $A = \mbox{\O{}}$, we define $\mbox{\O{}}$-semi-independent coupling to be
the standard independent coupling.
The following proposition states a few properties of this coupling, which will be useful for the sequel.
\begin{prop}
\label{prop:SICProps}
The following holds for the $A$-semi-independent coupling of $\lb \nu,\wt{\nu} \rb$:
\begin{enumerate}
	\item
		$X$, $\wt{X}$ are distributed according to $\nu$, $\wt{\nu}$ respectively.
	\item
		\label{item:SICProximityPreservation}
$		\Pst \lb \cup_{x \in A} \{ X = x \} \symdiff \{ \wt{X} = x \} \rb
			\leq \tfrac{3}{2} \sum_{x \in A} | \nu(x) - \wt{\nu}(x) |
$
	\item
		\label{item:SICSemiIndependence1}
$		\forall x \notin A,\  \Pst \lb X = x,\,\wt{X} \neq x \rb
			\geq \nu(x)\wt{\nu}(A^c \setminus \{x\})
$
		and \\
$		\forall x \notin A,\  \Pst \lb \wt{X} = x,\,X \neq x \rb
			\geq \wt{\nu}(x)\nu(A^c \setminus \{x\}).
$
\end{enumerate}
\end{prop}
\begin{proof}
Part one of the Lemma is immediate. Part \eqref{item:SICProximityPreservation}
follows from a straight forward calculation:
\begin{eqnarray*}
\Pst \lb \cup_{x \in A} \{ X = x \} \symdiff \{ \wt{X} = x \} \rb
	& \leq	& \Pst \lb U \leq \nu(A) \minwith \wt{\nu}(A) \rb \Pst \lb X \neq \wt{X} \labs
					U \leq \nu(A) \minwith \wt{\nu}(A) \rpr \rb + \\
	& 		& 	\quad \Pst \lb \nu(A) \minwith \wt{\nu}(A) < U < \nu(A) \maxwith \wt{\nu}(A) \rb \\
	& \leq 	& (\nu(A) \minwith \wt{\nu}(A)) \inv{2} \sum_{x \in A} \labs \lb \nu|_A(x) - \wt{\nu}|_A(x) \rb \rabs
		+ \labs \nu(A) - \wt{\nu}(A) \rabs	\\
	& \leq	& \tfrac{3}{2} \sum_{x \in A} \labs \nu(x) - \wt{\nu}(x) \rabs.
\end{eqnarray*}
As for part \eqref{item:SICSemiIndependence1}, we have:
\begin{eqnarray*}
\Pst \lb \{ X = x \} \setminus \{ \wt{X} = x \} \rb
	& \geq	& \Pst \lb U > \nu(A) \maxwith \wt{\nu}(A) \rb \nu|_{A^c}(x)(1-\wt{\nu}|_{A^c}(x)) \\
	& \geq 	& \nu(x) \wt{\nu}(A^c \setminus \{x\})
\end{eqnarray*}
and similarly for $\Pst(\wt{X}=x, X \neq x)$.
\end{proof}

We are now ready to define the coupling of $\gs_t$, $\wt{\gs}_t$ for this section. Fix $\gs_0, \wt{\gs}_0 \in \gS_n$ and $y_1, \dots y_{q-1} > 0$.
The {\it coordinate-wise coupling} with parameters $y_1, \dots, y_{q-1}$ and starting configurations $\gs_0$, $\wt{\gs}_0$ is defined as follows.
\begin{enumerate}
	\item Set $T^{(0)} = 0$, $k=1$.
		\label{item:CCStep1}
	\item As long as $k \leq q-1$:
		\begin{enumerate}						
		\item As long as $\labs S^k_t - \wt{S}^k_t \rabs > \frac{y_k}{n}$:
			\begin{enumerate}
				\item Draw ${I_{t+1}}$, $\wt{I}_{t+1}$, using
						a $\{1, \dots, k-1\}$-semi-independent coupling of
						$S_t$, $\wt{S}_t$.
				\item Draw ${J_{t+1}}$, $\wt{J}_{t+1}$, using
						a $\{1, \dots, k-1\}$-semi-independent coupling of 							
						$g_{\gb} \lb S_t - \frac{1}{n} \sfe_{I_{t+1}} \rb$,
						$g_{\gb} \lb \wt{S}_t - \frac{1}{n} \sfe_{\wt{I}_{t+1}} \rb$.
         	\item Change a uniformly chosen vertex of color $I_{t+1}$ in
                 $\gs_t$ to have color $J_{t+1}$ in $\gs_{t+1}$.
            \item Change a uniformly chosen vertex of color $\wt{I}_{t+1}$ in
			    $\wt{\gs}_t$ to have color $\wt{J}_{t+1}$ in $\wt{\gs}_{t+1}$.
				\item Set $t = t+1$.
			\end{enumerate}
		\item When $\labs S^k_t - \wt{S}^k_t \rabs \leq \frac{y_k}{n}$ set $T^{(k)} = t - T^{({k-1})}$ and $k = k+1$.
		\end{enumerate}
	\item Set $T^{CC} = \sum_{k=1}^{q-1} T^{(k)}$.
\end{enumerate}
We shall use $\bbP^{CC}_{\gs_0, \wt{\gs}_0}$ to denote the probability measure for this coupling and $\bbP^{CC(m)}_{\gs_0, \wt{\gs}_0}$ for the same coupling, only with $k=m$ instead of $k=1$ in step \eqref{item:CCStep1}, i.e.\ starting from the $m$-th stage. Notice that, in principle, the stopping condition at stage $k$, may never get satisfied,
in which case we stay at that stage forever and $T^{(k)} = T^{CC} = \infty$.

For $u,r > 0$, define
\[
\cH^k_{u,r} = \lbr (\gs,\wt{\gs}) \in \gS_n \times \gS_n : \:
	\normI{s^{[1,k]} - \wt{s}^{[1,k]}} < \frac{u}{n},	\\
	\max \lb \normII{s - \vq}, \normII{\wt{s} - \vq} \rb < \frac{r}{\sqrt{n}}
\rbr\,.
\]
where above $(s,\wt{s}) = (S(\gs), S(\wt{\gs}))$. Finally, set $\cH_{u,r} \bydef \cH_{u,r}^q$.
The following lemma will be the main ingredient in an inductive proof for an upper bound on $T^{CC}$:
\begin{lem}
\label{lem:SICBounds}
Fix $\gb < q/2$. Let $k \in [1,q-1]$. For all $u_{k-1}, r_{k-1}, \gep > 0$,
there exist $y_k, u_k, r_k, \gga_k > 0$, such that if
$(\gs_0, \wt{\gs}_0) \in \cH^k_{u_{k-1}, r_{k-1}}$
then
\begin{equation}
\label{eqn:SICLemResult}
\bbP^{CC(k)}_{\gs_0, \wt{\gs}_0} \lb T^{(k)} < \gga_k n ,\, (\gs_{T^{(k)}}, \wt{\gs}_{T^{(k)}}) \in \cH^k_{u_k, r_k} \rb
		\geq 1-\gep\,.
\end{equation}
\end{lem}

\begin{proof}
Recall the expression for the drift of one coordinate \eqref{eqn:OneCoordinateDrift}.
Near $\vq$, this becomes by Taylor expansion for any $i \in [1,q]$:
\[
\bbE^{CC(k)}_{\gs_0, \wt{\gs}_0} \lsb \lpr S^i_{t+1} - S^i_t \rabs \cF_t \rsb
 = 	\inv{n} \lbr - \lb 1 - \tfrac{2\gb}{q} \rb \wh{S}^i_t + O \lb \normII{\wh{S}_t}^2 \rb \rbr +
			O \lb n^{-2} \rb\,.
\]
and if $W_t = S_t - \wt{S}_t$ then
\begin{equation}
\label{eqn:OneCoordDiffDrift}
\bbE^{CC(k)}_{\gs_0, \wt{\gs}_0} [ W^k_{t+1} - W^k_t \mid S_t, \wt{S}_t ]	
	= \inv{n} \lbr - \lb 1 - \tfrac{2\gb}{q} \rb W^k_t +
				O \lb \normII{S_t - \vq}^2 + \normII{\wt{S}_t - \vq}^2 \rb \rbr +
				O \lb \tfrac{1}{n^2} \rb\,.
\end{equation}
Now, for some $r_k > 0$ to be chosen later,
let $\tau^{(k)} = \inf \lbr t :\: \normII{S_t - \vq} \maxwith \normII{\wt{S}_t - \vq} \geq \frac{r_k}{\sqrt{n}} \rbr$.
Then, from \eqref{eqn:OneCoordDiffDrift} it follows that there exists $y_k>0$ such that
$\ol{W}^k_t \bydef \labs W^k_{t \minwith \tau^{(k)} \minwith T^{(k)}} \rabs$
is a supermartingale.
Clearly $\labs \ol{W}^k_{t+1} - \ol{W}^k_{t} \rabs <
\frac{2}{n}$. Also, from Proposition \ref{prop:SICProps}, if $t < \tau^{(k)} \minwith T^{(k)}$:
\begin{eqnarray*}
	\bbP^{CC(k)}_{\gs_0, \wt{\gs}_0} \lb \lpr \ol{W}^k_{t+1} \neq \ol{W}^k_t \rabs \cF_t \rb
		& \geq	& \bbP^{CC(k)}_{\gs_0, \wt{\gs}_0} \lb \lpr {I_{t+1}} = k,\, {J_{t+1}} \neq k ,\, \wt{I}_{t+1} \neq k
					\rabs \cF_t \rb \\
		& \geq	& S^k_t \wt{S}^{k+1}_t \lb 1- g_{\gb}^k \lb S_t - \tfrac{1}{n} \sfe_k \rb \rb \\
		& = 	& \frac{q-1}{q^3} + O \lb n^{-\inv{2}} \rb\,,
\end{eqnarray*}
which implies that
$\bbE^{CC(k)}_{\gs_0, \wt{\gs}_0} \lsb \lpr \lb \ol{W}^k_{t+1} - \ol{W}^k_t \rb^2 \rabs \cF_t
 \rsb \geq \inv{2} q^{-2}n^{-2} + O \lb n^{-5/2} \rb$.
On the other hand, in view of~\eqref{eqn:OneCoordDiffDrift},
$\bbE^{CC(k)}_{\gs_0, \wt{\gs}_0} \lsb \lpr \ol{W}^k_{t+1} - \ol{W}^k_t \rabs \cF_t \rsb =
	O(n^{-3/2})$.
Combining the two bounds, we infer that there exists $C > 0$, which doesn't
depend on $r_k$ or $y_k$, such that
on $\lbr t < \tau^{(k)} \minwith T^{(k)} \rbr$
\begin{equation}
\label{eqn:SICVarLowerBound}
\Var^{CC(k)}_{\gs_0, \wt{\gs}_0} \lb \lpr \ol{W}^k_{t+1} \rabs \cF_t \rb \geq C n^{-2},
\end{equation}
for $n$ sufficiently large.
We now apply Lemma \ref{lem:DiffusionHittingTime} with
$Z_t = \ol{W}^k_t$, $z_0 = \frac{2 r_{k-1}}{\sqrt{n}}$ and $N = \tau^{(k)} \minwith T^{(k)}$. This gives for $\gga_k > 0$:
\begin{eqnarray*}
\bbP^{CC(k)}_{\gs_0, \wt{\gs}_0} \lb T^{(k)} \minwith \tau^{(k)} > \gga_k n \rb
	& \leq	& \frac{C r_{k-1}}{\sqrt{\gga_k}}\,,
\end{eqnarray*}
whence we may choose $\gga_k = \gga_k(r_{k-1}, \gep)$
independently of $r_k$, $y_k$ but sufficiently large, such that
\begin{equation}
\label{eqn:SICLemBound1}
\bbP^{CC(k)}_{\gs_0, \wt{\gs}_0} \lb T^{(k)} \minwith \tau^{(k)} > \gga_k n \rb
	\leq	\frac{\gep}{3} .
\end{equation}
This gives an upper bound on $T^{(k)}$, since by Proposition~\ref{clm:StayInRho} Part~\eqref{item:ClmStayInRho_2}
we may choose $r_k$ large enough such that:
\begin{equation}
\label{eqn:SICLemBound2}
\bbP^{CC(k)}_{\gs_0, \wt{\gs}_0} \lb \tau^{(k)} < \gga_k n \rb \leq \frac{\gep}{3}\,.
\end{equation}

It remains to ensure that we do not increase the distances in the first
$k-1$ coordinates, by too much. Proposition~\ref{prop:SICProps} implies that for any
$t$:
\begin{eqnarray*}
\lefteqn{\bbP^{CC(k)}_{\gs_0, \wt{\gs}_0} \lpr \lb
			\ol{W}^{[1,k-1]}_{t+1} \neq \ol{W}^{[1,k-1]}_t \rabs \cF_t \rb} \\
	& 	\leq & 	
\sum_{l \leq k-1}
\bbP^{CC(k)}_{\gs_0, \wt{\gs}_0} \lb \lpr
					\lb \lbr {I_{t+1}}=l \rbr \symdiff \lbr \wt{I}_{t+1}=l \rbr \rb \cup
					\lb \lbr {J_{t+1}}=l \rbr \symdiff \lbr \wt{J}_{t+1}=l \rbr \rb
				\rabs \cF_t \rb \\
	& \leq	& 	\frac{3}{2} \sum_{l \leq k-1} \lb
					|\ol{W}_t^l| +
					|g_{\gb}^l (S_t) - g_{\gb}^l (\wt{S}_t)| \rb + O (n^{-1}) \\
    & \leq  &   C \normI{\ol{W}_t^{[1, k-1]}} + r_k^2 O \lb  n^{-1} \rb .\\
\end{eqnarray*}
It follows that
\begin{eqnarray*}\
\bbE^{CC(k)}_{\gs_0, \wt{\gs}_0} \lsb \lpr \normI{\ol{W}_{t+1}^{[1, k-1]}} \rabs \cF_t \rsb
	& \leq		& \normI{\ol{W}_t^{[1, k-1]}} + \frac{C_1}{n} \normI{\ol{W}_t^{[1, k-1]}} + r_k^2 O \lb  n^{-2} \rb \\
	& \leq 		& \normI{\ol{W}_t^{[1, k-1]}} \lb 1 + \frac{C_1}{n} \rb + r_k^2 O \lb n^{-2} \rb.
\end{eqnarray*}
Taking expectation of both sides and using the assumption on $\normI{\ol{W}_0^{[1, k-1]}}$,
we have
\[
\bbE^{CC(k)}_{\gs_0, \wt{\gs}_0} \normI{\ol{W}_{\gga_k n}^{[1, k-1]}}
	\leq	\normI{\ol{W}_0^{[1, k-1]}} \lb 1 + \frac{C_1}{n} \rb^{\gga_k n}
	\leq 	C_1 \frac{u_{k-1}}{n} e^{C_2 \gga_k}.
\]
Hence by Markov's inequality, there exists $u_k > 0$ such that
\begin{equation*}
\bbP^{CC(k)}_{\gs_0, \wt{\gs}_0} \lb \normI{\ol{W}_{\gga_k n}^{[1, k-1]}} > \frac{u_k-y_k}{n} \rb
	\leq \frac{\gep}{3} .
\end{equation*}

Combined with \eqref{eqn:SICLemBound1} and \eqref{eqn:SICLemBound2}, the proof is complete.
\end{proof}

\begin{cor}
\label{cor:SICAnalysis}
Fix $\gb < q/2$. For any $\gep, r > 0$, there
exist $\gga, u, r^{\prime} > 0$ and $y_1, \dots y_{q-1} > 0$ such that
for $\gs_0, \wt{\gs}_0 \in \gS^{\frac{r}{\sqrt{n}}}_n$.
\begin{equation*}
\bbP^{CC}_{\gs_0, \wt{\gs}_0} \lb
			T^{CC} < \gga n ,\, (\gs_{T^{CC}},\wt{\gs}_{T^{CC}}) \in \cH_{u, r^{\prime}}
	\rb \geq 1-\gep.
\end{equation*}
\end{cor}
\begin{proof}
Starting from $r_0 = u_0 = r$ and applying Lemma~\ref{lem:SICBounds} inductively, we obtain for some $(y_k, u_k, r_k, \gga_k)_{k \in [1,q-1]}$:
\[
\bbP^{CC}_{\gs_0, \wt{\gs}_0} (T^{(k)} < \gga_k n, (\gs_{T^{(k)}},\wt{\gs}_{T^{(k)}}) \in \cH^k_{u_k, r_k} \ \forall k \in [1,q-1])
    \geq 1- \gep\,.
\]
It remains to set
$\gga = \sum_{k=1}^{q-1} \gga_k$, $r^{\prime} = r_{q-1}$
and $u = 2 u_{q-1}$.
\end{proof}

\subsection{Coalescence of Proportions Vector Chains}

The next lemma completes the coupling of the proportions chains.
\begin{lem}
\label{lem:CoalescenceOfS}
Fix $\gb < q/2$. For all $r,u,\gep > 0$ there exists $\gga > 0$ such that
if $\gs_0, \wt{\gs}_0 \in \gS_n$ satisfy $(\gs_0, \wt{\gs}_0) \in \cH_{u, r}$
and $t \geq \gga n$, then
\[
\bbP^{SC}_{\gs_0, \wt{\gs}_0} \lb S_{t} = \wt{S}_{t} \rb \geq 1-\gep\,.
\]
where under $\bbP^{SC}_{\gs_0, \wt{\gs}_0}$, the processes $(\gs_t)_{t \geq 0}$, $(\wt{\gs}_t)_{t \geq 0}$ evolve according to the synchronized coupling, as defined in Subsection \ref{sub:SynchronizedCoupling}.
\end{lem}
\begin{proof}
If $\rho$ is small enough, it follows from Lemma~\ref{lem:SyncCouplingCoalesence} that
\[
\Escx{\rho}_{\gs_0, \wt{\gs}_0} \normI{S_t - \wt{S}_t}
	\leq \lb 1 - \frac{1- 2\gb/q}{2n} \rb^t \frac{u}{n}\,,
\]
Combined with Proposition~\ref{clm:StayInRho} Part~\eqref{item:ClmStayInRho_1}, this implies that there exists
$\gga= \gga(u)$ such that
$\bbE^{SC}_{\gs_0, \wt{\gs}_0} \normI{S_{t} - \wt{S}_{t}} \leq \frac{\gep}{n}$
for $t \geq \gga n$.
Then by Markov's Inequality:
\begin{equation*}
\bbP^{SC}_{\gs_0, \wt{\gs}_0} \big( S_{t} \neq \wt{S}_{t} \big)
	= 	\bbP^{SC}_{\gs_0, \wt{\gs}_0} \big( \normI{S_{t} - \wt{S}_{t}} > \tfrac{1}{n} \big)
	\leq \gep\,.
\qedhere
\end{equation*}
\end{proof}

\subsection{Basket-wise Proportions Coalescence}

The next coupling will allow us to turn a well-mixed proportions chains into a well-mixed configurations chain.
Let $\frB = (\cB_m)_{m=1}^q$ be a partition of $[1,n]$. We shall refer to $\cB_m$ as a {\it basket} and call $\frB$ a $\gl$-partition
if $|\cB_m| > \gl n$ for all $m$.
Given $\gs \in \gS_n$, let $\bS(\gs)$ denote a $q \times q$ matrix whose $(m,k)$ entry is equal to the proportion in $\gs$ of color $k$ in basket $m$, namely
\[
\bS^{m,k}(\gs) = \frac{1}{|\cB_m|} \sum_{v \in \cB_m} \one_{\{\gs(v)=k\}}\,.
\]
$\bS$ is an element of $\bbS \bydef \prod_{m=1}^q \cS$ and we define
$\bbS^{\rho} = \prod_{m=1}^q \cS^{\rho}$ and $\bbS^{\rho+} = \prod_{m=1}^q \cS^{\rho+}$. We also let $\cB_{[m_0, m_1]} = \cup_{m_0 \leq m \leq m_1} \cB_m$ and as before use $\bS_t$ as a shorthand for $\bS(\gs_t)$.
The following is an analogue of Lemma~ \ref{lem:OSqrtFromCoalesence} for the basket proportions matrix.
\begin{lem}
\label{lem:BasketwiseOSqrtNFromEquilibrium}
Let $\frB$ be a $\gl$-partition for some $\gl > 0$.
If either of the following holds:
\begin{enumerate}
\item
\label{item:LemBasketwiseContraction1}
	$\gs_0 \in \gS_n^{\rho_0}$ and $t^{\ga_1}(n) \leq t \leq e^{\gga(\rho_0) n}$ where $\rho_0$, $\gga(\rho_0)$ are given in Proposition \ref{prop:ContractionForNorm} and $t^{\ga_1}(n)$ is defined in \eqref{eqn:TOfGammaDef}.
\item
\label{item:LemBasketwiseContraction2}
	$\bS_0 \in \bbS^{\frac{r_0}{\sqrt{n}}}$ and $t \leq \gga_0 n$ for some $r_0, \gga_0 > 0$,
\end{enumerate}
\smallskip
then
\[
\bbP_{\gs_0} \lb \bS_t \notin \bbS^{\frac{r}{\sqrt{n}}} \rb = O(r^{-2})\,.
\]
where the $O(r^{-2})$ term is as $r \to \infty$, uniformly in $n$.
\end{lem}
In order to prove Lemma~\ref{lem:BasketwiseOSqrtNFromEquilibrium}, we use the following proposition to bound the second moment of the basket proportions matrix.
\begin{prop}
\label{lem:BasketToGlobalProportions2edMomentComparison}
If $\frB$ is a $\gl$-partition for $\gl>0$ and $m,k \in [1,q]$, then
\begin{equation}
\label{eqn:BasketToGlobalProportions2edMomentComparison}
\bbE\left [(\bS^{m,k}_{t+1} - S^k_{t+1})^2 \mid \cF_t \right] = \left(1- \frac{2}{n}\right)(\bS^{m,k}_t - S^k_t)^2+ O \left(\frac{1}{n^2}\right)~.
\end{equation}
\end{prop}
\begin{proof}
Let $\gl_0 = |\cB_m|/n \geq \gl > 0$ and set
$\bQ^{m,k}_t = \bS^{m,k}_t - S^k_t$. Then:
\begin{eqnarray}
\nonumber
\lefteqn{
\bbE \left [(\bQ^{m,k}_{t+1})^2 - (\bQ^{m,k}_t)^2)
	\mid \cF_t \right]} \\
\nonumber
& = 	&
	p_1 ((\bQ^{m,k}_t + \tfrac{1}{n})^2 - (\bQ^{m,k}_t)^2) +
	p_2 ((\bQ^{m,k}_t - \tfrac{1}{n})^2 - (\bQ^{m,k}_t)^2) + \\
\nonumber
&&	\qquad p_3 ((\bQ^{m,k}_t - \tfrac{1}{\gl_0 n} + \tfrac{1}{n})^2 - (\bQ^{m,k}_t)^2) +
p_4 ((\bQ^{m,k}_t + \tfrac{1}{\gl_0 n} - \tfrac{1}{n})^2 - (\bQ^{m,k}_t)^2)	\\
\label{eqn:BasketRecursion}
& = 	&
	\frac{2}{n}\bQ^{m,k}_t \lb (p_1 - p_2) + \lb 1 - \inv{\gl_0} \rb (p_3 - p_4) \rb + O \lb \inv{n^2} \rb\,,
\end{eqnarray}
where (denoting by $V_{t+1}$ the chosen vertex at step $t+1$):
\begin{eqnarray*}
p_1 & = & \bbP (V_{t+1} \notin \cB_m,\, \gs_{t}(V_{t+1})=k,\, \gs_{t+1}(V_{t+1}) \neq k | \cF_t) = (S_t^1 - \gl_0 \bS^{m,k}_t)(1-g_{\gb}^k(S_t)) + O(n^{-1}), \\
p_2 &=& \bbP (V_{t+1} \notin \cB_m,\, \gs_{t}(V_{t+1}) \neq k,\, \gs_{t+1}(V_{t+1}) = k | \cF_t)
= (1-S_t^1 - \gl_0 + \gl_0 \bS^{m,k}_t) g_{\gb}^k(S_t)  + O(n^{-1}), \\
p_3 &=& \bbP (V_{t+1} \in \cB_m,\, \gs_{t}(V_{t+1})=k,\, \gs_{t+1}(V_{t+1}) \neq k | \cF_t) = \gl_0 \bS^{m,k}_t (1-g_{\gb}^k(S_t)
 + O(n^{-1}),\\
p_4 &=& \bbP (V_{t+1} \in \cB_m,\, \gs_{t}(V_{t+1}) \neq k,\, \gs_{t+1}(V_{t+1}) = k | \cF_t) = \gl_0 (1-\bS_t^{m,k}) g_{\gb}^k(S_t)  + O(n^{-1}).
\end{eqnarray*}
Plugging these into \eqref{eqn:BasketRecursion}, we obtain
\[\bbE [(\bS^{m,k}_{t+1} -
S^k_{t+1})^2 \mid \cF_t ] = (1- \tfrac{2}{n})(\bS^{m,k}_t - S^k_t)^2+
O(\tfrac{1}{n^2})
\]
as required.
\end{proof}

\begin{proof}[Proof of Lemma~\ref{lem:BasketwiseOSqrtNFromEquilibrium}]
Taking expectation in
\eqref{eqn:BasketToGlobalProportions2edMomentComparison}
and applying \eqref{eqn:RecursionSolution_SHat} one gets:
\[
\bbE_{\gs_0} \lsb \bS^{m,k}_{t} - S_{t}^k \rsb^2
	\leq 	\lb 1-\frac{2}{n} \rb^t \bbE_{\gs_0} \lsb \bS^{m,k}_0 - S_0^k \rsb^2 + O(n^{-1})\,.	
\]
In both cases (note that $\ga_1(\gb, q) > \inv{2}$ for all $\gb<q/2$), it implies
$\bbE_{\gs_0} \lsb \bS^{m,k}_t - S_t^k \rsb^2 = O(n^{-1})$. Summing over all $m$ and $k$ and using Markov's Inequality we get
\[
\bbP_{\gs_0} \lb \sum_{m \leq q} \normII{\bS^{m}_t - S_t} > r/(2\sqrt{n}) \rb = O(r^{-2}).
\]
Now, by Proposition~\ref{prop:ContractionForNorm} Case~\eqref{item:LemBasketwiseContraction1} and Proposition~\ref{clm:StayInRho} Case~\eqref{item:ClmStayInRho_2} we also have that
$$\bbP_{\gs_0} \lb \normII{S_t - \vq} > r/(2\sqrt{n}) \rb = O(r^{-2}).$$
Combining the two, we complete the proof.
\end{proof}

\medskip
Suppose now that you have two initial configurations $\gs_0, \wt{\gs}_0$, such that $s_0 = \wt{s}_0$. The following is a coupling under which eventually (with probability 1) also $\bS_{t} = \wt{\bS}_{t}$. Equality is achieved one basket at a time, indexed below by $m$ and once the proportions in a basket are equated they remain so. We shall call this coupling {\it Basket-wise Coupling} and denote by $\bbP^{BC}$ the underlying probability measure.
\begin{enumerate}
\item
	Set $t=0$, $m=1$.
\item
	As long as $m \leq q$:
	\begin{enumerate}
	\item
		As long as $\bS_t^{m} \neq \wt{\bS}_t^{m}$:
		\begin{enumerate}
		\item
			Choose ``old'' color $I_{t+1}$ according to distribution $S_t =
			\wt{S}_t$.
		\item
			Choose ``new'' color $J_{t+1}$ according to distribution
			$g_{\gb} \lb S_t - \frac{1}{n} \sfe_{I_{t+1}} \rb =
			g_{\gb} \lb \wt{S}_t - \frac{1}{n} \sfe_{I_{t+1}} \rb$.
		\item
			Choose a vertex $V_{t+1}$ uniformly among all vertices in
			$[1,n]$ having color $I_{t+1}$ under $\gs_t$.
		\item
		\label{item:BC_ChoosingVTilde}
			Choose $\wt{V}_{t+1}$:
			\begin{enumerate}
			\item \label{case-A}
				If $V_{t+1} \in \cB_{m_0}$ for $m_0 <m$, choose $\wt{V}_{t+1}$
				uniformly among all vertices in  $\cB_{m_0}$ having color
				$I_{t+1}$ under $\wt{\gs}_t$.
			\item \label{case-B}
				Otherwise, if $\bS_t^{m, I_{t+1}} \neq \wt{\bS}_t^{m, I_{t+1}}$ and
				$\bS_t^{m, J_{t+1}} \neq \wt{\bS}_t^{m, J_{t+1}}$, choose $\wt{V}_{t+1}$
				uniformly among all vertices in $\cB_{[m, q]}$
				having color $I_{t+1}$ under $\wt{\gs}_t$.
			\item \label{case-C}
				Otherwise, let $v_1, v_2, \dots$ be an enumeration of the vertices
				in $\cB_{[m, q]}$ having color $I_{t+1}$ under
				$\gs_t$ ordered first by the index of the basket they belong to and 			
				then by their index in $V$ and let $\wt{v}_1, \wt{v}_2, \dots$ be
				the same for $\wt{\gs}_t$. Then set $\wt{V}_{t+1}=\wt{v}_i$ where $i$ is such that $V_{t+1}=v_i$.
			\end{enumerate}
		\item
			Set $\gs_{t+1}(V_{t+1}) = \wt{\gs}_{t+1}(\wt{V}_{t+1}) = J_{t+1}$
			and $t=t+1$.
		\end{enumerate}
		\item		
			Set $m=m+1$.
	\end{enumerate}
\end{enumerate}

The following lemma gives an upper bound for the time of basket-wise proportions coalescence.
\begin{lem}
\label{lem:BasketwiseCouplingAnalysis}
Fix $\gb < q/2$. For any $\gl >0$, $r > 0$, $\gep > 0$, there exists $\gga = \gga(\gl, r, \gep)$, such that for any $\gl$-partition and any $\gs_0, \wt{\gs_0}$ such that $S_0 = \wt{S}_0$ and $\bS_0, \wt{\bS}_0 \in \bbS^{\frac{r}{\sqrt{n}}}$,
\begin{equation}
\label{eqn:BasketwiseCouplingAnalysis}
\bbP^{BC}_{\gs_0, \wt{\gs}_0} \lb \bS_{\gamma n} = \wt{\bS}_{\gamma n} \rb \geq 1 - \gep\,.
\end{equation}
\end{lem}
\begin{proof}
From the definition of the coupling, once the proportions of basket $m$ have coalesced they will remain equal forever. It suffices, therefore, to analyze the coalescence time of each basket separately. Note also that the coupling preserves the equality $S_t = \wt{S}_t$ for all $t \geq 0$.

Define $\bW_t = \bS_t - \wt{\bS}_t$,
$W^m_t = \normI{\bW^{m}_t}$ and let $\tau^{(0)} = 0$ and
$\tau^{(m)} = \min\{t \geq \tau^{(m-1)}: W^m_t = 0\}$ for $m \in [1,q]$.
Also set
\[
\tau_* = \inf \{t : \: \bS_t \notin \bbS^\rho \ \text{or} \
	\wt{\bS}_t \notin \bbS^\rho \}
\]
for $\rho>0$ sufficiently small and $\tau^{(m)}_* = \tau^{(m)} \minwith \tau_*$.
We claim that for all $m$,
$(W^m_t)_{t \geq 0}$ is a supermartingale between $\tau^{(m-1)}$ and
$\tau^{(m)}$ as long as $\tau_*$ is not reached. In order to see this, fix $m$, $t$ and assume $\{\tau^{(m-1)} \leq t < \tau^{(m)} \; ,\, \tau_* > t\}$. Then at step
\eqref{item:BC_ChoosingVTilde}, according to the coupling, there are 3 cases:
\begin{enumerate}[(A)]
\item
Clearly $\bS^{m}_{t+1} = \bS^{m}_{t}$ and
$\wt{\bS}^{m}_{t+1} = \wt{\bS}^{m}_{t}$ and hence
$W^m_{t+1} = W^m_t$.
\item  Notice that in this case, we have
$\bW_t^{m, I_{t+1}} \bW_{t+1}^{m, I_{t+1}} \geq 0$ and
$\bW_t^{m, J_{t+1}} \bW_{t+1}^{m, J_{t+1}} \geq 0$.
Therefore,
\begin{eqnarray*}
\lefteqn{\bbE_{\gs_0,\wt{\gs}_0}^{BC}[W^m_{t+1} - W^m_t | \cF_t]} \\
& = & \big| \bbE_{\gs_0,\wt{\gs}_0}^{BC} \lsb \bW^{m,I_{t+1}}_{t+1} | \cF_t \rsb \big| - \big|\bW^{m,I_{t+1}}_t \big|
			+ \big|\bbE_{\gs_0,\wt{\gs}_0}^{BC} \lsb \bW^{m,J_{t+1}}_{t+1} | \cF_t \rsb\big| -
				\big| \bW^{m,J_{t+1}}_t \big| \\
& \leq &  \frac{-|\bW^{m, I_{t+1}}_t|}{\sum_{m_0 \geq m} \bS^{m_0, I_{t+1}}_t |\cB_{m_0|}} + \frac{|\bW^{m, I_{t+1}}_t|}{\sum_{m_0 \geq m} \bS^{m_0, I_{t+1}}_t |\cB_{m_0}|}
	= 0\,.
\end{eqnarray*}
\item
If $V_{t+1}, \wt{V}_{t+1} \in \cB_m$ or $V_{t+1}, \wt{V}_{t+1} \notin \cB_m$, then
$W^m_{t+1} = W^m_t$, otherwise from the construction we must have:
\[
|\bW^{m, I_{t+1}}_{t+1}| - |\bW^{m, I_{t+1}}_{t}| = -\inv{|\cB_m|} \,,
\]
as well as
\[
|\bW^{m, J_{t+1}}_{t+1}| - |\bW^{m, J_{t+1}}_{t}| \leq \inv{|\cB_m|} \,.
\]
Summing these two, we obtain a non-positive drift for $W^m_t$.
\end{enumerate}

Observe that as long as $\tau_*$ is not reached, both $\Var^{BC}(W^m_{t+1} | \cF_t)$ under case (B) and the probability that this case happens are bounded below uniformly in $n$ and $t$. This gives a uniform lower bound on the variance
$\Var^{BC}(W^m_{t+1} | \cF_t)$.
Furthermore, if for some $t$, we have $\bS_t$, $\wt{\bS}_t \in \bbS^{\frac{r^{\prime}}{\sqrt{n}}}$, then in view of Lemma \ref{lem:BasketwiseOSqrtNFromEquilibrium} after $\gga^{\prime} n$ time,
we have $\bS_{t+\gga^{\prime} n}$, $\wt{\bS}_{t+\gga^{\prime} n} \in \bbS^{\frac{r^{\prime\prime}}{\sqrt{n}}}$ with probability $1-O((r^{\prime\prime})^{-2})$.
Therefore, using Lemma \ref{lem:DiffusionHittingTime} we may find $\gga_1, \dots, \gga_{q-1}$ such that inductively, conditioned on $\tau^{(m-1)}_* \leq \gga_{m-1} n$ with probability at least $1-\gep/(2q)$ we have $\tau^{(m)}_* \leq \gga_m n$. This in turn implies that $\tau^{(q-1)}_* \leq \gga n$ with probability at least $1-\gep/2$, where $\gga \bydef \gga_{q-1}$.

It remains to bound $\tau_*$ below with high probability.
Let $B^{m,j} = \bigcup_{t = 1}^{\gga n}\{ |\bS_t^{m, j} - 1/q | \geq \rho\}$ and
\[
Y^{m,j} = \labs \{t :\:\: \big|\bS_t^{m, j} - 1/q \big| \geq \rho/2, \, 1 \leq t \leq \gamma n \}\,\rabs.
\]
Using Lemma \ref{lem:BasketwiseOSqrtNFromEquilibrium}
we obtain that
\[\bbE_{\gs_0,\wt{\gs}_0}^{BC} [Y^{m,j}] \leq \gga n O(\tfrac{1}{n}) = O(\gga)\,.\]
Then as $B^{m,j}$ implies that $Y^{m,j} > \frac{n\lambda \rho}{2}$,
\[\bbP_{\gs_0,\wt{\gs}_0}^{BC}(B^{m,j}) \leq \bbP_{\gs_0,\wt{\gs}_0}^{BC}\Big(Y^{m,j} \geq \frac{n \gl \rho}{2}\Big) \leq \frac{2\bbE_{\gs_0,\wt{\gs}_0}^{BC}[Y^{m,j}]}{n \gl \rho} = O(n^{-1}) \,.\]
Summing over all $m$, $j$ and arguing the same for $\wt{\bS}_t$ we obtain
\[
\bbP_{\gs_0,\wt{\gs}_0}^{BC} (\tau_* < \gamma n) = O(n^{-1})\,,
\]
Finally by a union bound we have
$\bbP_{\gs_0,\wt{\gs}_0}^{BC} (\tau^{(q)} \leq \gamma n) \geq 1-\frac{\gep}{2} + O(n^{-1})$
as desired.
\end{proof}

\subsection{The Overall Coupling}
\label{sub:FullCoupling}

We now describe precisely how the previous couplings are combined together to create the {\it overall coupling}. This coupling will be the main tool in proving the upper bound. Formally, let $\gga_1, \gga_3, \gga_4, \gga_5$ and $y_1, \dots, y_{q-1}$ be positive numbers and $\gs_0 \in \gS_n$. The overall coupling with parameters $\gga_1, \dots, \gga_5$, $y_1, \dots y_{q-1}$ and initial configuration $\gs_0$ is a coupling of two chains $(\gs_t)_t$, $(\wt{\gs}_t)_t$ under measure $\bbP^{OC}_{\gs_0}$.
The initial configuration for $(\gs_t)_t$ is $\gs_0$, while $\wt{\gs}_0$ is chosen according to $\mu_n$. Then, the two processes evolve as follows.
\begin{enumerate}
\item
\label{item:FC_Burn}	
    Run $\gs_t$ and $\wt{\gs}_{t}$ independently until time $t^{(1)}(n) = \gga_1 n$.
    \begin{enumerate}
        \item[(\ref{item:FC_Burn}A)]
            Partition the vertex set $[1, n]$ into baskets $\frB = (\cB_1, \dots \cB_q)$ such that
            $\cB_k = \{v : \: \gs_{t^{(1)}(n)}(v) = k\}$ for $k\in [1, q]$.
    \end{enumerate}
\item
	Run $\gs_{t}$ and $\wt{\gs}_{t}$ independently (again) until time $t^{(2)}(n) = t^{(1)} + t^{\ga_1}(n)$ time where $t^{\ga_1}(n)$ is defined in \eqref{eqn:TOfGammaDef}.
\item
\label{item:FC_SIC}
	Run $\gs_{t}$ and $\wt{\gs}_{t}$ according to the coordinate-wise coupling with parameters $y_1, \dots, y_{q-1}$ until time
		$t^{(3)}(n) = t^{(2)}(n) + \gga_3 n$ (unless stopped before).
\item
\label{item:FC_Sync}
	Run $\gs_{t}$ and $\wt{\gs}_{t}$ according to the synchronized coupling until time $t^{(4)}(n) = t^{(3)}(n) + \gga_4 n$ time.
\item
	Run $\gs_{t}$ and $\wt{\gs}_{t}$ according to the basket-wise coupling for $t^{(5)}(n) = t^{(4)}(n) + \gga_5 n$ time with the baskets above.
\end{enumerate}

\subsection{Proof of Upper Bound in Theorem \ref{thm:SubCriticalCutoff}}
\label{sub:ProofOfUpperBoundA}
We will now use the overall coupling with appropriate parameters to establish the upper bound of the mixing time.
Recall that $\gb < \gb_s(q)$. Fix $\gep > 0$, pick $\rho > 0$ small enough and let $\gs_0$ be any initial configuration.
By Proposition~\ref{clm:StayInRho} Part~\eqref{item:ClmStayInRho_3}, we can choose $\gga_1$ large enough such that
\begin{equation}
\label{eqn:FCFirstStepAnalysis}
\bbP_{\gs_0}^{OC} \lb S_{t^{(1)}(n)} \in \cS^{\rho} \rb \geq 1-\gep\,.
\end{equation}

Assuming that this event indeed occurred, $\frB$ is a $(\inv{q} - \rho)$-partition and provided that $\rho$ is small enough, the conditions in Lemma~\ref{lem:OSqrtFromCoalesence} are
 satisfied. From the latter we conclude that for some $r > 0$, with probability at least $1-\gep$,
 $S_{t^{(2)}(n)} \in \cS^{\frac{r}{\sqrt{n}}}$. On the other hand, as in \eqref{eqn:StationaryMeasureConcentration}
  with probability at least $1-2\gep$ we also have $\wt{S}_{t^{(2)}(n)} \in \cS^{\frac{r}{\sqrt{n}}}$ if $r$ is large enough.
   Then Corollary \ref{cor:SICAnalysis} and
Lemma \ref{lem:CoalescenceOfS} ensure that there exist $y_1, \dots y_{q-1}$ and $\gga_3$, $\gga_4$, such that
$S_{t^{(4)}(n)} = \wt{S}_{t^{(4)}(n)}$
with probability at least $1-3\gep$. From Lemma
\ref{lem:BasketwiseOSqrtNFromEquilibrium} we have
that
$\bS_{t^{(4)}(n)}, \wt{\bS}_{t^{(4)}(n)} \in \bbS^{\frac{r^{\prime}}{\sqrt{n}}}$ with probability at least $1-4\gep$ for some $r^{\prime}>0$.  Then, by Lemma \ref{lem:BasketwiseCouplingAnalysis} we may choose $\gga_5$ such that $\bS_{t^{(5)}(n)} = \wt{\bS}_{t^{(5)}(n)}$ with
probability at least $1-5\gep$.

Now, by symmetry, for any $t \geq t^{(1)}(n)$ the distribution of $\sigma_t$, given $\cF_{t^{(1)}(n)}$, is invariant under permutations of the vertices in each basket of $\cB$ and the same is clearly true for $\mu_n$. Therefore we conclude that
\begin{eqnarray*}
\lefteqn{\normTV{\bbP_{\gs_0}^{OC} \lb \lpr \gs_{t^{(5)}(n)} \in \cdot \rabs \cF_{t^{(1)}(n)}, S_{t^{(1)}(n)} \in \cS^{\rho} \rb - \mu_n}} 	 \\
& = &
	\normTV{\bbP_{\gs_0}^{OC} \lb \lpr \bS_{t^{(5)}(n)} \in \cdot \rabs \cF_{t^{(1)}(n)}, S_{t^{(1)}(n)} \in \cS^{\rho} \rb - \mu_n \circ \bS^{-1}} \\
& \leq &
\bbP_{\gs_0}^{OC} \lb \lpr \bS_{t^{(5)}(n)} \neq \wt{\bS}_{t^{(5)}(n)} \rabs
	\cF_{t^{(1)}(n)}, S_{t^{(1)}(n)} \in \cS^{\rho} \rb
	\leq 5\gep\,.
\end{eqnarray*}
Then from Jensen's inequality we obtain
\begin{eqnarray}
\label{eqn:JensenForTV}
\lefteqn{\normTV{\bbP_{\gs_0} \lb \gs_{t^{(5)}(n)} \in \cdot \rb - \mu_n}}\\
\nonumber
	& \leq 	& \bbE_{\gs_0}^{OC} \lsb \lpr
					\normTV{\bbP_{\gs_0}^{OC} \lb \lpr \gs_{t^{(5)}(n)} \in \cdot \rabs 			
					\cF_{t^{(1)}(n)} \rb - \mu_n} \rabs S_{t^{(1)}(n)} \in \cS^{\rho} \rsb + 	
					\bbP_{\gs_0}^{OC} \lb S_{t^{(1)}(n)} \notin \cS^{\rho} \rb	\\
\nonumber
	& \leq 	& 5\gep + \gep = 6\gep.
\end{eqnarray}
Now $t^{(5)}(n)=t^{\ga_1}_{\gga}(n)$ (as defined in \eqref{eqn:TOfGammaDef}) with $\gga =\gga_1 + \gga_3 + \gga_4 + \gga_5$ and since $\gs_0$ is arbitrary and $\gep$ can be made arbitrarily small, by choosing $\gga$ large enough, this establishes the upper bound for the cutoff.
\qed

\section{Mixing in the Supercritical Regime}
\label{sec:SuperCriticalRegime}
\subsection{Proof of Theorem~\ref{thm:SuperCriticalSlowMixing}}
We first give the proof for the case $\gb_s < \gb < \gb_c$.
Recall (Subsection~\ref{sub:CWLDP})
that $\gb_c(q) = \frac{(q-1)}{q-2}\log (q-1)$ for $q\geq
3$ and $\gb_c(2) = 1$. We claim that for $q\geq 3$
\begin{equation}\label{eq-beta-c-q-q-1}
\gb_c(q) \big(1- 1/q\big) < \gb_c(q-1)\,.\end{equation}
It can be checked for $q=3$ and for $q \geq 4$, it suffices to prove that
$f(x) := \frac{x-1}{x (x-2)} \log (x-1)$ is decreasing in $x$ on $[3,
\infty)$. We compute the derivative and obtain that
\[f'(x) = -\frac{1}{x(x-2)}\Big(\frac{x^2 -2x+ 2}{x(x-2)} \log(x-1) -1\Big)\,,\]
which is negative for $x \geq 3$.

Now fix $\gd > 0$ and notice that if $s^1 \in [1/q, 1-\gd]$, conditional on $\{S^1 = s^1\}$,
$(S^i/(1-s^1):\: 2 \leq i \leq q)$ is distributed as the
proportions vector for the $(q-1)$-states Curie-Weiss Potts model on $(1-s^1) n$
vertices with $\gb' = \gb (1-s^1) \leq (1-1/q) \gb \leq (1 - 1/q)\gb_c(q) <
\gb_c(q-1)$. Therefore, for all $\gd_1 >0$
\begin{equation}
\label{eq-S-1-other-equal}
\mu_n \Big(\exists 2\leq i \leq q \mbox{ such that }
\big|S^i - \frac{1-s^1}{q-1}\big| \geq \gd_1 \mid S^1 = s^1\Big)
\to 0\,.\
\end{equation}
as $n \to \infty$ uniformly in $s^1 \in [1/q, 1-\gd]$.
Also uniformly in $s \in \cS_n$, recall that:
\[
\bbE \lsb \left. S^1_{t+1} - S^1_t \rabs S_t=s \rsb = \frac{1}{n} d_{\gb}(s) + O \lb n^{-2} \rb\,,
\]
where
$d_{\gb}(s) = -s^1 + g_{\gb}^1(s)$.
It now follows from the uniform continuity of $d_{\gb}(s)$
and \eqref{eq-S-1-other-equal} that
uniformly in $s^1 \in [1/q, 1-\gd]$,
\[
\bbE_{\mu_n}[S_{t+1}^1 - S_t^1 \mid S_t^1 = s^1] = \inv{n} (D_\gb(s^1)+ o(1))\,,
\]
where $D_{\gb}(s^1) = d_{\gb}\lb s^1, \frac{1-s^1}{q-1}, \dots, \frac{1-s^1}{q-1} \rb$.

Now if $\gb > \gb_s(q)$, from Proposition \ref{prop:DOfBetaProperties1},
there exists $\gd_2$, such that $D_{\gb}(s^1)$ is uniformly positive in a $\gd_2$-neighborhood of $s^*(\gb)$. All together we infer that there exists $\gep > 0$ such that uniformly in $s^1 \in (s^*(\gb) - \gd_2, s^*(\gb) + \gd_2)$ for all $n$ large enough:
\[
\bbE_{\mu_n}[S_{t+1}^1 - S_t^1 \mid S_t^1 = s^1] \geq \frac{\gep}{n}\,,
\]
and also
\[\bbP_{\mu_n}(S_{t+1}^1 = s_1^1 + \frac{j}{n} \mid S_t^1 = s_1^1) - \bbP_{\mu_n}(S_{t+1}^1 = s_1^1 + \tfrac{1}{n} + \frac{j}{n} \mid S_t^1 = s_1^1 + \tfrac{1}{n}) = o(1)\,,\]
for $j\in{-1,0,1}$ where the last inequality follows from the concentration of the conditioned measure as well as the continuity of the probability to stay put. These two formulas together imply that
\[
\bbP_{\mu_n}(S_{t+1}^1 = s_1^1 + \frac{1}{n} \mid S_t^1 = s_1^1) \geq \gl \bbP_{\mu_n}(S_{t+1}^1 = s_1^1  \mid S_t^1 = s_1^1 + \frac{1}{n})\,,
\]
for some fixed constant $\gl > 1$ for all $s^1 \in (s^*(\gb) - \gd_2, s^*(\gb)+\gd_2)$ when $n$ is sufficiently large. Since $(S_t)_{t \geq 0}$ is a reversible Markov chain, with $\mu_n$ its stationary measure,
\[
\bbP_{\mu_n}(S_{t+1}^1 = s_1^1 + \frac{1}{n}, S_t^1 = s_1^1) = \bbP_{\mu_n}(S_{t+1}^1 = s_1^1, S_t^1 = s_1^1 + \frac{1}{n})\,,
\]
and therefore, for all $s^1 \in
(s^*(\gb) - \gd_2, s^*(\gb)+\gd_2)$
\[
\mu_n(S^1 = s^1 + \frac{1}{n}) \geq \gl \mu_n(S^1 = s^1)\,,
\]
and hence
\begin{equation}
\label{eqn:FarMuComparison}
\mu_n(S^1 = s^*(\gb)+ \delta_2) \geq \gl^{2 \delta_2 n} \mu_n(S^1 = s^*(\gb) - \delta_2)\,.
\end{equation}

Now select the set $A = \{S_1 \geq s^*(\gb) - \delta_2\}$. By \eqref{eqn:FarMuComparison},
$\frac{\mu_n(\partial_{P_n} A)}{\mu_n(A)} \leq \gl^{-\delta_2 n}$, where
\[
\partial_{P_n} A = \{x \in A \;:\ P_n(x,y) > 0 \text{ for some } y \notin A \}
\]
and $P_n$ is the transition kernel of the Glauber dynamics. Since
$\gb < \gb_c(q)$ and $s^*(\gb) - \gd_2 > 1/q$ we also have
$\mu_n(A) = o(1)$ as $n \to \infty$. Therefore Cheeger's inequality (Theorem~\ref{thm:CheegersInequality}) immediately implies
an exponential lower bound on the mixing time.

The case $\gb \geq \gb_c(q)$ is simpler. As the large deviations analysis in Subsection~\ref{sub:CWLDP} shows, we may find $A = \{\normII{S - \check{s}_{\gb, q}} < \gd\}$, where $\check{s}_{\gb, q}$ is defined in \eqref{eqn:SCheckDef} and $\gd > 0$ is small enough such that
$\limsup_{n \to \infty} n^{-1} \log \mu_n(\partial_{P_n} A)<0$ and
$\liminf_{n \to \infty} n^{-1} \log \mu_n(A) = 0$.
Since symmetry implies $\mu_n(A) \leq 1/q$ (if $\gd$ is sufficiently small), exponential mixing time follows immediately from another application of Cheeger's inequality (Theorem~\ref{thm:CheegersInequality}).
\qed

\section{Mixing Near Criticality}
\label{sec:CriticalRegime}
We now assume $\gb(n) = \gb_s(q) - \xi(n)$, with $\xi(n) \to 0$ as $n \to \infty$. Once $\gb(n)$ approaches $\gb_s$ with $n$, we no longer have a uniform negative upper bound on the drift to the right of $1/q$ for each coordinate. Instead, near $s^*(\gb)$, the drift will be of order $\xi(n)$, possibly even positive and hence it will take longer than linear time to get close to $\vq$ and this may have an effect on the order of the mixing time and cutoff window. Accordingly, in addition to the coalescence time analysis near $\vq$, one has to obtain sharp asymptotics for the passage time near $s^*(\gb)$. This is achieved using several propositions which we state in Subsection~\ref{sub:PassgeTimeEstimates}. Their proofs will be deferred until the end of the section in favor of first showing how they are used along with the previous coalescence analysis to find the mixing time near criticality which gives the proof of Theorem~\ref{thm:NearCriticalMixing}.

Both the analysis and the results in Theorem~\ref{thm:NearCriticalMixing} are qualitatively different, depending on whether $\xi(n)$ decays faster or slower than some threshold rate. Accordingly, we shall distinguish between two regimes and write:
\begin{equation}
\label{eqn:DefRegimes}
\begin{array}{ll}
\xi \in \text{[CR]}	&
	\quad \text{if } \lim_{n \to \infty} n^{2/3} \xi(n) = \infty,\quad \xi(n)=o(1) \\
\xi \in \text{[NCR]}  &
	\quad \text{if } 0 \leq \liminf_{n\to\infty} n^{2/3} \xi(n) \leq 	
	\limsup_{n\to\infty} n^{2/3} \xi(n) < \infty
\end{array}
\end{equation}
([CR] stands for Cutoff Regime and [NCR] stands for No-Cutoff Regime). For $a > 0$, define also
\begin{equation}
\label{eqn:QuadDriftTotalPassTime}
t^{\xi, a}
_{\gga}(n) =
		\lbr
			\begin{array}{ll}		
				\frac{\pi}{\sqrt{a}} \frac{n}{\sqrt{|\xi(n)|}}
					+ \gga \lb \frac{n^{1/2}}{|\xi(n)|^{5/4}} \maxwith n \rb 		
							& \quad \text{if } \xi \in \text{[CR]} \\
				e^{\gga} n^{4/3}	
				 			& \quad \text{if } \xi \in \text{[NCR]}.
			\end{array}
		\rpr
\end{equation}
Both \eqref{eqn:DefRegimes} and \eqref{eqn:QuadDriftTotalPassTime} will be used for sequences other than $\xi$ as well.
We shall also employ the following notation for hitting times. Given a real-valued process $(X_t)_{t \geq 0}$ and a number $x \in \bbR$ we shall write
\[
	\tau_x^+ = \inf \{t :\: X_t \geq x \}
		\quad \text{and} \quad
	\tau_x^- = \inf \{t :\: X_t \leq x \}
\]
for the right and left hitting time of $X$ at $x$. Notice that this notation does not carry an indication for the process for which $\tau_x^+$ is a hitting time and in case this is not clear from the context, it will be mentioned explicitly.

\subsection{\texorpdfstring{Drift Analysis Near $s^*(\gb)$}{Drift Analysis Near s^*(\textbeta)}}
\label{sub:PassgeTimeEstimates}

The following proposition states several properties of the function $D_{\gb}$ near $s^*(\gb)$.
\begin{prop}
\label{prop:DOfBetaProperties2}
For all $q \geq 3$ the following holds:
\begin{enumerate}
\item
\label{item:DOfBetaProperties2-1}
The point $s^*(\gb_s)$ is the unique $s\in(\frac1q,1]$ such that $D_{\gb_s}(s)=0$.
\item
\label{item:DOfBetaProperties2-2}
    For $k=0,1, \dots$, the functions $D^*_k(\gb) \bydef \frac{d^k}{ds^k} D_{\gb}(s^*(\gb))$ are $C^{\infty}$
    in a neighborhood of $\gb_s$.  Furthermore:
    \begin{itemize}
    \item
        $\frac{d}{d \gb} D^*_0(\gb_s) > 0$.
    \item
        \label{eqn:FirstOrderMaxOfD}
        $ D^*_2(\gb_s) < 0$.
    \end{itemize}

\item
\label{item:DOfBetaProperties2-3}
    For all $\rho > 0$, there exists $\gd > 0$ such that:
    \begin{equation}
    \label{eqn:StrictNegativeRegionOfD}
		\sup \lbr D_{\gb}(s) : \: s \in \lb \inv{q} + \rho, 1 \rsb, |s - s^*(\gb)| > \rho,\, |\gb-\gb_s| < \gd \rbr < 0\,.
    \end{equation}
\end{enumerate}

\end{prop}

\bigskip
The next lemma gives sharp asymptotics for the passage time near $0$ for a  process with certain drift assumptions near $0$ (given by \eqref{eqn:QuadDriftRequirement} below). The one coordinate process will fall into this category if we analyze it near $s^*(\gb)$.

Formally, let $\lb (Z_t)_{t \geq 0}^n \ ; \; n \geq 0 \rb$ be a sequence of discrete time processes.
For all $n$, suppose that $\lb Z_t \rb_{t \geq 0} = \lb Z_t \rb_{t \geq 0}^n$ is
adapted to $\lb \cF_t \rb_{t \geq 0}^n$, satisfies $n|Z_{t+1}-Z_t| \in \{-1, 0,1\} $ with probability 1, and
\begin{equation}
\label{eqn:QuadDriftRequirement}
\bbE[Z_{t+1}-Z_t | \cF_t] =
\frac{1}{n} \lb \gz(n) + a Z_t^2 + b Z_t^3 + O \lb \gz(n) Z_t^2 + Z_t^4 \rb \rb
\end{equation}
where $a > 0$, $b \in \bbR$ and
$\gz(n)$ is a sequence satisfying  $\gz(n) \to 0$ as $n \to \infty$. We allow both $\gz \in \text{[CR]}$ and $\gz \in \text{[NCR]}$, but in the latter case, we assume in addition the existence of $d > 0$ such that for all $n$
\begin{equation}
\label{eqn:PosVarianceRequirement}
\Var[n(Z_{t+1} - Z_t)| \cF_t] \geq d.
\end{equation}
Write $\bbP_{z_0}$ for the probability measure under which this process is defined and starts from $z_0$.
\begin{lem}
\label{lem:QuadraticDriftPassTime}
Fix $\rho>0$ sufficiently small. Then for $z_0 = -\rho$
there exist functions $L^*,U^*:(-\infty,\infty) \to [0,1]$
satisfying $\lim_{\gga \to -\infty} L^*(\gga) = \lim_{\gga \to \infty} U^*(\gga) = 0$ such that for all $\gga$,
\begin{eqnarray}
	\label{eqn:QuadDriftPassTimeUpperBound}
	\limsup_{n \to \infty} \bbP_{z_0} \left( \tau_{\rho}^+ > t^{\gz, a}_{\gga}(n) \right) & \leq & U^*(\gga), \\
	\label{eqn:QuadDriftPassTimeLowerBound}
	\limsup_{n \to \infty} \bbP_{z_0} \left( \tau_{\rho}^+ < t^{\gz, a}_{\gga}(n) \right) & \leq & L^*(\gga),
\end{eqnarray}
where $\tau_{\rho}^+$ is a hitting time for $Z$.
Moreover, if  $\gz \in \text{\rm [NCR]}$ we can chose $U^*$ such that for all $\gga$ we have
\begin{equation}
\label{eqn:FastPassCond}
U^*(\gga) < 1.
\end{equation}
\end{lem}

\begin{rem}
\label{rem:ImprovedQuadraticDriftLem}
The upper (lower) bound in the lemma still holds
if $(Z_t)_{t \geq 0}$ satisfies \eqref{eqn:QuadDriftRequirement} with $\geq$ ($\leq$)
in place of the equality sign or if in place of $z_0 = -\rho$ we have
$z_0 \geq -\rho$ ($z_0 \leq -\rho$). Since $(Z_t)_{t \geq 0}$ has
$0, \pm \inv{n}$ steps this can be shown by a simple coupling argument.
\end{rem}

\medskip
The next proposition shows that the drift of one coordinate stays close to its upper bound
$D_{\gb}(\cdot)$ for sufficiently long time. More precisely,
for $\gs_0 \in \gS_n$, $t > 0$, $\gd > 0$, $y \in [0,1]$
let
\begin{equation}
\label{eqn:EOfNDefinition}
K_n(\gs_0, t, y, \gd) = \bbP_{\gs_0}
\lb \max_{0\leq \gt\leq \min\{t,\tau_y^- \}}   D_\gb(S_{\gt}^1)  - n \bbE_{\gs_0}
    \lsb S_{\gt+1}^1- S_{\gt}^1\mid\mathcal{F}_{\gt}  \rsb > \gd \rb,
\end{equation}
where $\tau_y^-$ is a hitting time for $S^1_t$.
Then,
\begin{prop}
\label{prop:SharpnessOfDBound}
Suppose that $\gb\leq q/2$ and set $\gs_0 \equiv 1$.
Then for any $y >\inv{q}$:
\begin{enumerate}
\item
\label{eqn:SharpnessOfDBoundOutOfWindow}
If $t(n)=o(n^2)$ and $\gd(n)n^{2}t(n)^{-1}\to\infty$ then $\lim_{n \to \infty}  K_n(\gs_0, t(n), y, \gd(n)) = 0$.
\item
\label{eqn:SharpnessOfDBoundInWindow}
If $t(n) = \gga n^{4/3}$ then for all $\gd>0$ we have
\[\lim_{\gga \to 0} \limsup_{n \to \infty}
	K_n(\gs_0, t(n), y, \gd n^{-2/3}) = 0.\]
\end{enumerate}
\end{prop}

\subsection{Proof of Theorem~\ref{thm:NearCriticalMixing}}
\subsubsection{Upper Bound on Mixing Time}
Fix $\rho>0$ small enough and let $\gs_0 \in \gS_n$ be given. By Proposition~\ref{prop:DOfBetaProperties2} Part~\eqref{item:DOfBetaProperties2-1}, we can find $\delta>0$ so that
\begin{equation}
\label{eqn:UniformNegativeDrift}
\sup \lbr D_{\gb}(s) : \: |\gb-\gb_s| < \gd, \: 1/q + \rho/2 < s < 1, \: |s - s^*| > \rho/2 \rbr < -\gd.
\end{equation}
where we use $s^*$ in place of $s^*(\gb)$.
Then by Lemma~\ref{lem:ImprovedAzuma} Part~\eqref{item:ImprovedAzuma1}, we have that, 
\[
\bbP_{\gs_0} \lb \tau_{(s^* + \rho)}^-  > (2 /\gd) n \rb = o(1)
\]
where this and all hitting times below are of $S_t^1$.
Define now $Z_t=s^* - S^1_{t+\tau_{(s^*+\rho)}^-}$.
Using \eqref{eqn:OneCoordinateDrift}, \eqref{eqn:LittleDOfBetaDef}
Proposition~\ref{prop:DOfBetaProperties2} and applying Taylor's expansion for $D_{\gb}(s)$
around $s^*$ and then again for $D_0^*(\gb)$ around $\gb_s$, we infer that there exist
$a > 0$, $\ga \neq 0$, $b \in \bbR$ such that
\[
\bbE[Z_{t+1}-Z_t | \cF_t] \geq
\frac{1}{n} \lb \gz(n) + a Z_t^2 + b Z_t^3 + O \lb \gz(n) Z_t^2 + Z_t^4 \rb \rb \,,
\]
where $\gz(n) = \ga \xi(n) + O(\xi(n)^2 + n^{-1})$ and also \eqref{eqn:PosVarianceRequirement} holds (if needed), since the probability of choosing any new color at time $t+1$ is bounded above and below, uniformly in $n$ and $S_t$. Hence by Lemma~\ref{lem:QuadraticDriftPassTime}
and Remark~\ref{rem:ImprovedQuadraticDriftLem}, for all $\gga$
\[
\bbP_{\gs_0} \lb \tau_{(s^*-\rho)}^- -
\tau_{(s^*+\rho)}^-  > t^{\gz, a}_{\gga}(n)
\rb \leq U^*(\gga) + o(1)
\]
Now, using the relation between $\gz(n)$ and $\xi(n)$,
it is not difficult to verify that $t^{\gz, a}_{\gga}(n) \leq t^{\xi, a^{\prime}}_{\gga^{\prime}}(n)$ for all $\gga^{\prime}$, where
$a^{\prime} = \ga a$ and $\gga = F(\gga^{\prime})$ for some $F$ such that $\gga \to \infty$ if $\gga^{\prime} \to \infty$.

From Lemma~\ref{lem:ImprovedAzuma} Part~\eqref{item:ImprovedAzuma3}, applied to the process
$(S^1_{t+\tau_{(s^*-\rho)}^-} - \lb s^* - \rho \rb : \; t \geq 0
)$, it follows that with $1-o(1)$ probability $S^1_{t+\tau_{(s^*-\rho)}^-}$ stays to
the left of $s^* - \rho/2$ for all $t < n^2$. Then we may apply Lemma~\ref{lem:ImprovedAzuma} Part~\eqref{item:ImprovedAzuma1} 
to the process $(S^1_{t+\tau_{(s^*-\rho)}^-} - (1/q + \rho/2) : \; t \geq 0)$ to conclude
\[
\bbP_{\gs_0} \lb \tau_{(q^{-1} + \rho/2)}^- -
\tau_{(s^*-\rho)}^- > (2/\gd) n
\rb =o(1)
\]
Finally another application of Lemma~\ref{lem:ImprovedAzuma} Part~\eqref{item:ImprovedAzuma3}
gives $S^1_{t+\tau_{(q^{-1}+\rho/2)}^-} \leq 1/q + \rho$ for all $t < n^2$
with $1-o(1)$ probability.
For the [CR] case, we  use union bound (over all coordinates):
\begin{equation}
\label{eqn:Order1InCR}
\bbP_{\gs_0} \lb S_{t_{\gga^{\prime}}^{\xi, a^{\prime}}(n)} \notin \cS^{\rho+} \rb
	\leq  q U^*(\gga) +o(1) .
\end{equation}
For the [NCR] case, define $\tau^{(1)} = \tau_{(q^{-1}+\rho/2)}^-$ and
$\tau^{(k)} = \inf \{t \geq \tau^{(k-1)}: S^k_t \leq q^{-1}+\rho/2 \}$
for $k > 1$.
Then, by inductive conditioning we obtain
\[
\bbP_{\gs_0} \lb \tau^{(k)} \leq t_{\gga^{\prime}}^{\xi, a}(n) : \: k=1, \dots, q \rb
	\geq (1-U^*(\gga))^q + o(1) .
\]
Since also $S^k_{t+\tau^{(k)}} \leq 1/q + \rho$ for all $k \in [1,q]$, $t < n^2$
with $1-o(1)$ probability, we arrive to
\begin{equation}
\label{eqn:Order1InNCR}
\bbP_{\gs_0} \lb S_{t_{\gga^{\prime}}^{\xi, a^{\prime}}(n)} \notin \cS^{\rho+} \rb
	\leq  1- (1-U^*(\gga))^q +o(1) .
\end{equation}

We now re-employ the overall coupling in Sub-section~\ref{sub:FullCoupling}, but in view of \eqref{eqn:Order1InCR} and \eqref{eqn:Order1InNCR} we change step \eqref{item:FC_Burn} and instead of running the two chains for $\gga_1 n$ time, we run them for $t^{(1)}(n) = t_{\gga^{\prime}}^{\xi,a^{\prime}}(n)$. As \eqref{eqn:Order1InCR}, \eqref{eqn:Order1InNCR} show, we can choose $\gga^{\prime}$ large enough such that
$\bbP^{OC}_{\gs_0} (S_{t_{\gga^{\prime}}^{\xi, a^{\prime}}(n)} \notin \cS^{\rho+}) \leq \gep$
for $n$ sufficiently large. The remaining steps in the coupling are left unchanged and we choose the same parameter values, as in the proof of Theorem~\ref{thm:SubCriticalCutoff}.

Using the analysis of the modified step \eqref{item:FC_Burn} given by \eqref{eqn:Order1InCR} and \eqref{eqn:Order1InNCR}, together with the analysis in Sub-section~\ref{sub:ProofOfUpperBoundA} of the remaining
steps - which carries over (uniformly in $\gb$ near $\gb_s(q)$), since it only required $\gb < \gb_c(q)$, we recover
\eqref{eqn:JensenForTV}, namely
\[
\normTV{\bbP_{\gs_0} \lb \gs_{t^{(5)}(n)} \in \cdot \rb - \mu_n}
	\leq 6\gep .
\]
The time is now given by
\[
t^{(5)}(n) = t^{\ga_1}_{\gga}(n) + t^{\xi, a^{\prime}}_{\gga^{\prime}}(n),
\]
for some $\gga > 0$.
Since $\gs_0$ is arbitrary and $\gep$ can be made arbitrarily small, by having $\gga$, $\gga^{\prime}$ large enough, this completes the proof for the upper bound in \eqref{eqn:CriticalMixingThm_OutOfCriticalWindow} and
\eqref{eqn:CriticalMixingThm_InCriticalWindow} with $\ga_2 = \pi/\sqrt{\ga a}$.

\subsubsection{No Cutoff in {\rm NCR} Case}
Using the modified overall coupling as introduced above, we obtain from \eqref{eqn:JensenForTV}, \eqref{eqn:Order1InNCR} and \eqref{eqn:FastPassCond}
for any $\gga^{\prime}$ and sufficiently large $\gga$
\[
\normTV{\bbP_{\gs_0} \lb \gs_{t^{(5)}(n)} \in \cdot \rb - \mu_n}
	\leq 1-\gep,
\]
for all $\gs_0$, large enough $n$ and some $\gep > 0$.
Then, since in the [NCR] case
\[
t^{(5)}(n) = t^{\ga_1}_{\gga}(n) + t^{\xi, a^{\prime}}_{\gga^{\prime}}(n)
	\leq t^{\xi, a^{\prime}}_{\gga^{\prime} + C}(n) ,
\]
this shows that there is no cut-off.

\subsubsection{Lower Bound on Mixing Time.}
Fix $\rho > 0$ small enough and start with $\gs_0 \equiv 1$ - the all '1' configuration.
Define:
\[
\gd(n) =
		\lbr
			\begin{array}{ll}		
				n^{-1}\xi(n)^{-1/2} A(n)
						\quad & \text{if } \xi \in \text{\rm [CR]} \\
				\gd_1 n^{-2/3}	
						& \text{if } \xi \in \text{\rm [NCR]}
			\end{array}
		\rpr
\]
where $A(n)$ is a sequence tending to $\infty$ sufficiently slowly and $\gd_1 >0$. Set $$N = \inf \lbr t : \: D_\gb(S_t^1)  - n \bbE \lsb S_{t+1}^1- S_{t}^1 \mid \cF_t \rsb > \gd(n) \rbr$$
and define the process $Y_t$ which is equal to $S_t^1$ up to time $N$, but after this time evolves like a birth-and-death processes with $\pm 1/n$ increments and drift $-n^{-1} D_{\gb}(S^1_t)$.
Then $(Z_t \bydef s^* - Y_t :\: t \geq 0)$ satisfies
\[
\bbE[Z_{t+1}-Z_t | \cF_t] \leq
\frac{1}{n} \lb \gz(n) + a Z_t^2 + b Z_t^3 + O \lb \gz(n) Z_t^2 + Z_t^4 \rb \rb \,,
\]
with $a$, $b$, $\ga$ as in the upper bound case, but with
$\gz(n) = \ga \xi(n) + \gd(n) + O(\xi(n)^2 + n^{-1})$ and condition \eqref{eqn:PosVarianceRequirement} holds (if needed) as before.
Then, using Lemma~\ref{lem:QuadraticDriftPassTime} and Remark~\ref{rem:ImprovedQuadraticDriftLem}, we have for $n$ large enough
$\bbP_{\gs_0} \lb \tau_{\rho}^+ < t^{\gz, a}_{\gga}(n) \rb \leq 2 L^*(\gga)$, where $\tau_{\rho}^+$ is a hitting time for $Z$ and $\gga \in \bbR$. As before, it is not difficult to verify that if  $A(n)$ is increasing slowly enough, $t^{\gz, a}_{\gga}(n) \geq t^{\xi, a^{\prime}}_{\gga^{\prime}}(n)$, where $\gga = F(\gga^{\prime})$ satisfies $\gga \to -\infty$ if $\gga^{\prime} \to -\infty$.

Now define $T = \inf \{t: \: S_t \in \cS^{\rho+}\}$ and
$\tau^{\prime}$ as $\tau_{\rho}^+$, only with $S_t^1$ in place of $Y_t$. Then
\begin{eqnarray*}
\bbP_{\gs_0} \lb T < t^{\xi, a^{\prime}}_{\gga^{\prime}}(n) \rb
	& \leq 	&
		\bbP_{\gs_0} \lb \tau_{\rho}^+ < t^{\xi, a^{\prime}}_{\gga^{\prime}}(n) \rb + 		
		\bbP_{\gs_0} \lb N < \tau^{\prime} \minwith t^{\xi, a^{\prime}}_{\gga^{\prime}}(n) \rb	\\
	& \leq 	& 2L^*(\gga) + K_n(\gs_0, t^{\xi, a^{\prime}}_{\gga^{\prime}}(n), s^* - \rho, \gd(n))
\end{eqnarray*}
where $K_n$ is defined in \eqref{eqn:EOfNDefinition}.
Then, if $\rho$ is sufficiently small, we can use \eqref{eqn:UpperBoundOnSqrtNCloseToEquilibrium}
for $\wh{S}_t$ starting from time $T$ to obtain for all $r>0$ and $\gga^{\prime\prime}$:
\begin{eqnarray*}
\lefteqn{\bbP_{\gs_0} \lb \normII{\wh{S}_{t^{\ga_1}_{\gga^{\prime\prime}}(n) + t^{\xi, a^{\prime}}_{\gga^{\prime}}(n)}} <
    \frac{r}{\sqrt{n}} \rb} \\
    & \leq  &
        \bbP_{\gs_0} \lb T < t^{\xi, a^{\prime}}_{\gga^{\prime}}(n) \rb +
        \bbP_{\gs_0} \lb \lpr \normII{\wh{S}_{t^{\ga_1}_{\gga^{\prime\prime}}(n) + t^{\xi, a^{\prime}}_{\gga^{\prime}}(n)}}
            < \frac{r}{\sqrt{n}} \rabs T \geq t^{\xi, a^{\prime}}_{\gga^{\prime}}(n) \rb    \\
    &  \leq & 2L^*(\gga) + K_n(\gs_0, t^{\xi, a^{\prime}}_{\gga^{\prime}}(n), s^* - \rho, \gd(n)) + O((e^{-C_2 \gga^{\prime\prime}} - r)^{-2})
\end{eqnarray*}
Using Proposition~\ref{prop:SharpnessOfDBound} for the middle term, the last inequality
gives \eqref{eqn:UpperBoundOnSqrtNCloseToEquilibrium1} with
$t^{\ga_1}_{\gga^{\prime\prime}}(n) + t^{\xi, a^{\prime}}_{\gga^{\prime}}(n)$ in place of $t^{\ga_1}_{\gga}(n)$. The remaining of the proof is identical to the subcritical case and this shows the lower bound for both parts of Theorem \ref{thm:NearCriticalMixing} with $\ga_2 = \pi/\sqrt{\ga a}$.
\qed

\subsection{Proofs for Subsection~\ref{sub:PassgeTimeEstimates}}
\begin{proof}[Proof of Proposition~\ref{prop:DOfBetaProperties2}]
First observe that for all
$\gb$, $D_\gb(\frac1q)=0$ and for all $s>\frac1q$,
\[
\frac{d}{d\gb} D_{\gb}(s) = \frac{d}{d\gb} \lb -s + \frac1{1+(q-1)e^{-\frac{2\gb q}{q-1}(s-\frac1q)}} \rb > 0\,.
\]
Now since $D_\gb(s)$ is smooth as a function of $s$ and $\gb$ and $d_0(s)=-s+\frac1q$ it follows that $\gb_s>0$.  We have that
\begin{align}\label{eqn:g_properties}
\frac{d}{ds} D_{\gb}(s) &= -1 + \frac{2\gb q e^{-\frac{2\gb q}{q-1}(s-\frac1q)}}{\left(1+(q-1)e^{-\frac{2\gb q}{q-1}(s-\frac1q)}\right)^2}\,,\nonumber \\
\frac{d^2}{ds^2}D_{\gb}(s) &=  \frac{4\gb^2 q^2 e^{-\frac{2\gb q}{q-1}(s-\frac1q)}(1-(q-1)e^{-\frac{2\gb q}{q-1}(s-\frac1q)})}{(q-1)\left(1+(q-1)e^{-\frac{2\gb q}{q-1}(s-\frac1q)}\right)^3}\,,
\end{align}
and so
\begin{equation}
\left . \frac{d}{ds} D_\gb(s) \right|_{s=\frac1q}=-1+\frac{2\gb}{q},\qquad \left . \frac{d^2}{ds^2} D_\gb(s) \right|_{s=\frac1q}=\frac{4\gb^2 (q-2)}{q(q-1)}>0\,.
\end{equation}
which implies that $d_{q/2}(s) > 0$ when $s\in(\frac1q,\frac1q+\epsilon)$ for some small $\epsilon$.  This implies that $\gb_s<q/2$. It follows that $$\left . \frac{d}{ds} D_{\gb_s}(s) \right|_{s=\frac1q} < 0\,,$$ and so $D_{\gb}(s)<0$ when $\gb \in [\gb_s,\gb_s+\epsilon]$ and $s\in(\frac1q,\frac1q+\epsilon)$ for some small $\epsilon$.    It follows by compactness then that for some $\frac1q + \epsilon\leq s^*(\gb_s)\leq 1$ that $D_{\gb_s}(s^*(\gb_s))=0$.  By the definition of $\gb_s$ and since $D_\gb(s)$ is smooth we have that
\begin{equation}\label{eqn:gxM_properties}
\left . \frac{d}{ds} D_{\gb_s}(s) \right|_{s=s^*(\gb_s)}= 0 ,\qquad \left . \frac{d^2}{ds^2}D_{\gb_s}(s) \right|_{s=s^*(\gb_s)} \leq 0\,.
\end{equation}
The equation $ \frac{d}{ds} D_{\gb_s}(s)= 0$ is equivalent to
\[
2\gb q e^{-\frac{2\gb_s q}{q-1}(s-\frac1q)}=\left(1+(q-1)e^{-\frac{2\gb_s q}{q-1}(s-\frac1q)}\right)^2\,,
\]
which is a quadratic equation in $e^{-\frac{2\gb_s q}{q-1}(s-\frac1q)}$ and hence has at most 2 solutions which we denote $s_1,s_2$ with $s_1< s_2$.  Since $ \frac{d}{ds} D_{\gb_s}(0)<0$ then $D_{\gb_s}(s_1)<0$ and so $s^*(\gb_s)=s_2$.  In particular this implies that $s^*(\gb_s)$ is the unique $s\in(\frac1q,1]$ such that $D_{\gb_s}(s)=0$.  Also it follows that $ \frac{d}{ds} D_{\gb_s}(s)>0$ for $s\in(s_1,s^*(\gb_s))$ and that there exists $s'\in (s_1,s^*(\gb_s))$ such that $ \frac{d^2}{ds^2} D_{\gb_s}(s')=0$.  Since by equation~\eqref{eqn:g_properties} there is at most one $s$ such that $\frac{d^2}{ds^2}D_{\gb_s}(s) = 0$ it follows that
\[
 \left . \frac{d^2}{ds^2}D_{\gb_s}(s) \right|_{s=s^*(\gb_s)} < 0.
\]
Hence by the Inverse Function Theorem $s^*(\gb)$ is a
smooth function of $\gb$ when $\gb$ is in a small neighborhood of $\gb_s$.  Then we have that
\begin{align*}
\frac{d}{d \gb} D^*_0(\gb_s) &= \left. \frac{d}{d \gb} D_\gb(s^*(\gb_s))\right|_{\gb=\gb_s} +\lb \left.  \frac{d}{d\gb}s^*(\gb) \right|_{\gb=\gb_s} \rb \left.  \frac{d}{ds}D_{\gb_s}(s) \right|_{s=s^*(\gb_s)}\\
&= \left.\frac{d}{d \gb} D_\gb(s^*(\gb_s))\right|_{\gb=\gb_s}\\
&> 0\,,
\end{align*}
since $\frac{d}{ds}D_{\gb_s}(s)|_{s = s^*(\gb_s)} = 0$ which completes the proof of the second part.

We now turn to prove the third part. As we have observed
$D_{\gb_s}(s)$ is a smooth function satisfying
\[D_{\gb_s}(s^*(\gb_s)) = 0\,,\, \frac{d}{ds}D_{\gb_s}(s)|_{s = s^*(\gb_s)} = 0\,, \mbox{ and } \frac{d^2}{ds^2}D_{\gb_s}(s)|_{s = s^*(\gb_s)} < 0\,.\]
Therefore, we deduce that for any $\rho > 0$, there exists
$\delta_1 > 0$ such that
\begin{equation}\label{eq-d-beta-M-s}
D_{\gb_s}(s) < -2\delta_1 \mbox{ for all } s\in\left\{ s:|s - s^*(\gb_s)| \geq
\rho/2, s\in \left(\frac1q+\rho,1\right]\right\}\,.\end{equation}
Since $\left . \frac{d^2}{ds^2}D_{\gb_s}(s) \right|_{s=s^*(\gb)} < 0$  for all $\rho > 0$,
there exists $\delta_2 > 0$ such that
\[|s^*(\gb) - s^*(\gb_s)| < \rho /2 \mbox{ for all } |\gb - \gb_s| < \delta_2\,.\]
Combined with \eqref{eq-d-beta-M-s}, it follows that
\begin{equation}\label{eq-d-beta-M-s-2}
D_{\gb_s}(s) < -2\delta_1, \mbox{ for all } |s - s^*(\gb)| \geq
\rho \mbox{ and } |\gb - \gb_s| < \delta_2\,.
\end{equation}
Now that $D_\gb(s)$ can be viewed as a continuous function of
$(\gb, s)$ and by compactness, there exists $\delta_3 > 0$ such
that $|D_\gb(s) - D_{\gb_s}(s)| \leq \delta_1$ for all $s\in
[1/q, 1]$ and $|\gb - \gb_s| \leq \delta_3$. Combined with
\eqref{eq-d-beta-M-s-2}, it completes the proof by taking $\delta =
\delta_1 \wedge \delta_2 \wedge \delta_3$.
\end{proof}

\begin{proof}[Proof of Lemma~\ref{lem:QuadraticDriftPassTime}]
We do not lose anything by assuming that
\begin{equation}
\label{eqn:QuadDriftRequirementRest}
\bbE[Z_{t+1}-Z_t | \cF_t] = f(Z_t)
\quad \text{;} \qquad
f(z) = \inv{n} (\gz(n) + a z^2 + b z^3 + c z^4)\one_{[-\rho_0, +\rho_0]}(z) \\
\end{equation}
for some $a >0$, $b$, $c$ with $\rho_0 = 2\rho$ and that once $Z_t$ exits
$[-\rho_0, +\rho_0]$ it is stopped.
Indeed, having $f$ vanish outside of $[-\rho_0, +\rho_0]$, does not change the asymptotics of the passing time. Clearly, this is the case for $z > +\rho_0$. For $z < -\rho_0$, it follows from
\begin{equation}
\label{eqn:NoGoingBack1}
	\bbP_{-\rho} (\tau_{-\rho_0}^- < t^{\gz, a}_{\gga}(n)) = o(1),
\end{equation}
for all $\gga$, which is a consequence of Lemma~\ref{lem:ImprovedAzuma} part \eqref{item:ImprovedAzuma3} since the drift of $Z_t$ is at least $\tfrac{c}{n}$ on $[-\rho_0, -\rho]$ for some positive $c$ uniformly in $n$ (if $\rho$ is small enough).

As for replacing the error term in \eqref{eqn:QuadDriftRequirement} by $c Z_t^4$, as the proof below shows, the functions $U^*$, $L^*$ in the lemma
restricted to condition \eqref{eqn:QuadDriftRequirementRest} can be chosen
to be continuous in $a$ in a small interval $[a_0-\gep, a_0+\gep]$
and the limits \eqref{eqn:QuadDriftPassTimeUpperBound},
\eqref{eqn:QuadDriftPassTimeLowerBound} hold uniformly in $a$
in this interval. This together with Remark \ref{rem:ImprovedQuadraticDriftLem}  implies the existence of $U^*$, $L^*$ under which
\eqref{eqn:QuadDriftPassTimeUpperBound},
\eqref{eqn:QuadDriftPassTimeLowerBound}
hold in the general case \eqref{eqn:QuadDriftRequirement} with
$t_{\gga}^{\gz, a+O(\gz(n))}(n)$. Now, it is not difficult to see that the latter is bounded above and below by $t_{\gga \pm C}^{\gz, a}(n)$ for some $C>0$ and hence
\eqref{eqn:QuadDriftPassTimeUpperBound},
\eqref{eqn:QuadDriftPassTimeLowerBound}
hold with $t_{\gga}^{\gz, a}(n)$. Similar considerations apply for \eqref{eqn:FastPassCond}.

Set
\begin{equation*}
\Psi(z) = \int_0^z \frac{1}{f(x)} dx
\quad \text{and} \qquad
Y_t = \Psi(Z_t) - \Psi(Z_0) - t.
\end{equation*}
The motivation behind the above definitions comes from a continuous time deterministic analog of \eqref{eqn:QuadDriftRequirementRest} in the form of an ODE
\begin{equation}
\label{eqn:DeterministicAnalog}
	\dot{z}(t) = f(z(t))
\end{equation}
for which $z(t) = \Psi^{-1}(t-t_0)$ is a solution (roughly speaking
$Y_t$ measures how far behind or ahead ``in schedule'' $Z_t$ is, judging from its position).

Start with the [CR] case and set
\[
t^{\gz, a}(n) = t^{\gz,a}_0(n) =
\frac{\pi}{\sqrt{a}} \frac{n}{\sqrt{\gz(n)}} \quad; \quad \quad
w^{\gz}(n) = \frac{n^{1/2}}{\gz^{5/4}(n)} \maxwith n.
\]
In the deterministic setting
the time it takes for $z(t)$ to pass from $z(0) = -\rho$ to $\rho$ is
\[
\Psi(\rho) - \Psi(-\rho)
 	= 	\int_{-\rho}^{\rho} \frac{n\, dx}{\gz(n) + a x^2 + b x^3 + c x^4 }
	=	t^{\gz, a}(n) + O(n)
\]
if $\rho$ is small enough. This will be shown in Proposition~\ref{prop:DoobDecompBounds} below. Thus, bounding the passage time $\tau_{\rho}^+$ around $t^{\gz, a}(n)$ can be achieved by
bounding $|Y_{t^{\gz, a}(n)}|$.

Accordingly, let $Y_t = M_t + A_t$ be the Doob-decomposition of $Y_t$, with $M_t$ a zero-mean martingale and $A_t$ the predictable process. The next proposition will allow us to bound $Y_t$. The proof of this proposition will be deferred to the end of this section.
\begin{prop}
\label{prop:DoobDecompBounds}
If $\rho$ is small enough and $\gz \in \text{\rm [CR]}$ then
\begin{enumerate}
\item
	\label{item:DoobDecompBoundPsi}
	$\Psi(\rho) - \Psi(-\rho) =	t^{\gz, a}(n) + O(n)$.
\item
	\label{item:DoobDecompBoundM}
	$\bbE M^2_{t_{\gamma}^{\gz, a}(n)} = O( w^{\gz}(n)^2)$.
\item
	\label{item:DoobDecompBoundA}
	$A_{t_{\gga}^{\gz, a}(n)} = o(w^{\gz}(n))$ with probability $1$.
\end{enumerate}
\end{prop}

Now using the monotonicity of $\Psi$ in $[-\rho_0, +\rho_0]$ we have
\begin{eqnarray*}
\bbP_{-\rho} ( \tau_{\rho}^+ > t_{\gamma}^{\gz, a}(n) )
	& \leq 	& \bbP_{-\rho} \lb Z_{t_{\gamma}^{\gz, a}(n)} < \rho \rb \\
	& \leq 	& \bbP_{-\rho} \lb Y_{t_{\gamma}^{\gz, a}(n)} < t^{\gz, a}(n) - t_{\gamma}^{\gz, a}(n)
					+ O(n) \rb  \\
	& \leq 	& \bbP_{-\rho} \lb M_{t_{\gamma}^{\gz, a}(n)} < -\gamma w^{\gz}(n) + O(n) +
					o \left( w^{\gz}(n) \right) \rb \\
	& \leq	& \lb \frac{O(w^{\gz}(n))}{(\gga + O(1)) w^{\gz}(n)} \rb^2 \,,
\end{eqnarray*}
where the last inequality is a second moment bound. This shows
\eqref{eqn:QuadDriftPassTimeUpperBound}. For the lower bound, if $-\gga$ is large enough, we may write
\begin{eqnarray*}
\bbP_{-\rho} \left( \tau_{\rho}^+ < t_{\gamma}^{\gz, a}(n) \right)
	& = 		& \bbP_{-\rho} \left( \exists t < t_{\gamma}^{\gz, a}(n) \;\; : \;\;
					Z_{t}  \geq \rho \right)  \\
	& \leq 		& \bbP_{-\rho} \left( \exists t < t_{\gamma}^{\gz, a}(n) \;\; : \;\;
					Y_{t} \geq t^{\gz, a}(n) - t^{\gz, a}_{\gga}(n) + O(n) \right) \\
	& \leq 	& \bbP_{-\rho} \left( \exists t < t_{\gamma}^{\gz, a}(n) \;\; : \;\;
					M_{t} \geq -\gamma w^{\gz}(n) + O(n) + o \left( w^{\gz}(n) \right) \rb  \\
	& \leq	& \frac{O(w^{\gz}(n))}{(-\gga + O(1)) w^{\gz}(n)}
\end{eqnarray*}
where the last inequality follows from Doob's inequality.
This shows \eqref{eqn:QuadDriftPassTimeLowerBound}

\medskip
Next, we address the [NCR] case. We can no longer use the means analysis
\eqref{eqn:DeterministicAnalog} throughout the entire passage interval $[-\rho, \rho]$ as $Z_t$ is not concentrated around its mean near $0$. Accordingly, we analyze the passage time in each of the following segments
separately:
\[
[-\rho, -r n^{-1/3}]	\quad ; \qquad [-r n^{-1/3}, +r n^{-1/3}]
	\quad ; \qquad [+r n^{-1/3}, +\rho]\,,
\]
for some $r > 0$ to be chosen later.

We start with the upper bound. For the sequel, let $w = r n^{-1/3}$. The upper bound will follow if we show the following:
\begin{enumerate}
\item
	{\it Segment $[-\rho, -w]$.}
	For any $\gga$,
	\begin{equation}
	\label{eqn:PassTimeSeg1}
    	\lim_{r \to \infty} \limsup_{n \to \infty} \bbP_{-\rho} (\tau^+_{-w} > t^{\gz, a}_{\gga}(n)) = 0
	\end{equation}
\item
	\label{item:PassageTimeSeg2}
	{\it Segment $[+w, +\rho]$.}
	For any $\gga$,
	\begin{equation}
	\label{eqn:PassTimeSeg2}
    	\lim_{r \to \infty} \limsup_{n \to \infty}
			\bbP_{w} (\tau^+_{\rho} > t^{\gz, a}_{\gga}(n)) = 0
	\end{equation}
\item
	\label{item:PassageTimeSeg1}
	{\it Segment $[-w, +w]$.}
    For any $r > 0$, there exists $u:\bbR \to [0,1)$ such that
	\begin{equation}
	\label{eqn:PassTimeSeg3}
    	\limsup_{n \to \infty} \bbP_{-w} \lb \tau^+_{w} > t_{\gga}^{\gz, a}(n) \rb \leq u(\gga) <1 \,,
	\end{equation}
	for all $\gga$. Furthermore, $u(\gga) \to 0$ as $\gga \to \infty$.
\end{enumerate}
All are hitting times for $Z$. Indeed, by first choosing large enough $r$ and then choosing large enough $\gga$  both \eqref{eqn:QuadDriftPassTimeUpperBound} and \eqref{eqn:FastPassCond} will follow by multiplication. We proceed to prove each of the above statements.

\medskip
\noindent {\it Segments $[-\rho, -w]$, $[+w, +\rho]$.}
Here we can use the means analysis as in the [CR] case.
As before, we do not change the asymptotics of the passage time through these intervals, if we assume that $f(z)$ in \eqref{eqn:QuadDriftRequirementRest} satisfies:
\[
f(z) = \inv{n} (\gz(n) + a z^2 + b z^3 + c z^4)\one_{[-\rho_0, -w_0] \cup
[+w_0, +\rho_0]}(z)
\]
where $a$, $b$, $c$, $\rho_0$ are as before, $w_0 = w/2$ and once $Z_t$ exits
$[-\rho_0, -w_0] \cup [+w_0, +\rho_0]$ it is stopped. Indeed this follows from the same reasoning and in addition since for any $\gga$
\[
\lim_{r \to \infty} \bbP_{w} (\tau^-_{w_0} < t^{\gz, a}_{\gga}(n)) = 0\,,
\]
uniformly in $n$ (large enough) as it follows from Lemma~\ref{lem:ImprovedAzuma} part \eqref{item:ImprovedAzuma2} since the drift of $Z_t$ is non-negative on $[w_0, w]$.

We use the same definitions for $Y_t$, $M_t$ and $A_t$ as above. In place of Proposition~\ref{prop:DoobDecompBounds} we have
\begin{prop}
\label{prop:DoobDecompBoundsFC}
Assume $\gz \in \text{\rm [NCR]}$. There exists $k(r)$ satisfying $k(r) \to 0$ as $r \to \infty$ such that for any $\rho$ small enough, $r$ large enough, $n$ large enough and all $t$:
\begin{enumerate}
\item
	\label{item:Time_Bound}
    $\Psi(-w) - \Psi(-\rho),\, \Psi(\rho) - \Psi(w) \leq k(r) n^{4/3}$.
\item
	\label{item:DDFC_VBound}
		$\bbE M_{t}^2 \leq k(r) t n^{4/3}$.
\item
	\label{item:DDFC_ABound}
	$|A_{t}| \leq k(r) t $ with probability $1$.
\end{enumerate}
\end{prop}
The proof is again deferred. Now, as before
\begin{eqnarray}
\nonumber
\bbP_{-\rho} \left( \tau^+_{-w} > t^{\gz, a}_{\gga}(n) \right)
    & \leq     & \bbP_{-\rho} \left( Z_{t^{\gz, a}_{\gga}(n)} < -w  \right) \\
\nonumber
    & \leq  & \bbP_{-\rho} \left( Y_{t^{\gz, a}_{\gga}(n)} <
			\Psi(-w) - \Psi(-\rho) - t^{\gz, a}_{\gga}(n) \right)  \\
\nonumber
    & \leq  & \bbP_{-\rho} \left( M_{t^{\gz, a}_{\gga}(n)} <
			k(r) n^{4/3} + k(r) e^\gga n^{4/3} - e^\gga n^{4/3} \right) \\
\label{eqn:ProofOfPassTimeSeg1}
    & \leq  & \frac{k(r) e^\gga}{\lb (1-k(r))e^\gga - k(r) \rb^2}
\end{eqnarray}
where the last inequality is Chebyshev. This goes to zero as $r \to \infty$ for any $\gga$. This shows \eqref{eqn:PassTimeSeg1}. Similarly,
\begin{eqnarray*}
\bbP_{w} \left( \tau^+_{\rho} > t^{\gz, a}_{\gga}(n) \right)
    & \leq     & \bbP_{w} \left( Z_{t^{\gz, a}_{\gga}(n)} < \rho \right) \\
    & \leq  & \frac{k(r) e^\gga}{\lb (1-k(r))e^\gga - k(r) \rb^2}
\end{eqnarray*}
and this shows \eqref{eqn:PassTimeSeg2}.

\medskip
\noindent
{\it Segment $[-w, w]$}.
Here we still assume \eqref{eqn:QuadDriftRequirementRest}, but instead of absorbing $Z_t$ at the boundaries, we shall now suppose that $Z_t$ evolves like a symmetric random walk with $\pm n^{-1}$ steps, once it exits $[-\rho_0, +\rho_0]$.

We first show that $u$ can be chosen to vanish at infinity.
Consider the process $U_t = (U_0 - (Z_t - Z_0+ \gd n^{-5/3}t)$, for $\gd>0$ with
$U_0$ to be chosen later and set $N = \inf \{t :\: U_t \leq 0\}$. Then, by
the definition of the [NCR] regime for $n$ large enough $U_{t \minwith N}$ is a non-negative supermartingale satisfying the requirements of Lemma~\ref{lem:DiffusionHittingTime} and hence
\begin{eqnarray*}
\bbP_{-w}(\tau^+_{w} > t_{\gga}^{\gz, a}(n))
    & =     & \bbP_{-w}(Z_t < w  \ \ \forall t \leq t_{\gga}^{\gz, a}(n)) \\
    & \leq  & \bbP_{-w}(U_t > U_0 + Z_0 - (w + \gd e^\gga n^{-1/3}) \ \ \forall t \leq t_{\gga}^{\gz, a}(n)) \\
    & =     & \bbP_{-w}(N > t_{\gga}^{\gz, a}(n)) \\
    & \leq  & \frac{4(w + \gd e^\gga n^{-1/3} + w)}{\sqrt{d} n^{-1} e^{\gga/2} n^{2/3}}
        \leq C (r+1)e^{-\gga/2}
\end{eqnarray*}
where we choose $U_0 = w + \gd e^\gga n^{-1/3} - Z_0$ and $\gd = e^{-\gga}$. The last expression can be made arbitrarily small by taking $\gga$ large enough, uniformly in $n$ if it is sufficiently large.

To show that $u$ can satisfy $u(\gga) < 1$ for all $\gga$, we have to show that $Z$ can cross from $-w$ to $w$ in   $t^{\gz, a}_{\gga}(n)$-time for arbitrarily small $\gga$. If $\rho$ is small and $n$ is large, then $X_t = n(Z_t - (-w))$ satisfies the conditions in Lemma~\ref{lem:HittingTimeLowerBound} with $\gd=\gz(n)^-$ and $a=d$. Therefore
\[
\begin{split}
\bbP_{-w}(\tau^+_{w} > t_{\gga}^{\gz, a}(n))
	& \ \geq\ \bbP_{-w} (\exists t \leq t_{\gga}^{\gz, a}(n) : \: X_t \geq 2r n^{2/3})	\\
	& \ \geq\  C_1 \exp \{-C_2 (2r e^{-\gga/2} + e^{\gga/2} \gz(n)^- n^{2/3})^2 \} + O(n^{-2/3})
\end{split}
\]
which is positive for all $\gga$, once $n$ is large enough. This proves
\eqref{eqn:PassTimeSeg3} and concludes the proof of the upper bound.

To show the lower bound in the [NCR] case,
set $V_t = Z_t - \gd n^{-5/3}t+w$ and choose $\gd, r > 0$ such that $V_{t \minwith \tau^+_{w}}$ has non-positive drift whenever $V_{t \minwith \tau^+_{w}} \geq 0$. Then,
\begin{eqnarray*}
\bbP_{-\rho} (\tau_{\rho}^+ < t_{\gga}^{\gz, a}(n))
    & \leq     & \bbP_{-\rho} (\exists t < t_{\gga}^{\gz, a}(n) \ \ :\:
                    Z_t \geq w) \\
    & \leq  & \bbP_{-\rho} (\exists t < t_{\gga}^{\gz, a}(n) \ \ :\:
                    V_{t \minwith \tau^+_{w}} \geq 2w - \gd e^\gga n^{-1/3}) \\
    & =  & \bbP_{-\rho} (\exists t < t_{\gga}^{\gz, a}(n) \ \ :\:
                    V_{t \minwith \tau^+_{w}} \geq (2r-\gd e^\gga)n^{-1/3})
\end{eqnarray*}
and part~\eqref{item:ImprovedAzuma2} of Lemma~\ref{lem:ImprovedAzuma} shows that the last expression goes to $0$ as $\gga \to -\infty$ uniformly in $n$ (large enough). This proves~\eqref{eqn:QuadDriftPassTimeLowerBound} and completes the [NCR] case.
\end{proof}

It remains to prove Propositions~\ref{prop:SharpnessOfDBound}--\ref{prop:DoobDecompBoundsFC}.
\begin{proof}[Proof of Proposition~\ref{prop:SharpnessOfDBound}]
Let
$\tau^* = \min\left\{ \tau_y^{-} , t(n),\min_{i\geq 2} \min\{t\geq 0 : S_t^i \geq \frac1q\} \right \}$.
Fix some $2\leq i<j\leq q$ and set
\[
Y_t = S_t^i - S_t^j.
\]
Let $U_{t-1} = Y_t-Y_{t-1} -  \bbE_{\gs_0}  \left[Y_t-Y_{t-1}\mid \mathcal{F}_{t-1}\right]$ and then since $|Y_t-Y_{t-1}| \leq \frac2n$ we have that
$|U_i | \leq \frac4n$.   Define the process $Z_t$ by $Z_0=0$ and
\[
Z_t -Z_{t-1} :=  \hbox{sign}(Z_{t-1})\hbox{sign}(Y_{t-1}) U_{t-1}
\]
where
\[
\hbox{sign}(x)= \begin{cases} 1 & x \geq 0,\\ -1 &x <0. \end{cases}
\]
With this definition $Z_t$ is clearly a martingale and since $|Z_t-Z_{t-1}|\leq \frac4n$, then $
\bbE_{\gs_0} Z_t^2 \leq \frac{16t}{n^2}$ and so by Doob's maximal inequality,
\begin{equation}\label{e:maximalBound}
 \bbE_{\gs_0}  \left[ \max_{0\leq t \leq t(n)} |Z_t| \right]^2 \leq 2 \bbE_{\gs_0} Z_{t(n)}^2 \leq \frac{32t(n)}{n^2}.
\end{equation}
Now when $t<\tau^*$ we have that $S^i_t,S_t^j<\frac1q$ and so $$\left|e^{2\gb(S_t^i-\frac1q)}-e^{2\gb(S_t^j-\frac1q)}\right| \leq 2\gb \left| S_t^i
- S_t^j\right|.$$  By Jensen's inequality $\sum_{k=1}^q e^{2\gb(S_t^i-\frac1q)} \geq q$ so
\begin{align}
\nonumber
 &\bbE_{\gs_0}  \left[\hbox{sign}(Y_{t-1})(Y_t-Y_{t-1}) \mid \mathcal{F}_{t-1}\right]\\
\nonumber
 &\qquad= \bbE_{\gs_0}  \left[\hbox{sign}(S_{t-1}^i - S_{t-1}^j)\frac1n\left(\frac{e^{2\gb(S_t^i-\frac1q)}-e^{2\gb(S_t^j-\frac1q)}}{\sum_{k=1}^q
 e^{2\gb(S_t^i-\frac1q)}}  - (S_{t-1}^i - S_{t-1}^j)\right) \right]\\
\label{e:expectedChangeBound}
&\qquad\leq -\frac{q-2\gb}{qn} \left| S_t^i - S_t^j\right| \leq 0.
\end{align}
Now when $|Y_{t-1}| \geq \frac2n$ we have that $|Y_t| - |Y_{t-1}| = \hbox{sign}(Y_{t-1})(Y_t-Y_{t-1})$ and we always have $$|Z_t| - |Z_{t-1}| \geq
\hbox{sign}(Z_{t-1})(Z_t-Z_{t-1})= \hbox{sign}(Y_{t-1}) U_{t-1} $$  with equality when $\hbox{sign}(Z_t) = \hbox{sign}(Z_{t-1})$. Hence it follows that
when $|Y_{t-1}| \geq \frac2n$ and $t \leq \tau^*$,
\begin{align*}
|Y_t| - |Y_{t-1}| &=  \hbox{sign}(Y_{t-1})(Y_t-Y_{t-1})\\
&\leq \hbox{sign}(Y_{t-1}) U_{t-1} \\
&\leq |Z_t| - |Z_{t-1}|
\end{align*}
where the first inequality follows from equation \eqref{e:expectedChangeBound}.  It follows by induction that $|Z_t| \geq |Y_t| - \frac3n$ for all
$t\leq \tau^*$.  In particular we have by equation \eqref{e:maximalBound}  that
\begin{equation}\label{e:maxSiDiff}
 \bbE_{\gs_0} \lsb \max_{2\leq i < j \leq q} \max_{0\leq t \leq \tau^*} \left|S_{t}^i - S_{t}^j \right|  \rsb^2 =
 O\left(\frac{t(n)}{n^2}\right)=o(1).
\end{equation}
By Markov's inequality with probability tending to 1 we have that $\left|S_{\tau^*}^i -
S_{\tau^*}^j \right| = o(1)$ for every pair $2\leq i,j\leq q$.  Now by construction
 $S_{\tau^*}^1 \geq y-\frac1n$ so with high probability we have that $S_{\tau^*}^i
\leq \frac1q-\frac{y-\frac1q}{q-1} + o(1) < \frac1q$ which implies that with high probability
$ \tau^* = \min\{t(n),\tau_y^{-}\}$.

Now, by Taylor series expansions,
\begin{align*}
0\leq \left(\sum_{i=2}^q e^{2\gb S_t^i}\right) - (q-1)e^{2\gb \frac{1-S_t^1}{q-1}} & \leq \sum_{i=2}^q (2\gb S_t^i - 2\gb \frac{1-S_t^1}{q-1} ) +
\sum_{i=2}^q O\left ( ( 2\gb S_t^i- 2\gb \frac{1-S_t^1}{q-1})^2\right)\\
&\leq   O\left ( \left(\max_{2\leq i<j\leq q} \left| S_t^i - S_t^j   \right|\right)^2\right),
\end{align*}
where the first inequality is by Jensen, and we have used the fact that $\sum_{i=2}^q  S_t^i = 1 - S_t^1$.  It therefore follows that with high
probability for all $0\leq t\leq \max\{t(n),\tau_y^{-}\} $ that
\begin{align*}
n \bbE_{\gs_0} \left[ S_{t}^1- S_{t-1}^1\mid\mathcal{F}_{t-1}\right] & = \frac{e^{2\gb S_{t-1}^1}}{\sum_{i=1^q} e^{2\gb S_{t-1}^i}} - S_{t-1}^1\\
&= \frac{e^{2\gb S_{t-1}^1}}{e^{2\gb S_{t-1}^1} + (q-1)e^{2\gb \frac{1-S_{t-1}^1}{q-1}}} - S_{t-1}^1 - O\left ( \left(\max_{2\leq i<j\leq q} \left|
S_t^i - S_t^j   \right|\right)^2\right) \\
&=  D_\gb(S_{t-1}^1) -  O\left ( \left(\max_{2\leq i<j\leq q} \left| S_t^i - S_t^j   \right|\right)^2\right)
\end{align*}
and hence that
\begin{align*}
& \max_{0\leq t \leq \tau^*} D_\gb(S_{t-1}^1) - n \bbE_{\gs_0} \left[ S_{t}^1- S_{t-1}^1\mid\mathcal{F}_{t-1}\right] \\
 &\qquad\leq O\left ( \left(\max_{2\leq i < j \leq q} \max_{0\leq t \leq \tau^*} \left| S_t^i - S_t^j   \right|\right)^2\right)
\end{align*}
which combined with equation \eqref{e:maxSiDiff} and Markov's inequality completes the result.
\end{proof}

\begin{proof}[Proof of Proposition~\ref{prop:DoobDecompBounds}]
Starting with part \eqref{item:DoobDecompBoundPsi},
\begin{eqnarray*}
\lefteqn{
	\Psi(\rho) - \Psi(-\rho)}	\\
	& =	&
 		\int_{-\rho}^{\rho} \frac{n\, dx}{\gz(n) + a x^2} +
		\int_{-\rho}^{\rho} \lb \frac{n}{\gz(n) + a x^2} -
			\frac{n}{\gz(n) + a x^2 + b x^3 + c x^4} \rb dx \\
	& = 	& \frac{2n}{\sqrt{a \gz(n)}} \tan^{-1}
				\left( \frac{\rho \sqrt{a}}{\sqrt{\gz(n)}}  \right)
				+ n \int_{-\rho}^{\rho} \left( - \frac{b x^3 + c x^4}
					{(\gz(n)+a x^2)^2} + O\left( \frac{ \left( b x^3 + c x^4
					\right)^2 }{(\gz(n)+a x^2)^3} \right) \right) dx \\
	& =	&
		\frac{\pi}{\sqrt{a}} \frac{n}{\sqrt{\gz(n)}} + O(n) +
			O\left( n \int_{-\rho}^{\rho} \left( - \frac{c x^4}{a^2 x^4} + \frac{
			\left( b x^3 + c x^4 \right)^2 }{a^3 x^6} \right) dx \right) \\
	& = 	&
		t^{\gz, a}(n) + O(n) + O \lb n \int_{-\rho}^{\rho} dx \rb = t^{\gz, a}(n) + O(n) .
\end{eqnarray*}

To prove part \eqref{item:DoobDecompBoundM}, we use the law of total variance:
\begin{eqnarray*}
\Var M_{t} & = & \Var \; \bbE[M_{t}|\cF_{t-1}] \; + \; \bbE \; \Var[M_{t}|\cF_{t-1}] \nonumber \\
& = & \Var M_{t-1} \; + \; \bbE \; \Var[Y_{t} | \cF_{t-1}] \nonumber \\
& \leq & \Var M_{t-1} \; + \; \max_{|z| \leq 2 \rho_0} \left| \Psi^{\prime}(z) \right|^2 \bbE \; \Var[Z_{t} - Z_{t-1} | \cF_{t-1}] \nonumber \\
& \leq & \Var M_{t-1} \; + \; \frac{n^2}{\gz^2(n)} \frac{1}{n^2}.
\end{eqnarray*}
Hence by induction
\[
\bbE M^2_{t_{\gamma}^{\gz, a}(n)} = \Var M_{t_{\gamma}^{\gz, a}(n)} \leq \frac{ t_{\gamma}^{\gz, a}(n)}{\gz^2(n)} =
	O ( w^{\gz}(n)^2)	.
\]

As for part \eqref{item:DoobDecompBoundA},
\begin{eqnarray*}
A_{t+1} - A_{t}
	& = 		& \bbE[Y_{t+1}-Y_{t}|\cF_{t}] \\
	& = 		& \bbE \left[ \lpr \Psi(Z_{t+1}) - \Psi(Z_t) \rabs
					\cF_t \right] - 1 \\
	& = 	& \bbE \left[ \lpr \Psi'(Z_t) (Z_{t+1}-Z_t) \; + \;
					O(\max_{|z| \leq 2 \rho} \left| \Psi^{\prime\prime}(z)
						\right| (Z_{t+1}-Z_t)^2) \; \rabs \; \cF_t \right] - 1 \\
	& =	& \Psi^{\prime}(Z_t) \bbE[Z_{t+1}-Z_t|\cF_t] \; + \;
					O(\max_{|z| \leq 2 \rho_0} \left| \Psi^{\prime\prime}(z) 		
					\right| n^{-2}) - 1 \\
	& = 	& O(\max_{|z| \leq 2 \rho} \left| \Psi^{\prime\prime}(z) \right|
					n^{-2})
			= O \left( \frac{1}{n \gz^{3/2}(n)} \right) ,
\end{eqnarray*}
where the last inequality follows from
\begin{eqnarray*}
\left| \frac{d^2}{dz^2} \Psi(z) \right|
	& = 		& \frac{n | 2 a z + 3 b z^2 + 4 c z^3 |}
						{(\gz(n) + a z^2 + b z^3 + c z^4)^2} \\
	& \leq 	& \frac{n C_1 |z|}{(\gz(n) + C_2 z^2)^2} =
		O \lb \frac{n}{\gz^{3/2}(n)} \rb .
\end{eqnarray*}
if $\rho$ is small enough. Then again by induction, we conclude that
\begin{equation*}
A_{t_{\gamma}^{\gz, a}(n)} = O \lb \frac{t_{\gamma}^{\gz, a}(n)}{n \gz^{3/2}(n)} \rb
	= o \left( w^{\gz}(n) \right).
\qedhere
\end{equation*}
\end{proof}

\begin{proof}[Proof of Proposition \ref{prop:DoobDecompBoundsFC}]
If $r$ is large enough and $\rho$ is small, we have $f(z) \geq n^{-1} (a/2)z^2$
for all $w_0 < |z| < \rho_0$, where as before $w_0 = w / 2$ and $\rho_0 = 2 \rho$. This immediately gives part \eqref{item:Time_Bound} with $k(r) = C r^{-1}$.

The proof for parts \eqref{item:DDFC_VBound},\eqref{item:DDFC_ABound} are similar to the ones in Proposition~\ref{prop:DoobDecompBounds}. This time the bounds on the derivatives become
\[
\max_{w_0/2 < |z| < 2\rho_0} \left| \Psi^{\prime}(z) \right|^2 \leq C r^{-4} n^{10/3}
	\quad ; \quad
\max_{w_0/2 < |z| < 2\rho_0} \left| \Psi^{\prime\prime}(z) \right| \leq
C r^{-3} n^2  .
\]
Proceeding by induction as before, we obtain \eqref{item:DDFC_ABound}, \eqref{item:DDFC_VBound} with $k(r)=C r^{-4}$ and $k(r) = C r^{-3}$ respectively.
\end{proof}

\section{Essential Mixing}
\label{sec:EssentialMixing}
\begin{proof}[Proof of Theorem~\ref{thm:EssentialMixing}]
As the reader can verify, most statements in Sections \ref{sec:DriftAnalysis} and \ref{sec:SubcriticalRegime} hold when $\gb < \gb_c(q)$ and even $\gb < q/2$ (the restrictions on $\gb$ are indicated before each statement there). The only time $\gb < \gb_s(q) < \gb_c(q)$ is required is in step \eqref{item:FC_Burn} of the overall coupling, where the condition ensures that the drift of each single coordinate $S^i_t$ is negative in all $(1/q, 1]$, which, in turn, implies that for any initial configuration, after $t=O(n)$ time, $\gs_t \in \gS_n^{\rho}$, which is a necessary starting point for the couplings that follow.

Now, if $\gb \geq \gb_s(q)$, but still $\gb <\gb_c(q)$, we may replace this step, with the requirement that $\gs_0$ is initially chosen from $\wt{\gS}_n = \gS_n^{\rho}$.
The analysis of the overall coupling will remain the same, with the coalescence time being even smaller (but just by a linear term, which can be absorbed in the cutoff-window term). Thus, the restricted mixing time $t^{\wt{\gS_n}}_{\mix(\gep)}(n)$ will be upper bounded as before.
In addition, the lower bound in Subsection \ref{sub:ProofOfLowerBound} will also hold for $t^{\wt{\gS_n}}_{\mix(\gep)}(n)$, since as initial configuration, we may take any $\gs_0 \in \partial_{P_n} \gS_n(\rho)$ for $\rho > 0$.

It remains to show that $\gS_n \setminus \wt{\gS}_n$ has an exponentially decreasing probability under $\mu_n$. This follows immediately from the large deviations analysis in Subsection~\ref{sub:CWLDP}. If $\gb < \gb_c(q)$, the rate function $I_{\gb, q}$ is strictly positive away from $\vq$ and in particular there exist $C_1 > 0$, $C_2 > 0$,
such that
\[
\mu_n \lb \gS_n \setminus \wt{\gS}_n \rb
	= \pi_n \lb \cS_n \setminus \cS^{\rho_1}_n \rb
	\leq C_1 e^{-C_2 n}.
\]
This concludes the proof of the theorem.
\end{proof}
\newpage
\section*{Acknowledgments}
This work was initiated while P.C., O.L.~ and A.S.\ were interns at the Theory Group of Microsoft Research, and
they thank the Theory Group for its hospitality.

\end{document}